\documentclass[11pt, psamsfonts, reqno]{amsart}

\usepackage[T1]{fontenc}
\usepackage[latin1]{inputenc}
\usepackage[frenchb, english]{babel}
\usepackage{amsthm}
\usepackage{amscd}
\usepackage{amssymb,amsmath,bbm}
\usepackage{tikz-cd}
\usepackage[all]{xy}
\usepackage{mathrsfs}
\usepackage[left=0.0cm,right=0.0cm,top=1.7cm,bottom=1.7cm]{geometry}
\usepackage{hyperref}
\usepackage{verbatim}

\theoremstyle{plain}
\newtheorem{theorem}{Theorem}[section]

\newtheorem{lemma}[theorem]{Lemma}
\newtheorem{proposition}[theorem]{Proposition}
\newtheorem{prop}[theorem]{Proposition}
\newtheorem{corollary}[theorem]{Corollary}
\newtheorem{cor}[theorem]{Corollary}
\newtheorem*{theorem*}{Theorem}

\theoremstyle{definition}
\newtheorem{definition}[theorem]{Definition}
\newtheorem{example}[theorem]{Example}

\newtheorem{conj}[theorem]{Conjecture}

\newtheorem{convention}[theorem]{Convention}

\theoremstyle{remark}
\newtheorem{remark}[theorem]{Remark}

\usepackage{tikz}
\usetikzlibrary{calc}

\def\Q{{\bf Q}}

\def\Z{{\bf Z}}

\def\N{{\bf N}}
\def\R{{\bf R}}

\def\O{{\mathcal{O}}}

\def\zp{{\Z_p}}

\def\qp{{\Q_p}}

\def\Hom{\mathrm{Hom}}
\def\Ext{\mathrm{Ext}}

\def\iRHom{R\underline{\mathrm{Hom}}}
\def\Dist{\mathcal{D}}

\DeclareMathOperator{\hocolim}{hocolim}

\def\epsilon{\varepsilon}

\def\GL{\mathbf{GL}}

\def\Cond{\mathrm{Cond}}
\def\cond{\mathrm{cond}}

\def\solid{\mathrm{solid}}
\def\iHom{\underline{\mathrm{Hom}}}

\def\sol{{\blacksquare}}

\DeclareMathOperator{\Mod}{Mod}

\DeclareMathOperator{\Spa}{Spa}
\DeclareMathOperator{\gr}{gr}

\DeclareMathOperator{\Lie}{Lie}
\DeclareMathOperator{\Top}{Top}
\DeclareMathOperator{\Cont}{{Cont}}
\DeclareMathOperator{\Sets}{Sets}
\DeclareMathOperator{\Extdis}{Extdis}
\DeclareMathOperator{\Condab}{CondAb}

\newcommand{\n}[1]{\mathcal{#1}}
\newcommand{\bb}[1]{\mathbb{#1}}
\newcommand{\f}[1]{\mathfrak{#1}}
\newcommand{\bbf}[1]{\mathbf{#1}}

\DeclareMathOperator{\proet}{proet}

\def\SolidK{\Mod_K^{\solid}}
\def\SolidOK{\Mod_{\O_K}^\solid}
\def\SolidKG{\Mod^{\solid}_{K_\sol[G]}}
\def\SolidOKG{\Mod^{\solid}_{\O_{K,\sol}[G]}}

\def\hcoind{\text{$h$-}\mathrm{coind}}
\def\hind{\text{$h$-}\mathrm{ind}}

\title{Solid locally analytic representations of $p$-adic Lie groups}

\makeatletter
\def\@tocline#1#2#3#4#5#6#7{\relax
  \ifnum #1>\c@tocdepth 
  \else
    \par \addpenalty\@secpenalty\addvspace{#2}%
    \begingroup \hyphenpenalty\@M
    \@ifempty{#4}{%
      \@tempdima\csname r@tocindent\number#1\endcsname\relax
    }{%
      \@tempdima#4\relax
    }%
    \parindent\z@ \leftskip#3\relax \advance\leftskip\@tempdima\relax
    \rightskip\@pnumwidth plus4em \parfillskip-\@pnumwidth
    #5\leavevmode\hskip-\@tempdima
      \ifcase #1
       \or\or \hskip 1em \or \hskip 2em \else \hskip 3em \fi%
      #6\nobreak\relax
    \dotfill\hbox to\@pnumwidth{\@tocpagenum{#7}}\par
    \nobreak
    \endgroup
  \fi}
\makeatother
\usepackage{hyperref}

\begin{document}

\begin{abstract}
We develop the theory of locally analytic representations of compact $p$-adic Lie groups from the perspective of the theory of condensed mathematics of Clausen and Scholze. As an application, we generalise Lazard's isomorphisms between continuous, locally analytic and Lie algebra cohomology to solid representations. We also prove a comparison result between the group cohomology of a solid representation and of its analytic vectors.
\end{abstract}

\author{Joaqu\'in Rodrigues Jacinto and Juan Esteban Rodr\'iguez Camargo} 
\maketitle
\selectlanguage{english}

\setcounter{tocdepth}{2}
\tableofcontents

\section{Introduction}

The theory of $p$-adic representations of $p$-adic  groups has a long history and it has played a key role in the field of Number Theory during the last decades, as witnessed, e.g., in the study of the $p$-adic Langlands correspondence \cite{ColmezPhiGamma}, \cite{CDP}. In this article, we intend to reformulate the theory of locally analytic representations of $p$-adic Lie groups as developed in \cite{SchTeitGl2}, \cite{SchTeitUg}, \cite{SchTeitDist}, \cite{SchTeitDuality}, \cite{Emerton}, using the theory of condensed mathematics of  Clausen and Scholze. \\

More precisely, we define and study the notions of analytic and locally analytic representations of $p$-adic Lie groups on solid modules. One of our main new results, which was the departing point of our investigations, is a generalisation of Lazard's comparison theorems \cite[Th\'eor\`emes  V.2.3.10 et V.2.4.10]{Lazard} between continuous, locally analytic and Lie algebra  cohomology of a finite dimensional representation of a compact $p$-adic Lie group over $\qp$ to arbitrary solid locally analytic representations\footnote{In particular, to any complete compactly generated locally convex vector space, e.g. metrizable.}. Generalisations of Lazard's comparison between locally analytic and Lie algebra cohomology have already been considered in \cite{HuberKingsNaumann}, \cite{Lechner}, \cite{Tamme}. We also give, following the lines of \cite{Tamme}, a proof of this result in the solid context. Our second main new result is a comparison between continuous cohomology of solid representation and the continuous cohomology of its locally analytic vectors. This results can be seen as a $p$-adic analogue of  theorems of G. D. Mostow \cite{Mostow} and P. Blanc \cite{Blanc}.  \\

\subsection{Background}

Let $G$ be a topological group, one usual way of studying continuous $G$-representations   is through their cohomology. In a favorable situation, the cohomology groups are $\Ext$-groups in certain abelian categories. For instance, this is the case for the category of continuous representations of $G$ on discrete modules.  On the other hand, since the category of continuous $G$-representations  on topological abelian groups is not abelian, it is not clear that one can define cohomology groups as $\Ext$-groups in this setting. The continuous group cohomology of $G$ with values in a topological representation is usually  defined via a complex of  continuous cochains.  This definition lacks of some conceptual advantages.  For example:   it is not clear whether the cohomology groups carry a natural topology,  short exact sequences do not necessarily induce long exact sequences in cohomology,  and certain basic results such  as  Hochschild-Serre need to be proved by hand  (cf. \cite[\S 2]{SerreGaloisCoh}).   \\



These inconveniences  would be overcome if one is able to find a suitable abelian category where  topological representations can be embedded. The theory of condensed mathematics developed by Clausen and Scholze \cite{ClausenScholzeCondensed2019}, \cite{ClauseScholzeanalyticspaces},  \cite{ClausenScholzeCM} provides a natural approach to making this work. Very vaguely, the condensed objects in a category $\mathscr{C}$ can be defined as sheaves on the pro\'etale site of a point with values in $\mathscr{C}$. It is shown in \cite[Theorem  2.2]{ClausenScholzeCondensed2019} that the category of condensed abelian groups is an abelian category satisfying Grothendieck's axioms. Hence, one is on a good footing for doing homological algebra. \\

Let $p$ be a prime number. In this paper we will be interested in the case where $G$ is a compact $p$-adic Lie group.  The classical theory of $p$-adic representations of $G$ comes in different flavours: there exist  the notions of continuous, analytic and locally analytic representations $V$ of $G$,  according to whether the orbit maps $o_v : G \to V$, $g \mapsto g \cdot v$ (for all $v \in V$)  are  continuous, resp. analytic, resp. locally analytic functions of  $G$. \\

As an example, in order to study locally analytic representations from an algebraic perspective, in \cite{SchTeitDist}, Schneider and Teitelbaum defined the notion of a locally analytic  admissible representation, and they showed  that these objects form an abelian category. Even with this result at hand, there are still various obstacles to define a satisfactory cohomological theory of admissible locally analytic representations. For instance, it not clear how to show that the $\Ext$-groups in this abelian category coincide with the cohomology groups defined via locally analytic cochain complexes;  this is one of the main results of \cite{Kohlhaase}.    \\

 


\subsection{Statement of the main results}

Let us now describe with some more detail what is carried out in this article.

\subsubsection{Solid non-archimedean functional analysis} 


Let $K$ be a finite extension of $\qp$. The field $K$ naturally defines a condensed ring which, moreover, has an analytic ring structure $K_{\sol}$ in the sense of \cite[Definition 7.4]{ClausenScholzeCondensed2019}, usually called the solid ring structure on $K$. We let $\SolidK$ be the category of solid $K$-vector spaces. This category is stable under limits, colimits and extensions, it has a tensor product $\otimes_{K_\sol}$ and an internal Hom denoted by $\iHom_K(-, -)$ \cite[Proposition 7.5]{ClausenScholzeCM}.
Let us point out that all the important spaces in the classical theory of non-archimedean functional analysis \cite{SchneiderNFA} live naturally in $\SolidK$. Indeed,   there is a natural functor
\begin{equation} \label{FunctorLCSolid}
\mathcal{LC}_K \to \SolidK
\end{equation}
from the category of complete locally convex $K$-vector spaces to solid $K$-vector spaces, as any complete locally convex $K$-vector space can be written as a cofiltered limit of Banach spaces. Moreover, it is fully faithful on a very large class of complete locally convex $K$-vector spaces, e.g. all compactly generated ones, e.g. all metrizable ones. The main notions of the theory of condensed non-archimedean functional analysis we use are due to Clausen and Scholze \cite{ClausenScholzeCondensed2019}, \cite{ClausenScholzeCM}, \cite{GuidoDrinfeld}.  All the vector spaces considered in this text are solid $K$-vector spaces, unless otherwise specified. \\


Our first result is an anti-equivalence between two special families of solid $K$-vector spaces. Let us first give some definitions. A Smith space is a $K$-vector space of the form $\iHom_{K}(V, K)$, where $V$ is a Banach space. In classical terms, a Smith space is the dual of a Banach space equipped with the compact-open topology. An $LS$ space is a countable filtered inductive limit of Smith spaces with injective transition maps. We then have the following result.

\begin{theorem} [Theorem \ref{theorem:duality}]
\label{Theodualityintro} 
The functor $V \mapsto V^\vee:=\underline{\Hom}_K(V,K)$ induces an anti-equivalence between Fr\'echet and $LS$ spaces such that $\iHom_K(V,V')=\iHom_K(V'^{\vee},V^{\vee})$.
\end{theorem}

\begin{remark}
The previous theorem restricts in particular to an anti-equivalence between classical nuclear Fr\'echet spaces and $LB$ spaces of compact type (see, e.g., \cite[Theorem 1.3]{SchTeitGl2}). 
\end{remark}

\subsubsection{Representation theory}

Let $G$ be a compact $p$-adic Lie group. A representation of $G$ on a solid $K$-vector space $V$ is a map of condensed sets $G \times V \to V$ satisfying the usual axioms. Define the Iwasawa algebra of $G$ with coefficients in $K$ as $K_\sol[G]$;  explicitly, 
\[
K_\sol[G]=\big(\varprojlim_N\O_K[G/N]\big)[\frac{1}{p}],
\]
where $N$ runs over all the open normal subgroups of $G$. This is the solid algebra defined by the classical Iwasawa algebra endowed with the weak topology. The category of $G$-representations on solid $K$-vector spaces is equivalent to the category $\SolidKG$ of solid $K_{\sol}[G]$-modules. Observe that the category of continuous representations of $G$ on complete locally convex $K$-vector spaces lives naturally in $\SolidKG$ via the functor \eqref{FunctorLCSolid}. \\

Inspired by Emerton's treatment \cite{Emerton}, we define analytic and locally analytic vectors of solid representations of $G$. Roughly speaking, they are defined as those vectors whose induced orbit map is analytic or locally analytic. One advantage of our approach is that definitions make sense at the level of derived categories, so one can speak about derived (locally) analytic vectors of complexes $C \in D(K_\sol[G])$ in the derived category of $K_{\sol}[G]$-modules. The derived functors of the locally analytic vectors (for Lie groups over finite extensions of $\qp$) for admissible representations have been considered in \cite{Schmidt}. \\

More precisely, let $\bb{G}$ be an affinoid  group neighbourhood of $G$, i.e., an analytic affinoid group over $\Spa (\qp,\zp)$ such that  $G=\bb{G}(\qp,\zp)$.  Suppose in addition that $\bb{G}$ is isomorphic to a finite disjoint union of polydiscs.   In practice, the group  $\bb{G}$ will be constructed using some local charts of $G$,  see Remark \ref{remintroex} below for a more detailed description. Let $C(\bb{G},K):= \mathscr{O}(\bb{G})\otimes_{\qp} K$ be the algebra of functions of $\bb{G}$. The affinoid algebra $C(\bb{G},K)$ has a natural analytic ring structure denoted by $C(\bb{G},K)_{\sol}$, see  \cite[Theorem 3.27]{Gregory}.   We denote by  $\n{D}(\bb{G},K)=\mathscr{O}(\bb{G})^\vee \otimes_{\mathbf{Q}_p} K$  the distribution algebra of the affinoid group $\bb{G}$.   We define the derived $\bb{G}$-analytic vectors of an object $C \in D(K_{\blacksquare} [G])$ to be the complex\footnote{The tensor product ${\otimes_{K_\blacksquare}} C(\bb{G},K)_{\blacksquare}$ is the derived base change of modules over analytic rings, see  \cite[Proposition 7.7]{ClausenScholzeCondensed2019}.}
\begin{equation}
\label{eqLAaffine}
C^{R\bb{G}-an}:= R\underline{\Hom}_{K_{\blacksquare}[G]}(K, C\otimes^L_{K_\blacksquare} C(\bb{G},K)_{\blacksquare} ),
\end{equation}
where $K$ is the trivial representation, and  the $G$-action on $C\otimes^L_{K_\blacksquare} C(\bb{G},K)_{\blacksquare}$ is the  diagonal one induced by the action on $C$ and the left regular action on $C(\bb{G},K)$. We endow $C^{R \mathbb{G}-an}$ with the right regular action of $G$. It turns out that there is a natural map $C^{\bb{G}-an}\to C$, and we say that $C$ is derived $\bb{G}$-analytic if this map is a quasi-isomorphism. If $V$ is a Banach $G$-representation, then $V \otimes^L_{K_\sol} C(\bb{G},K )_\sol= V \otimes_{K_\sol} C(\bb{G},K)$ coincides with the projective tensor product of Banach spaces. Thus, our definition of derived $\bb{G}$-analytic vectors is the derived extension of \cite[Definition  3.3.13]{Emerton}.    \\

Now let $\mathring{\bb{G}}$ be a Stein group neighbourhood of $G$,  i.e.  an analytic group over $\Spa(\mathbf{Q}_p,  \mathbf{Z}_p)$  which is written as a strict  increasing union of affinoid group neighbourhoods $\bb{G}^{(h)}$ of $G$.   In practice, $\mathring{\mathbb{G}}$ will be as in Remark \ref{remintroex} below. We also denote by $\n{D}(\mathring{\bb{G}},K):=\mathscr{O}(\mathring{\bb{G}})^{\vee}\otimes_{\qp} K = \varinjlim_{h} \n{D}(\bb{G}^{(h^+)},K)$  the distribution algebra of $\mathring{\bb{G}}$. Then the derived $\mathring{\bb{G}}$-analytic vectors of $C\in D(K_\blacksquare[G])$ are defined as the complex 
\begin{equation}
\label{eqlaStein}
C^{R\mathring{\bb{G}}-an}:= R\varprojlim_{h} C^{R\bb{G}^{(h)}-an}, 
\end{equation}
and we say that $C$ is derived  $\mathring{\bb{G}}$-analytic if the natural map $C^{R\mathring{\bb{G}}-an}\to C$ is a quasi-isomorphism. Again, if $V$ is a Banach $G$-representation, this definition is compatible with the $\mathring{\bb{G}}$-analytic vectors of \cite[Definition 3.4.1]{Emerton}. The main theorem is the following:
\begin{theorem} [Theorem \ref{TheoMain}]
\label{TheoMainintro}
 A complex $C\in D(K_{\blacksquare}[G])$ is  derived  $\mathring{\bb{G}}$-analytic if and only if it is a module over $\n{D}(\mathring{\bb{G}},K)$.  
\end{theorem}

This theorem is a generalisation  of  the integration map constructed by Schneider and Teitelbaum  to solid $G$-modules, cf. \cite[Theorem 2.2]{SchTeitGl2}. It will serve us as a bridge between solid analytic representations and solid modules over the distribution algebras.

\begin{remark} \label{remintroex}
Suppose that $G$ is a uniform pro-$p$-group and $\phi : \Z_p^d \to G$ is an analytic chart given by a basis of $G$ (see, e.g., Example \ref{exampleGL2}).   Using this chart, one can define for any $h \in \Q_{>0}$ affinoid group neighbourhoods $\bb{G}^{(h)}$ of $G$ whose underlying adic spaces are isomorphic to   $\mathbf{Z}_p^d + p^h \mathbb{D}^d_{\mathbf{Q}_p} \subset \bb{D}^d_{\mathbf{Q}_p}$,  where $\bb{D}_{\mathbf{Q}_p}=\Spa(\mathbf{Q}_p\langle T \rangle , \mathbf{Z}_p \langle T \rangle ) $ is the closed unit disc. The $p^h$-analytic functions on $G$ with respect to the chart $\phi$ (or simply the $\bb{G}^{(h)}$-analytic functions)  are defined as the rigid functions of  $\mathbb{G}^{(h)}$.  In this situation, the $\bb{G}^{(h)}$-analytic vectors are those whose orbit map is $p^h$-analytic. Given $h > 0$, the group $\mathbb{G}^{(h^+)} := \bigcup_{h' > h} \mathbb{G}^{(h)}$ is a Stein group neighbourhood of $G$. In the definition of derived analytic vectors of equations \eqref{eqLAaffine} and \eqref{eqlaStein} we will take $\bb{G}=\bb{G}^{(h)}$ and $\mathring{\bb{G}}=\bb{G}^{(h^+)}$.  The locally analytic functions on $G$ are \[ C^{la}(G, K) = \varinjlim_{h \to +\infty} C(\mathbb{G}^{(h)}, K) = \varinjlim_{h \to +\infty} C(\mathbb{G}^{(h^+)}, K). \]
Now let $\mathcal{D}^{la}(G, K) = C^{la}(G, K)^\vee$ be the algebra of locally analytic distributions. We point out that the analogous statement of Theorem \ref{TheoMainintro} does not hold in general for locally analytic representations and $\mathcal{D}^{la}(G, K)$-modules. In the particular case of locally analytic representations on $LB$ spaces of compact type, this is nevertheless true, and was already known by \cite[Theorem 2.2]{SchTeitGl2}.
\end{remark}


\subsubsection{Comparison theorems in cohomology}


We finish this introduction by describing the main applications of Theorems \ref{Theodualityintro} and \ref{TheoMainintro} to the study of the cohomology of continuous representations. \\

For $C \in D(K_\sol[G])$, we define the solid group cohomology of $C$ to be the complex
\[ R\iHom_{K_{\sol}[G]}(K, C). \]
Let $\f{g}$ be the Lie algebra of $G$ and $U(\f{g})$ its universal enveloping algebra. Let $\mathring{\bb{G}}$ be a Stein group neighbourhood of $G$ as in Theorem \ref{TheoMainintro} and $\n{D}(\mathring{\bb{G}},K)$ the distribution algebra of $\mathring{\bb{G}}$. If in addition $C$ is $\mathring{\mathbb{G}}$-analytic, then it is equipped with an action of $\mathfrak{g}$ by derivations and we define the $\mathring{\mathbb{G}}$-analytic cohomology (resp.  the Lie algebra cohomology) of $C$ as
$ R\iHom_{\n{D}(\mathring{\bb{G}}, K)}(K, C)$ (resp. $R\iHom_{U(\f{g})}(K, C)$). Using bar resolutions and Theorem \ref{TheoMainintro}, one verifies that these definitions recover the usual continuous, analytic, and Lie algebra cohomology groups. \\

Our first new result compares continuous   cohomology of a solid representation and the continuous  cohomology of its locally analytic vectors. This result can be seen as a $p$-adic analogue of a theorem of P. Blanc \cite{Blanc} and G. D. Mostow \cite{Mostow} in the archimedean setting, which compares continuous and differentiable cohomology of a real Lie group $G$. \\

Let $C\in D(K_\sol[G])$, we define the derived locally analytic vectors of $C$ as the homotopy colimit 
\[
C^{Rla}= \underset{\bb{G}}{\hocolim}\; C^{R \bb{G}-an},
\]
where $\bb{G}$ runs over all the affinoid neighbourhoods of $G$. We say that $C$ is derived locally analytic if $C^{Rla} = C$. If $V$ is a Banach representation, then $H^0(V^{Rla})$ coincides with the locally analytic vectors of $V$ in the sense of \cite[Definition 3.5.3]{Emerton}.  We have the following theorem.

\begin{theorem}[Theorem \ref{theocohom1}]
\label{theocoholaintro}
Let $C \in D(K_\sol[G])$ and let $C^{Rla}$ be the complex of derived locally analytic vectors of $C$. Then
\[ R\iHom_{K_{\sol}[G]}(K, C) \cong R\iHom_{K_{\sol}[G]}(K, C^{Rla}). \]
In particular, if $V\in \SolidKG$ then,  setting $V^{R^ila} := H^i(V^{Rla})$ for $i\geq 0$ there is a spectral sequence of solid $K$-vector spaces
\[ E_2^{i,j} := \underline{\Ext}^i_{K_{\sol}[G]}(K, V^{R^jla}) \implies \underline{\Ext}^{i + j}_{K_{\sol}[G]}(K, V). \]
\end{theorem}

We give an application in the classical context. Let $V$ be a continuous representation of $G$ on a complete locally convex $K$-vector space. Denote $H^i_{cont}(G, V)$ the usual continuous cohomology groups. These coincide with the underlying sets of the solid cohomology groups $\underline{\Ext}^i_{K_{\sol}[G]}(K, V)$. We say that $V$ has no higher locally analytic vectors if $V^{R^i la} = 0$ for all $i > 0$. This is the case for admissible representations (cf. \cite[Theorem 2.2.3]{LuePan} or Proposition \ref{PropAdmissible}). One deduces the following corollary.

\begin{corollary}
If $V$ has no higher locally analytic vectors, then for all $i\geq 0$,
\[ H^i_{cont}(G, V) = H^i_{cont}(G, V^{la}). \]
\end{corollary}


Our last result concerns a generalisation of Lazard's comparison between continuous, (locally) analytic and Lie algebra cohomology from finite dimensional representations $V$ to arbitrary solid derived (locally) analytic representations. We have the following theorem.

\begin{theorem} [Continuous vs. analytic  vs. Lie algebra cohomology, Theorem \ref{theocohom2}] \label{theocohom2intro}
 Let $C \in D(K_{\sol}[G])$ be a   derived $\mathring{\mathbb{G}}$-analytic complex. Then\footnote{The Lie algebra cohomology $R\underline{\Hom}_{U(\f{g})}(K,C)$ lands naturally in the derived category of smooth representations of $G$ on solid $K$-vector spaces.  Since $K$ is of characteristic $0$,  taking $G$-invariants in this category is exact and  the superscript $G$ means the composition with this functor,     cf. Remark \ref{RemLiealgcohom}.}
\[ 
R\iHom_{K_{\sol}[G]}(K, C) \cong R\iHom_{\n{D}(\mathring{\bb{G}}, K)}(K, C)  \cong (R\iHom_{U(\f{g})}(K, C))^G.
\]
\end{theorem}




\subsection{Organisation of the paper}

In Section \ref{Sectionrecollections} we review very briefly the theory of condensed mathematics and solid abelian groups. We recall the notion of analytic ring and give some examples that will be used throughout the text. \\

In Section \ref{SectionNonArch}   we develop the theory of solid $K$-vector spaces following the appendix of \cite{GuidoDrinfeld}. Most of the results exposed in this section will be  presented in the forthcoming work of Clausen and Scholze \cite{ClausenScholzeCM}. We review in particular the main properties of classical vector spaces: Banach, Fr\'echet, $LB$ and  $LF$ spaces. Our main original result is Theorem \ref{Theodualityintro} generalising the classical anti-equivalence  between $LB$ spaces of compact type and nuclear Fr\'echet vector spaces. \\

In Section \ref{SectionRepTheory} we introduce the different analytic neighbourhoods $\bb{G}^{(h)}$ of our $p$-adic Lie  group $G$. We begin in \S \ref{subsecFuncDist} by introducing spaces of analytic functions, following \cite[\S 2]{Emerton} closely. Then, we define the algebras of distributions as the duals of the spaces of analytic functions. We also introduce another class of distribution algebras, used already in \cite[\S 4]{SchTeitDist}, which are more adapted to the coordinates of the Iwasawa algebra. In \S \ref{subsecAnalyticrep},  we define the notion of analytic and derived analytic representation, we prove Theorem \ref{TheoMainintro}, except for a technical lemma whose proof is postponed to Section \ref{SecCohomology}.  We finish the section with some applications to locally analytic  and admissible representations. \\

Finally, in Section \ref{SecCohomology},  we recall (Theorem \ref{LazardSerre}) a lemma of Serre used by Lazard to construct finite free resolutions of the trivial representation when $G$ is a uniform pro-p-group.  We use this result, as well as its enhancement due to Kohlhaase (Theorem \ref{propKohlhaase}) to    prove the technical lemma  necessary for  Theorem \ref{TheoMainintro}.  We state Theorems \ref{theocoholaintro} and  \ref{theocohom2intro} in \S \ref{SubsecMainResults} and give a proof in \S \ref{Subsectionproof}. We conclude with some formal consequences, namely by showing a solid version of Hochschild-Serre and proving a duality  between group  homology and cohomology.

\subsection{Acknowledgements}

The authors learnt the theory of non-archimedean functional analysis from Guido Bosco in the context of a study group held in La Tourette, November 2020. We would like to express our gratitude to Guido for this, and for sharing with us his notes and an earlier version of \cite{GuidoDrinfeld}. We would also like to thank all the other participants of La Tourette, from whom we have profited from several fruitful conversations; special thanks to Andreas Bode, George Boxer and Vincent Pilloni, for their comments and corrections of an earlier version of this paper. We also thank Antonio Cauchi for useful comments, Gabriel Dospinescu for his interest in our cohomological comparison results, Arthur-C\'esar Le Bras for suggesting the duality result, and Lucas Mann for sharing with us some sections of his Ph.D. thesis. We thank Peter Scholze for his many comments and careful corrections. Finally,  we thank the anonymous referee for the careful proofreading of the document and all the  numerous comments and corrections that  considerably improved the exposition. The first named author was supported by the ERC-2018-COG-818856-HiCoShiVa.

\section{Recollections in condensed mathematics}
\label{Sectionrecollections}

First,  we review some elementary  notions in condensed mathematics: we recall the definitions of  condensed sets,  solid abelian groups and analytic rings. In the future we will be only interested  in the categories of solid modules, and modules of analytic rings over $\Z_\sol$.

\subsection{Condensed objects}

In their recent work \cite{ClausenScholzeCondensed2019} and \cite{ClauseScholzeanalyticspaces},  Clausen and Scholze have introduced the new world of condensed mathematics, which aims to be the good framework where algebra and topology live together. Let $*_{\proet}$ be the   pro-\'etale site of a geometric point, equivalently, the category of profinite sets with continuous maps and coverings given by finitely many jointly surjective  maps. 

\begin{definition}[{\cite[Definitions 2.1 and 2.11]{ClausenScholzeCondensed2019}}]
A condensed set/group/ring/...  $\n{F}$ is a sheaf over $*_{\proet}$ with values in  $Sets$, $Groups$,  $Rings$,... We denote by $\Cond$ the category of condensed sets and $\Condab$ that of condensed abelian groups.  If $R$ is  a condensed  ring, we denote by $\Mod_{R}^{\cond}$ the category of condensed $R$-modules.
\end{definition}

\begin{remark}
There are set theoretical subtilities  with this definition as $*_{\proet}$ is not small. What Clausen and Scholze do (cf. \cite[Lecture II]{ClausenScholzeCondensed2019}) is to cut-off by an uncountable strong limit cardinal  $\kappa$, considering the category of sheaves of  $\kappa$-small profinite sets $*_{\proet}$ and take the direct limit. Hence the category of condensed sets is not the category of sheaves on a site. In this article we proceed in the same manner and we subsequently avoid any further mention to $\kappa$.
\end{remark}

Let $\Top$ denote the category of $T1$ topological spaces, we define the functor $( \underline{\hspace{8pt}}): \Top \to  \Cond$ mapping a topological space $T$ to the condensed set 
\[
T: S \mapsto  \underline{T}(S)=  \Cont(S,T),
\]
where $\Cont(S,T)$ is the set of continuous functions from $S$ to $T$. Recall that a topological space $X$ is compactly generated if a map $X\to Y$ to another topological space $Y$ is continuous  if and only if the composition $S\to X\to Y$ is continuous for all maps $S\to X$ from a compact Hausdorff space. Then the restriction of the functor $T \mapsto \underline{T}$ to the category of compactly generated Hausdorff topological spaces is fully faithful. 

A compact Hausdorff space can be written as a quotient of a profinite set. Indeed, if $S$ is a compact Hausdorff space, let $S_{dis}$ denote the underlying set with the discrete topology, and $\beta S_{dis}$ its Stone-\v{C}ech compactification. Then $\beta S_{dis}$ is profinite and the natural map $\beta S_{dis}\to S$ is a surjective map of compact Hausdorff spaces. In particular, we can test if a topological space is compactly generated by restricting to  profinite sets.

Recall that a condensed set $X$ is called quasicompact if there is a profinite set $S$ and a surjective map $\underline{S} \to X$. Similarly, a condensed set $X$ is quasiseparated if for any pair of profinite  sets $S$ and $S'$ over $X$ the fiber product $\underline{S}\times_{X} \underline{S'}$ is quasicompact. More generally,  a map $Y\to X$ of condensed sets is quasicompact if for any profinite set $S$ and any map $S\to X$,  the fiber product $S\times_X Y$ is quasicompact,  the map $Y\to X$ is quasiseparated if the diagonal $Y \to  Y\times_X Y$  is quasicompact.

From now on we identify a profinite set $S$ with the condensed set $\underline{S}$.  Given a condensed set $X$ and a profinite set $S$ we denote by $\Cont(S,X)=X(S)$ the maps from $S$ to $X$ as condensed sets,  we also define    $\underline{\Cont}(S,X)$  to be the condensed set whose value at a profinite set $S'$ is
\[
\underline{\Cont}(S,X)(S')= \Cont(S\times S',X)=X(S\times S'). 
\]
We have the following result.

\begin{prop}[{\cite[Proposition 1.2]{ClauseScholzeanalyticspaces}}]
\label{propFunctorTopCond}
Consider the functor $T \mapsto \underline{T}$ from $\Top$ to $\Cond$

\begin{enumerate}
    \item The functor has a left adjoint $X\mapsto X(*)_{\mathrm{top}}$ sending any condensed set $X$ to the set $X(*)$ equipped with the quotient topology arising from the map 
    \[
    \bigsqcup_{S, a\in X(S)} S \to X(*)
    \]
    with $S$ profinite. 
    
    \item Restricted to compactly generated topological spaces, the functor is fully faithful. 
    
    \item The functor induces an equivalence between the category of compact Hausdorff spaces and qcqs condensed sets. 
    
    \item The functor induces a fully faithful functor from the category of compactly generated weak Hausdorff spaces, to quasiseparated condensed sets. The category of quasiseparated condensed sets is equivalent to the category of ind-compact Hausdorff spaces ``$\varinjlim_i T_i$'' where all transition maps $T_i\to T_j$ are closed immersions.  If $X_0\to X_1\to \cdots$ is a sequence of compact Hausdorff spaces with closed immersions and $X=\varinjlim_n X_n$ as topological spaces,  then the map 
    \[
    \varinjlim_n \underline{X}_n \to \underline{X}
    \]
    is an isomorphism  of condensed sets. 
    
\end{enumerate}

\end{prop}

Among the class of profinite sets there is the special class of extremally disconnected sets, which are the projective objects in the category $*_{\proet}$.  Moreover, all of them are retractions of a Stone-\v{C}ech compactification of a discrete set. Let $\rm Extdis$ denote the full subcategory of extremally disconnected sets.    The condensed sets can be defined using only this kind of profinite sets. 

\begin{prop}[{\cite[Proposition 2.7]{ClausenScholzeCondensed2019}}]
\label{PropExtDisc}
Consider the site of extremally disconnected sets with covers given by finite families of jointly surjective maps.  Its category  of sheaves is equivalent to the category of condensed sets via the  restriction from profinite sets.  Hence,  a condensed set is a functor $X:  \Extdis\to \Sets$ such that $X(\emptyset)=*$ and $X(S_1\bigsqcup S_2)= X(S_1)\times X(S_2)$. 
\end{prop}

Extremally disconnected sets play a similar role as points do for locally ringed spaces.  Namely, if $F\to G$ is a map of condensed sets, it is injective (resp. surjective) if and only if it is so after evaluating at all extremally disconnected sets. 

The inclusion $\Condab\to \Cond$ admits a left adjoint $T\mapsto \Z[T]$, where $\Z[T]$ is the sheafification of $S\mapsto \Z[T(S)]$. Let $S$ be an extremally disconnected set,  Proposition \ref{PropExtDisc} implies that  the object $\Z[S]$ is projective in the category $\Condab$. We have a tensor product in $\Condab$ given by the sheafification of the usual tensor product at the level of points.  We also have an internal $\underline{\Hom}$ in $\Condab$ defined as 
\[
\underline{\Hom}_{\Z}(M,N)(S)= \Hom_{\Z}(M\otimes \Z[S], N)
\]
for any profinite set $S$. If $R$ is a condensed ring we write $R[S]:= R\otimes \Z[S]$. They form a family of compact projective generators of $\Mod_{R}^{\cond}$. Then, if $X$ is a condensed abelian group and $S$ is a profinite set we have $\underline{\Cont}(S,X)= \underline{\Hom}_{\Z}(\Z[S], X)$.  All the nice properties of the category of condensed abelian groups are summarised in the following theorem 

\begin{theorem}[{\cite[Theorem 2.2]{ClausenScholzeCondensed2019}}] \label{TheoScholzeAB}
The category of condensed abelian groups is an abelian category which satisfies the Grothendieck axioms $(AB3)$,  $(AB4)$,  $(AB5)$,  $(AB6)$,  $(AB3^*)$ and $(AB4^*)$:  all limits $(AB3^*)$ and colimits $(AB3)$ exist, arbitrary products $(AB4^*)$, arbitrary direct sums $(AB4)$ and filtered colimits $(AB5)$ are exact, and $(AB6)$: for all index sets $J$ and filtered categories $I_j$, $j\in J$,  with functor $i\mapsto M_i$, from $I_j$ to condensed abelian groups, the natural map 
\[
\varinjlim_{(i_j\in I_j)_{j\in J}} \prod_{j\in J} M_{i_j} \to \prod_{j\in J} \varinjlim_{i_j\in I_j} M_{i_j}
\]
is an isomorphism. Moreover, the category of condensed abelian groups is generated by compact projective objects given by $\Z[S]$ with $S$ an extremally disconnected set. 
\end{theorem}

\subsection{Quasiseparated condensed sets}

We collect here some basic facts on quasiseparated condensed sets that will be needed later. 

\begin{proposition} [{\cite[Proposition 4.13]{ClauseScholzeanalyticspaces}}]
\label{PropQuasiseparatedCondensed}
Let $X$ be a quasiseparated condensed set. Then quasicompact injections $\iota: Z \to X$ are equivalent to closed subspaces $W \subseteq X(*)_{\mathrm{top}}$ via $Z \mapsto Z(*)_{\mathrm{top}}$, resp. sending a closed subspace $W \subseteq X(*)_{\mathrm{top}}$ to the subspace $Z \subseteq X$ given by
\[ Z(S) = X(S) \times_{\mathrm{Cont}(S, X(*)_{\mathrm{top}})} \mathrm{Cont}(S, W). \]
\end{proposition}

From now on, we will refer to quasicompact injections $\iota : Z \to X$ as closed subspaces of $X$.

\begin{lemma}[{\cite[Lemma 4.14]{ClauseScholzeanalyticspaces}}]
\label{LemmaQuasiseparatedQuotient}
Let $X$ be a condensed set. The inclusion of the category of quasiseparated condensed sets into the category of all condensed sets admits a left adjoint $X \mapsto X^{qs}$, with the unit $X \to X^{qs}$ being a surjection of condensed sets. The functor $X \mapsto X^{qs}$ preserves finite products. Moreover, for any quasiseparated condensed ring $A$, this functor defines a similar left adjoint for the inclusion of quasiseparated condensed $A$-moules into all condensed $A$-modules.
\end{lemma}

\subsection{Analytic rings}

Next, we recall the notion of analytic ring

\begin{definition}[{ \cite[Definitions 7.1, 7.4]{ClausenScholzeCondensed2019}  and \cite[Definition 6.12]{ClauseScholzeanalyticspaces}}]
A pre-analytic ring $(\n{A}, \n{M})$ is the data of a condensed ring $\n{A}$  (called the underlying condensed ring of the analytic ring) equipped with a functor 
\[
\Extdis \to \Mod_{\n{A}}^{\cond}, \;\;\; S\mapsto \n{M}[S],
\]
called the functor of measures of $(\n{A}, \n{M})$, that sends finite disjoint union into products, and a natural transformation  of functors
 $S \to \n{M}[S]$.

A pre-analytic ring is said to be analytic, if for any complex $C: \hdots \to C_i \to \hdots \to C_1 \to C_0 \to 0$ of $\n{A}$-modules such that each $C_i$ is a direct sum of objects of the form $\n{M}[T]$ for varying extremally disconnected sets $T$, the map


 \[
 R\underline{\Hom}_{\n{A}}(\n{M}[S], C) \to R \underline{\Hom}_{\n{A}}(\n{A}[S],C)
 \]
is an isomorphism for all extremally disconnected sets $S$. An analytic ring $(\n{A}, \n{M})$ is normalized if $\n{A}\to \n{M}[*]$ is an isomorphism. 
\end{definition}

\begin{example} \leavevmode
\label{Exampleanalyticrings}
\begin{enumerate}
    \item  (\cite[Theorem 5.8]{ClausenScholzeCondensed2019}).  We define the pre-analytic ring $\Z_{\sol}$ with underlying ring $\Z$ and functor of measures mapping an extremally disconnected $S=\varprojlim_{i}S_i$, written as an inverse limit of finite sets, to the condensed abelian group $\Z_{\sol}[S]:= \varprojlim_{i} \Z[S_i]$. Then $\Z_{\sol}$ is an analytic ring.
    
    \item (\cite[Theorem 8.1]{ClausenScholzeCondensed2019}). More generally, let $A$ be a discrete commutative algebra, and let $S=\varprojlim_i S_i$ be an extremally disconnected set.  We have an analytic ring $A_{\sol}$ with underlying condensed ring $A$, and  with functor of measures  given by 
    \[
    A_{\sol}[S]= \varinjlim_{B\subset A} \varprojlim_i B[S_i],
    \]
    where $B$ runs over all the $\Z$-algebras of finite type in $A$.  
    
    \item (\cite[Proposition 7.9]{ClausenScholzeCondensed2019}) Let $p$ be a prime number,  $K$ a finite extension of $\Q_p$ and $\O_K$ its valuation ring. Let $S=\varprojlim_i S_i$ be an extremally disconnected set.  We have analytic structures for the rings $\O_K$ and $K$, denoted $\O_{K,\sol}$ and $K_{\sol}$ respectively, given by 
    \[
    \O_{K,\sol}[S]:= \varprojlim_i \O_K[S_i]\mbox{ and } K_{\sol}[S]:= \O_{K,\sol}[S][\frac{1}{p}]= K\otimes_{\O_K} \O_{K,\sol}[S]. 
    \] 
     This analytic ring structure is induced from the analytic ring structure of $\Z_{\sol}$ by base change to $\n{O}_K$.
    
    \item (\cite[Theorem 1.5]{Gregory}) Let $(A,A^+)$ be a  complete Huber pair. Andreychev defines an analytic ring $(A,A^+)_{\sol}$ associated to $(A, A^+)$, whose underlying ring is $\underline{A}$, and with functor of measures 
    \[
    (A,A^+)_{\sol}[S] = \varinjlim_{B\to A^+,M} \varprojlim_i \underline{M}[S_i],
    \]
    where  the colimit is taken over all the finitely generated subrings $B\subset A^+$ and all the quasi-finitely generated $B$-submodules $M$ of $A$ (cf. \cite[Definition 3.4]{Gregory}). If $A^+=A^{\circ}$ is the ring of power bounded elements,   we simply write $A_{\sol}$ for $(A,A^{\circ})_{\sol}$. 
    
    \item (\cite[Lemma 4.7]{Gregory}) A particular example of the previous case is the analytic ring associated to the Tate algebra $(K\langle T \rangle , \O_K \langle T \rangle)$, namely, the analytic ring  $K\langle T  \rangle_{\sol}$ whose functor of measures is given by 
    \[
    K\langle T \rangle_{\sol}[S]= K[S] \otimes_{\Z_{\sol}} \Z[T]_{\sol}.    
    \]
    
\end{enumerate}
\end{example}

The following theorem explains the importance of the analytic rings 
\begin{theorem}[{\cite[Proposition 7.5]{ClausenScholzeCondensed2019}}]
\label{TheoLemmaAnalyticrings}
Let $(\n{A}, \n{M})$ be an analytic ring. 
\begin{enumerate}
    \item  The full subcategory 
    \[
    \Mod_{(\n{A},\n{M})}^{\cond} \subset \Mod_{\n{A}}^{\cond}
    \]
of all $\n{A}$-modules $M$ such that for all extremally disconnected set $S$, the map 
\[
\Hom_{\n{A}}(\n{M}[S],  M) \to \Hom_{\n{A}}(\n{A}[S],  M)
\]
is an isomorphism,  is an abelian subcategory stable under all limits, colimits and extensions.  Objects of the form $\n{M}[S]$,  where $S$ is an extremally disconnected profinite set, constitute a family of compact projective generators of $\Mod_{(\n{A}, \n{M})}^{\cond}$.  The inclusion functor admits a left adjoint 
\[
\Mod_{\n{A}}^{\cond}\to \Mod_{(\n{A},\n{M})}^{\cond}, \;\;\; M\mapsto M \otimes_{\n{A}}(\n{A}, \n{M})
\]
which is the unique colimit preserving extension of the functor given by $\n{A}[S]\mapsto \n{M}[S]$. Finally, if $\n{A}$ is commutative,   there is a unique symmetric monoidal tensor product $-\otimes_{(\n{A},\n{M})} -$ on $\Mod_{(\n{A},\n{M})}^{\cond}$ making the functor $-\otimes_{\n{A}}(\n{A},\n{M})$ symmetric monoidal.  We call $ \Mod_{(\n{A},\n{M})}^{\cond}$  the category of $(\n{A},  \n{M})$-modules.   

\item The functor of derived categories 
\[
D(\Mod_{(\n{A},\n{M})}^{\cond}) \to D(\Mod_{\n{A}}^{\cond})
\]
is fully faithful and its essential image  is stable under all limits and colimits and given by those $C\in  D(\Mod_{\n{A}}^{\cond})$ for which the map 
\[
R\Hom_{\n{A}}(\n{M}[S], C) \to R\Hom_{\n{A}}(\n{A}[S],C)
\]
is an isomorphism for all extremally disconnected set  $S$.  In that case, the map 
    \[
    R\underline{\Hom}_{\n{A}}(\n{M}[S], C) \to R\underline{\Hom}_{\n{A}}(\n{A}[S],C)
    \]
    is also an isomorphism. 
    
    An object $C\in D(\Mod_{\n{A}}^{\cond})$ lies in $   D(\Mod_{(\n{A},\n{M})}^{\cond})$ if and only if for each $n\in \Z$,  the cohomology group $H^n(C)$ lies in $\Mod_{(\n{A},\n{M})}^{\cond}$. The inclusion functor     $D(\Mod_{(\n{A},\n{M})}^{\cond}) \subset  D(\Mod_{\n{A}}^{\cond})$ admits a left adjoint 
    \[
    D(\Mod_{\n{A}}^{\cond}) \to D(\Mod_{(\n{A},\n{M})}^{\cond}), \;\;\; C \mapsto C\otimes^{L}_{\n{A}} (\n{A}, \n{M})
    \]
    which is the left derived functor of $\n{M}\mapsto \n{M}\otimes_{\n{A}}(\n{A},\n{M})$.  Finally, if $\n{A}$ is commutative, there is a unique symmetric monoidal tensor product $-\otimes^{L}_{(\n{A},\n{M})}-$ on $D(\Mod_{(\n{A},\n{M})}^{\cond})$ making the functor $-\otimes_{\n{A}}^{L} (\n{A},\n{M}) $ symmetric monoidal. 
\end{enumerate}
\end{theorem}

\begin{remark}
The functor $-\otimes^{L}_{(\n{A},\n{M})}-$ is the derived functor of $-\otimes_{(\n{A},\n{M})}-$ if and only if $\n{M}[S']\otimes^L_{(\n{A},\n{M})}\n{M}[S'']= \n{A}[S\times S']\otimes^{L}_{\n{A}} (\n{A}, \n{M})$ sits in degree $0$, cf. \cite[Warning 7.6]{ClausenScholzeCondensed2019}. This is the case for  the analytic rings of Example \ref{Exampleanalyticrings}, and in fact for all analytic ring over $\bb{Z}_\sol$. 
\end{remark}

\begin{remark}
The terminology for the functor of measures is justified as follows.  If $(\n{A},\n{M})$ is an analytic ring, $P$ is an $(\n{A},\n{M})$-module, $S$ is an extremally disconnected set, $f \in \Cont(S, P)$ and $\mu \in \mathcal{M}[S]$ then, using the isomorphism $\Hom_{\n{A}}(\n{M}[S], P) = \Cont(S, P)$, one can evaluate $f$ at $\mu$ and define $\int f \cdot \mu = f(\mu) \in P$, which allows to see $\mu$ as a function on the space of functions $f \in \Cont(S, P)$.
\end{remark}

In the following we shall write $D(\n{A},\n{M})$ for the derived category $D(\Mod_{(\n{A},\n{M})}^{\cond})$. The functor  $-\otimes^{L}_{\n{A}} (\n{A},\n{M})$ should be thought as a completion with respect to the measures $\n{M}$.

One of the main theorems of \cite{ClausenScholzeCondensed2019} is the proof that $\Z_{\sol}$ is an analytic ring, it has as input several non trivial computations of $\Ext$-groups of locally compact abelian groups. The category $\Mod_{\Z}^{\solid}:= \Mod^{\cond}_{\Z_{\sol}}$ is called the category of solid abelian groups.

Let $p$ be a prime number, then $\Z_p$ is a solid abelian group (being an inverse limit of discrete abelian groups) and for $S$ extremally disconnected we have (cf. \cite[Proposition 7.9]{ClausenScholzeCondensed2019})
\begin{equation} \label{EqZpsol} \Z_{p,\sol}[S]= \Z_p \otimes^L_{\Z_{\sol}} \Z_{\sol}[S] = \Z_p \otimes_{\Z_{\sol}} \Z_{\sol}[S].
\end{equation}

\begin{definition}
Let $K$ be a finite extension of $\Q_p$ and $\O_K$ its ring of integers, we denote by $\SolidOK$  (resp. $\SolidK$) the category  of solid $\O_K$-modules  (resp. the category  of solid $K$-vector spaces) and $D(\O_{K, \sol})$, $D(K_{\sol})$ their respective derived categories.  
\end{definition}

\begin{remark} 
An object in $\SolidOK$ (resp. $\SolidK$) is the same as an object in $\Mod_{\Z}^{\solid}$ endowed with an action of $\O_K$ (resp. $K$). Indeed, this follows directly from Theorem \ref{TheoLemmaAnalyticrings} and \eqref{EqZpsol}.
\end{remark}

\subsection{Analytic rings attached to Tate power series  algebras }

Let $A$ be a reduced  Tate algebra of finite type over $\Q_p$,  recall that $A_{\sol}=(A,A^{\circ})_{\sol}$ where $A^{\circ}\subset A$ is the subring of power bounded elements.    
\begin{definition}
  We define the ring $A\langle\langle  T^{-1}  \rangle  \rangle$ to be  
\[
A\langle\langle  T^{-1}  \rangle  \rangle  = \bigg(\varprojlim_{s\to \infty} A^{\circ}/\varpi^s ((T^{-1})) \bigg)[\frac{1}{\varpi}],   
\]
where $A^{\circ}/\varpi^s ((T^{-1})) = (A^{\circ}/\varpi^s )[[T^{-1}]][T]$. 
\end{definition}

\begin{lemma} \label{lemmares}
The residue pairing  $\mathrm{res}_0: A\langle \langle  T^{-1} \rangle\rangle \times  A\langle T\rangle \to A $ defined as 
\[
(\sum_{n\in \Z} a_{n} T^{n} ,  \sum_{m\in \N} b_{m} T^{m} ) \mapsto \mathrm{res}_0( (\sum_{n\in \Z} a_{n} T^{n}) (\sum_{m\in \N} b_{m} T^{m} )) = \sum_{n + m = -1} a_n b_m
\]
induces an isomorphism  of $A\langle T \rangle$-modules
\[
\frac{A\langle \langle  T^{-1} \rangle\rangle}{A\langle T \rangle } \xrightarrow{\sim } \iHom_A(A\langle T \rangle , A ).  
\]

\proof
We first observe that the pairing is well defined as a map of condensed sets. Indeed, the ring $A\langle \langle T^{-1}\rangle \rangle$ can be written as $A\langle \langle T^{-1}\rangle \rangle = (T^{-1}A^{\circ}[[T^{-1}]]\oplus A^{\circ}\langle T \rangle )[\frac{1}{\varpi}]$ and,  in particular,  it is the condensed set associated to its underlying topological space.  Since the pairing is well defined at the level of topological spaces,  taking associated condensed sets we are done. 

It is clear from the definition of $\mathrm{res}_0$ that the induced  map $A\langle\langle T^{-1} \rangle\rangle \to \iHom_A(A\langle T\rangle, A) $ is $A\langle T \rangle$ linear,  and that it factors through $\frac{A\langle\langle T^{-1} \rangle\rangle}{A\langle T \rangle}$.   On the other hand, we have
\begin{eqnarray*}
\iHom_A(A\langle T\rangle, A) &=& \iHom_{A^{\circ}}(A^{\circ}\langle T \rangle, A^{\circ})[\frac{1}{\varpi}] \\
&=& \varprojlim_{s\in \N} \iHom_{A^{\circ}/\varpi^s} (A^{\circ}/\varpi^s[T],  A/\varpi^{s}) [\frac{1}{\varpi}] \\
&=& \varprojlim_{s\in \N}   \bigg( \prod_{n\in \N}( A^{\circ}/\varpi^s) T^{n,\vee}  \bigg)[\frac{1}{\varpi}]  \\
& = & \bigg( \prod_{n\in \N} A^{\circ}T^{n,\vee} \bigg) [\frac{1}{\varpi}],
\end{eqnarray*}
where $\{T^{n,\vee}\}_{n\in \N}$ is the dual basis of $\{T^n\}_{n\in \N}$. But the map $\mathrm{res}_0$ sends $ T^{-n-1}$ to $T^{n,\vee}$.  The lemma follows.    
\endproof
\end{lemma}

\begin{remark}
Let $a_1,\ldots,  a_k\in A^{\circ}$ be a finite set of topological generators over $\zp$,  then $A^{\circ}$ is quasi-finitely generated $\Z[a_1,\ldots, a_k]$-module,  namely,  $A^{\circ}/\varpi^s$ is a finite  $\Z[a_1,\ldots, a_k]$-module for all $s\geq 1$,  cf. \cite[Definition 3.4]{Gregory}. By \cite[Lemma 3.26]{Gregory}, this implies that
\[
 \prod_{I} \Z  \otimes_{\Z_{\sol}}   A_{\sol} =  \bigg( \prod_{I} A^{\circ} \bigg) [\frac{1}{\varpi}].  
\]
In particular,  since $\frac{A\langle \langle T^{-1} \rangle \rangle}{A\langle T \rangle} =  \frac{A^\circ((T^{-1}))}{ A^\circ[T]}[\frac{1}{\varpi}]$ as $A$-module, we get
\[
\frac{A\langle \langle T^{-1} \rangle \rangle}{A\langle T \rangle}=  \frac{\Z((T^{-1}))}{ \Z[T]} \otimes_{\Z_{\sol}} A_{\sol}.  
\]
\end{remark}

Let us recall the following proposition concerning the solidification  of $\Z[T]_{\sol}$-modules.

\begin{prop}[{\cite[Proposition 3.13]{Gregory}}]
\label{PropAndreychevSolidification}
Let $C \in D( \Z_{\sol} )$,  then 
\[
C\otimes^L_{\Z_{\sol}} \Z[T]_{\sol} = R\iHom_{\Z}( \frac{\Z((T^{-1}))}{\Z[T]},  C).  
\]
\end{prop}

\begin{cor}
\label{PropsixFunctorTate}
Let  $A$ be a reduced Tate algebra of finite type over $\qp$.  Let   $A\langle \underline{T} \rangle = A \langle T_1, \ldots, T_d\rangle$ be  the Tate power series algebra in $d$ variables,   and  $A\langle \underline{T} \rangle^{\vee}= \iHom_{A}(A\langle \underline{T} \rangle,  A)$.  Then for any $W,C\in D(A_{\sol})$  we have a functorial quasi-isomorphism 
\[
R\iHom_{A}(W,  C\otimes^L_{A_{\sol}} A\langle \underline{T} \rangle_{\sol}) = R\iHom_{A}( W\otimes_{A_{\sol}}^L A\langle\underline{T} \rangle^{\vee}, C). 
\]
\end{cor}

\begin{proof}
Let us first proof the case of one variable, which follows essentially from Proposition \ref{PropAndreychevSolidification}. Indeed, by  \cite[Lemma 4.7]{Gregory} we have $A\langle T \rangle_{\sol}= A_{\sol}\otimes_{\Z_{\sol}}^L \Z[T]_{\sol}$,  so that $C\otimes^{L}_{A_{\sol}} A\langle T \rangle_{\sol}= C \otimes^{L}_{\Z_{\sol}} \Z[T]_{\sol}$.  Thus,  we get 
\begin{eqnarray*}
R\iHom_{A}(W,  C\otimes^L_{A_{\sol}} A\langle \underline{T} \rangle_{\sol}) & = & R\iHom_{A}(W, C \otimes^L_{\Z_{\sol}} \Z[T]_{\sol}) \\
                & = & R\iHom_{A}\bigg(W\otimes_{\Z_{\sol}}^L\frac{\Z((T^{-1}))}{\Z[T]},  C \bigg)\\
                & = & R\iHom_{A}( W \otimes^L_{A_{\sol}} \frac{A\langle \langle T^{-1} \rangle\rangle}{A\langle T \rangle},  C   ) \\
                & = & R\iHom_{A}(W \otimes_{A_{\sol}} A\langle T \rangle^{\vee}, C).  
\end{eqnarray*}

Assume that the result holds for $d-1$ variables and let $A' = A\langle T_1,\ldots ,T_{d-1} \rangle$, and set $W'= W\otimes_{A_{\sol}}^{L} A'_{\sol}$ and $C' = C \otimes^{L}_{A_{\sol}} A'_{\sol}$.  Then,  by the case of one variable we have 
\begin{equation}
\label{eqOneVariableCase}
R\iHom_{A'}(W', C' \otimes^{L}_{A'_{\sol}} A'_{\sol}\langle T_{d} \rangle) = R\iHom_{A'}(W' \otimes^{L}_{A'} \iHom_{A'}(A'\langle T_d \rangle , A') , C')
\end{equation}
Observe also that
\[
\iHom_{A'}(A'\langle T_d \rangle,  A') = A \langle T_d \rangle^\vee \otimes_{A_{\sol}}^{L} A'_{\sol} 
\]
and $(A')^{\vee}\otimes_{A_{\sol}}^L A\langle T_d \rangle^{\vee} = A \langle T_1, \ldots,  T_d\rangle^{\vee}$. Then, 
\begin{eqnarray*}
R\iHom_{A}(W, C\otimes^L_{A_{\sol}} A\langle \underline{T} \rangle_\sol) & = & R\iHom_{A'}(W',  C' \otimes^L_{A'_{\sol}} A'\langle T_{d}\rangle_{\sol}) \\&= & R\iHom_{A'} (W\otimes^L_{A_{\sol}} A\langle T_d \rangle^\vee \otimes^L_{A_{\sol}} A'_{\sol} ,  C'  ) \\
& = & R\iHom_{A} (W\otimes^L_{A_{\sol}} A\langle T_d \rangle^\vee  ,  C\otimes^L_{A_{\sol}} A'_{\sol}  ) \\
& = & R\iHom_{A}( W\otimes^L_{A_{\sol}} A\langle T_d \rangle^{\vee} \otimes^L_{A_\sol} (A')^{ \vee},  C  ) \\
& = & R\iHom_{A}(W\otimes^L_{A_{\sol}} A\langle \underline{T} \rangle^{\vee},  C),
\end{eqnarray*}
where the second equality equality follows from Equation \eqref{eqOneVariableCase}, the third equality is a base change from $A$ to $A'$ and the fourth equality uses the induction step for $d-1$ variables.  This finishes the proof.  
\end{proof}

\begin{remark} 
Corollary \ref{PropsixFunctorTate} can be seen as an instance of the six functor formalism of \cite[Theorem 8.2]{ClausenScholzeCondensed2019}. Indeed, let $f:  \Spa(A\langle T \rangle,  A^{\circ}\langle T \rangle) \to \Spa(A,A^{\circ})$. By a slight extension of the results of \cite[Lecture VIII]{ClausenScholzeCondensed2019}, one can define pairs of adjoint functors $f^*:  D(A_{\sol}) \rightleftharpoons  D(A\langle T \rangle_{\sol}):f_*$ and  $f_!: D(A\langle T \rangle_{\sol})  \rightleftharpoons   D(A_{\sol})  :  f^!$,  which are the basic cases of  the six functor formalism.  The functor $f_* $ is the forgetful functor and $f^*$ is given by the solidification functor $- \otimes^L_{A_{\sol}} A\langle T \rangle_\sol$. The other two functors are given by $f_! M = M \otimes_{(A\langle T \rangle, A^{\circ})_{\sol}}^L \frac{A\langle \langle T^{-1} \rangle \rangle}{A\langle T \rangle} [-1] $ and $f^! N = f^*N \otimes_{A\langle T \rangle_{\sol}} f^! A $, where $f^! A \cong A\langle T \rangle[1]$.  Then,  Corollary \ref{PropsixFunctorTate} reflects the adjunction $R \iHom_{A_{\sol}}(f_! M, N) = R \iHom_{A\langle T \rangle_{\sol}}(M, f^! N)$.
\end{remark}

\section{Non-archimedean condensed functional analysis}
\label{SectionNonArch}

The main purpose of this section is to state a duality between two classes of solid vector spaces over a finite extension of $\Q_p$,  namely Fr\'echet spaces and LS spaces (to be defined below), generalising the duality between Banach spaces and Smith spaces (c.f. \cite[Theorem 3.8]{ClauseScholzeanalyticspaces}), and the duality between nuclear Fr\'echet spaces and LB spaces of compact type (see, e.g., \cite[Theorem 1.3]{SchTeitGl2}). For doing so, we will use many results on the theory of condensed non-archimedean functional analysis developed by Clausen and Scholze. Since these  results haven't yet appeared in the literature, we give a detailed account with proofs included. The reader should  be aware that most of the theory in this chapter must be  attributed to Clausen and Scholze \cite{ClausenScholzeCM}, and the only original work is that  concerning duality, principally Theorem \ref{theorem:duality}.

\subsection{Banach and Smith spaces}
\label{SubsecSolidVS}

Let $K$ and $\O_K$ be as in the previous section, and let $\varpi$ be a uniformiser of $\O_K$.  In this paragraph  we focus our attention on the category of solid $K$-vector spaces (or $K_{\sol}$-vector spaces),  i.e. the category $\SolidK$.  We start with some basic concepts 

\begin{definition} \leavevmode
\begin{enumerate}
    \item A condensed set $T$ is said to be discrete if $T=\underline{X}$ where $X$ is a discrete topological space.  

    \item A solid $\O_K$-module $M$ is said to be $\varpi$-adically complete if $M = \varprojlim_{s\in \N} M/\varpi^s M $ as solid $\O_K$-modules. 
    \end{enumerate}
    
\end{definition}

Next we introduce  Banach and Smith spaces  in the category of solid $K$-vector spaces. 

\begin{definition} \leavevmode
\begin{enumerate}

    \item A solid $K$-Banach space is a solid $K$-vector space $V$ admitting a $\varpi$-adically complete  $\O_{K,\sol}$-module $V^0\subset V$ such that: 
    \begin{enumerate}
        \item $V=V^0\otimes_{\O_{K,\sol}} K$.
        \item $V^0/\varpi^s V^0$ is discrete for all $s\in \N$. 
    \end{enumerate}
    We say that $V^0$ is a lattice of $V$.  
    
    \item  A solid $K$-Smith space  is a solid $K$-vector space $W$ admitting a profinite $\O_{K,\sol}$-submodule $W^0$ such that $W=W^0 \otimes_{\O_{K,\sol}} K=W^0[\frac{1}{\varpi}]$. We say that $W^0$ is a lattice of $W$. 
\end{enumerate}
\end{definition}

\begin{remark} Over $\R$, a Smith space (\cite[Definition  3.6]{ClauseScholzeanalyticspaces}) is a complete locally convex topological $\R$-vector space $W$ containing a compact absolutely convex subset $C\subset V$ such that $V=\bigcup_{a>0} aC$. As its $p$-adic analogue, we define a classical $K$-Smith space to be a topological $K$-vector space $W$ containing a profinite $\O_K$-module $W^0\subset W$ such that $W= W^0[\frac{1}{\varpi}]$.
\end{remark}

\begin{example}
One of the first examples of  Banach algebras over $K$ is the Tate algebra $K\langle T \rangle$,  it is the ring of  global sections of an affinoid  disc  $\bb{D}^1_{K}= \Spa(K\langle T \rangle,  \O_{K}\langle T \rangle)$. An orthonormal basis for $K\langle T \rangle$ is provided by the monomials $\{T^n\}_{n\in \N}$.   Let $K\langle S \rangle$ be another Tate algebra,  one has that $K\langle T \rangle \otimes_{K_{\sol}} K \langle S \rangle = K\langle T,S \rangle$ (cf. Lemma \ref{CorotensorBanach}) is the Tate algebra over $K$ in $2$ variables,  i.e.  the global sections of $\bb{D}^2_K= \bb{D}^1_K\times \bb{D}^1_K$. By Corollary \ref{PropsixFunctorTate}, the dual $K\langle T \rangle^{\vee}$ is isomorphic to $\frac{K \langle \langle T^{-1} \rangle \rangle}{K\langle T \rangle} \cong \prod_{n\geq 1} \O_K T^{-n} [\frac{1}{p}]$ which is a Smith space. 
\end{example}

\begin{proposition} \label{RemarkBanachSmith}
The functor $V\mapsto V(*)_{\mathrm{top}}$ induces an exact equivalence of categories between solid  and classical $K$-Banach spaces (resp.  solid  and classical $K$-Smith spaces). 
\end{proposition}

\begin{proof}
Let $V$ be a Banach space over $K$ in the classical sense, and let $V^0\subset V$ be the unit ball. Then $V^0$ is endowed with the  $\varpi$-adic topology, i.e. it is the inverse limit $V^0= \varprojlim_s V^0/\varpi^s V^0$, with $V^0/\varpi^s V^0$ discrete for all $s\in \N$. By Proposition \ref{propFunctorTopCond} we know that \[ \underline{V}= \varinjlim_{n\in \N} (\varprojlim_{s\in \N} \underline{V^0/\varpi^s V^0} ) \varpi^{-n}= \underline{V^0}\otimes_{\O_{K,\sol}} K, \] so that $\underline{V}$ is a solid $K$-Banach space. Conversely, if $V$ is a solid Banach space then $V(*)_{\mathrm{top}}$ is clearly a $K$-Banach space in the classical sense.

Let $W$ be a (classical) $K$-Smith space and $W^0\subset W$ a lattice. As $W^0$ is compact,  it is profinite and Proposition \ref{propFunctorTopCond} implies that $\underline{W}= \underline{W^0}\otimes_{\O_{K,\sol}} K$ is a solid Smith space as in our previous definition. The converse is clear.

Let $ 0 \to V' \to V \to V'' \to 0 $ be a strict exact sequence of classical Banach (resp. Smith) spaces. Since $V'$ is identified with a closed subspace of $V$ and $V''$ is identified with the topological quotient $V/V'$,  one easily verifies that 
$0 \to \underline{V}' \to \underline{V} \to \underline{V}'' \to 0$ is an exact sequence of solid $K$-vector spaces. Conversely, let $ 0 \to V' \to V \to V'' \to 0 $ be an exact sequence of solid Banach or Smith spaces.  Since the functor $T\mapsto T(*)$ from condensed abelian groups to abelian groups is exact,  one has a short exact sequence of $K$-vector spaces
\begin{equation}
\label{eqStrictExactClassic}
0 \to V'(*) \to V(*) \to V''(*) \to 0.
\end{equation}
The fact that $V' \to V$ is a quasicompact immersion (being the pullback of $0\to V''$ by $V\to V''$)  implies  that $V'(*)_{\mathrm{top}}\subset V(*)_{\mathrm{top}}$ is a closed subspace,  cf. Proposition \ref{PropQuasiseparatedCondensed}.  Finally,  since $V'' \cong V/ V'$ as condensed $K$-vector spaces,  $V''(*)_{\mathrm{top}}$ is isomorphic to the topological quotient $V(*)_{\mathrm{top}}/V'(*)_{\mathrm{top}}$ proving that \eqref{eqStrictExactClassic} is strict exact. 
\end{proof}

\begin{convention}
From now on we will refer to solid Banach and Smith  spaces   simply as  Banach and Smith spaces respectively.  The Banach and Smith spaces considered as topological spaces will be called classical Banach and Smith spaces.  
\end{convention}

\begin{proposition}
There is a natural functor
\[
\mathcal{LC}_K \to \SolidK \; : \; V \mapsto \underline{V} 
\]
from the category of complete locally convex $K$-vector spaces to solid $K$-vector spaces. Moreover,  the restriction of $\mathcal{LC}_K$ to the subcategory of compactly generated complete locally convex $K$-vector spaces is fully faithful.  
\proof
This follows from the fact that any complete locally convex $K$-vector space can be written as a cofiltered limit of classical  Banach spaces \cite[Chapter  I, \S 4]{SchneiderNFA}, and the adjunction of Proposition \ref{propFunctorTopCond}. For the fully faithfulness,  let $V_1 $ and $V_2$ be two objects in $\mathcal{LC}_K$ which are compactly generated as topological spaces,  then,  by \cite[Proposition 1.2 (2)]{ClauseScholzeanalyticspaces} the natural map $\Hom_{\Top}(V_1,  V_2 ) \to  \Hom_{\mathrm{Cond}}(\underline{V_1},  \underline{V_2})$ is a bijection. Since a continuous map $f: V_1 \to V_2$ (resp. a morphism of condensed sets $f: \underline{V_{1}}\to \underline{V_{2}}$) is a continuous $K$-linear map  (resp.  a morphism of condensed $K$-vector spaces)  if and only if it makes the obvious diagrams commute, the proposition follows.
\endproof
\end{proposition}

From now, all the complete locally convex $K$-vector spaces will be considered as solid $K$-vector spaces unless otherwise specified. The following result shows that solid Banach and Smith spaces over $K$ have orthonormal basis. 

\begin{lemma} \leavevmode \label{LemmaBanachSmith}
\begin{enumerate}
    \item A solid $K$-vector space is Banach if and only if  it is of the form \[\widehat{\bigoplus}_{i\in I} K := \big( \varprojlim_{s}(\bigoplus_{i\in I} \O_K/\varpi^s) \big) [\frac{1}{\varpi}], \] for some index set $I$. 
    
    \item  A solid $K$-vector space is Smith if and only if it is of the form \[ (\prod_{i\in I} \O_K)[\frac{1}{\varpi}], \] for some index set $I$. In particular, by \cite[Corollary 5.5]{ClausenScholzeCondensed2019}  and Theorem \ref{TheoLemmaAnalyticrings}, Smith spaces form a family of compact projective generators of the category $\SolidK$ of solid $K$-vector spaces.
\end{enumerate}
\proof
 By definition, an object of the form $\widehat{\bigoplus}_{i\in I}K$ is a solid $K$-Banach space. Conversely, let $V$ be a solid  $K$-Banach space and $V^0\subset V$ a lattice. By Proposition \ref{RemarkBanachSmith} $V(*)_{\mathrm{top}}$ is a classical  Banach space with unit ball $V^0(*)_{\mathrm{top}}\subset V(*)_{\mathrm{top}}$ and $\underline{V(*)_{\mathrm{top}}}=V$. But then $V^0(*)_{\mathrm{top}}$ has an orthonormal $\O_K$-basis by taking any lift of a $\O_K/\varpi\O_K$-basis of $V^0/\varpi$.  This proves (1).  
 
 To prove (2), let $W$ be a solid Smith space and $W^0\subset W$ a lattice. Since $W^0$ is a profinite $\O_{K,\sol}$-module,  $W^0=\underline{W^0(*)_{\mathrm{top}}}$ with  $W^0(*)_{\mathrm{top}}$  a profinite $\O_K$-module in the usual sense.  Moreover,  $W^0(*)_{\mathrm{top}}$ is $\O_K$-flat and the topological Nakayama's lemma implies that $W^0(*)_{\mathrm{top}}$ must be of the form $\prod_{i\in I} \O_K$ (\cite[{Expos\'e {$\mathrm{VII}_B$}, 0.3.8}]{SGA3}).  
\endproof
\end{lemma}

We will need the following useful proposition

\begin{prop} [\cite{ClausenScholzeCM}]
\label{PropGuidoNotes}
Let $V$ be a solid $K$-vector space. The following statements  are equivalent
\begin{enumerate}
    \item $V$ is a Smith space
    
    \item $V$ is  quasiseparated, and there is a quasicompact    $\O_K$-submodule $M\subset V$ such that  $V=M[\frac{1}{\varpi}]$. 
\end{enumerate}
Moreover, the class of Smith spaces is stable under extensions, closed subobjects and quotients by closed subobjects. 
\end{prop}

\begin{proof}
The fact that $(1)$ implies $(2)$ follows immediately from the definition. The other implication follows from the fact that any qsqc $\mathcal{O}_K$-module $M$ is profinite. Indeed, as $M$ is quasicompact there is a profinite $S$ and an epimorphism $f: \O_{K,\sol}[S]\to M$.  Since $M$ is quasiseparated,  the kernel of $f$ is a closed subspace of $\O_{K,\sol}[S]$ which is profinite.  This implies that $M\cong \O_{K,\sol}[S]/\ker f = \underline{\O_{K,\sol}[S](*)_{\mathrm{top}} /\ker f (*)_{\mathrm{top}}} $ is profinite.   

Since Smith spaces are projective (cf. Lemma \ref{LemmaBanachSmith}(2)), every extension splits and in particular they are stable under extensions. The stability under closed subobjects and quotients follows from the description of a Smith space as in $(2)$ of the equivalence.
\end{proof}

The next lemma provides a anti-equivalence between Banach and Smith spaces as solid $K$-vector spaces, c.f. \cite{Smith} (or also \cite[Theorem 3.8]{ClauseScholzeanalyticspaces}) for the analogous statement over the real or complex numbers.

\begin{lemma}[\cite{ClausenScholzeCM}] \leavevmode
\label{LemmaDuals1}
The assignment $V \mapsto V^\vee$ induces an anti-equivalence between $K$-Banach spaces and $K$-Smith spaces. More precisely, the following hold.
\begin{enumerate}
\item $\underline{\Hom}_{\O_K}(\widehat{\bigoplus}_i \mathcal{O}_K, \mathcal{O}_K) = \prod_i \mathcal{O}_K$ and $\underline{\Hom}_{K}(\widehat{\bigoplus}_{i\in I} K,K)= (\prod_{i\in I} \O_K)[\frac{1}{p}]$. 
\item $\underline{\Hom}_{\O_K}(
\prod_{i\in I} \O_K,  \O_K)= \widehat{\bigoplus}_{i\in I} \O_K$ and $\underline{\Hom}_K((\prod_{i\in I} \O_K)[\frac{1}{p}],  K)=\widehat{\bigoplus}_{i\in I} K$. 
\end{enumerate}
\end{lemma}

\begin{proof}

To prove (1), notice that 
\begin{eqnarray*}
\underline{\Hom}_{\O_K}(\widehat{\bigoplus}_i \O_K, \O_K) & = & \varprojlim_s \underline{\Hom}_{\O_K/\varpi^s}(\bigoplus_{i} \O_K/\varpi^s  , \O_K/\varpi^s) \\ 
& = & \varprojlim_s \prod_i \O_K/\varpi^s \\
& = & \prod_i \O_K.
\end{eqnarray*}
To prove the second equality it is enough to show that 
\[
\underline{\Hom}_{K}(\widehat{\bigoplus}_i K,K)= \underline{\Hom}_{\O_K}(\widehat{\bigoplus}_i \O_K, \O_K)[\frac{1}{\varpi}].
\]
Let $S$ be an extremally disconnected set, by adjunction  it is enough to show that 
\[
\Hom_{K}(\widehat{\bigoplus}_iK, \underline{\Cont}(S,K))= \Hom_{\O_K}(\widehat{\bigoplus}_i \O_K, \underline{\Cont}(S,\O_K))[\frac{1}{\varpi}].
\]
But all the solid spaces involved arise as the condensed set associated to a compactly generated  Hausdorff topological space, the claim follows from the fact that lattices are mapped to lattices for continuous maps of classical Banach spaces. 

Part (2) follows from the fact that an object of the form $\prod_{i\in I}\O_K$ is a retraction of a compact projective generator $\O_{K,\sol}[S]$ for $S$ extremally disconnected, and the fact that \[ \underline{\Hom}_{\O_K}(\O_{K,\sol}[S], \O_K)= \underline{\Cont}(S,\O_K). \] 
This finishes the proof of the Lemma.
\end{proof}

\begin{remark}
Part (2) of the previous lemma also holds with $R\iHom$ since $\prod_{i \in I} \O_K$ is a projective $\O_{K, \sol}$-module. The same proof of the first assertion of (1) can also be adapted to show that $R\underline{\Hom}_{\O_K}(\widehat{\bigoplus}_i \mathcal{O}_K, \mathcal{O}_K) = \prod_i \mathcal{O}_K$. Indeed,  
\begin{eqnarray*}
R \iHom_{\O_K}(\widehat{\bigoplus}_{i} \O_K,  \O_K) & = & R\varprojlim_{s\in \N} R\iHom_{\O_K/\varpi^s} ( \bigoplus_{i} \O_K/\varpi^s,  \O_K/\varpi^s) \\ 
& = & R\varprojlim_{s\in \N} \prod_{i} \O_K/\varpi^s \\
 & = & \prod_i \O_K.  
\end{eqnarray*}
The authors ignore how to calculate $R \iHom_K(\widehat{\bigoplus}_{i \in I} K, K)$.
\end{remark}

We now study the behaviour of the tensor product.

\begin{proposition} \label{PropTensprodSmith}
Let $V = \prod_{i \in I} \O_K$ and $W = \prod_{j \in J} \O_K$. Then
\[ V \otimes_{\O_{K, \sol}}^L W = \prod_{(i,j) \in I \times J} \O_K. \]
\end{proposition}

\begin{proof}
See \cite[Proposition 6.3]{ClausenScholzeCondensed2019}.
\end{proof}

The following useful result shows that the solid tensor product of solid $K$-Banach spaces  coincides with the projective tensor  product of  classical $K$-Banach spaces.

\begin{lemma}[{\cite{ClausenScholzeCM}}]
\label{CorotensorBanach}
Let $V$ and $V'$ be classical  Banach spaces over $K$, and let $V\widehat{\otimes}_K V'$ denote its projective  tensor product.  Then $\underline{V\widehat{\otimes}_K V'}= \underline{V} \otimes_{K_\sol} \underline{V'}$. 
\end{lemma}

\begin{proof}
Fix an isomorphism $V=\widehat{\bigoplus}_{i\in I} K$. As a convergent series $\sum_{i} a_i$ has only countably many terms different from $0$, we can write
\[
\underline{V}= \varinjlim_{I'\subset I} \widehat{\bigoplus}_{i\in I'} K,
\]
where $I'$ runs over all the countable subsets of $I$.  Therefore, we can assume that $V\cong V'\cong \widehat{\bigoplus}_{n\in \N} K$.  Let $\mathscr{S}$ denote the direct set of  functions  $f:\N \to \Z$ such that $f(n)\to +\infty$ as $n\to +\infty$, endowed with the order $f\preceq g$ iff $f(n)\geq g(n)$ for all $n\in \N$.  Thus, we can write 
\begin{equation}
\label{eqBanachColimitofSmiths}
\widehat{\bigoplus}_{n\in \N}K =  \varinjlim_{f\in \mathscr{S}} \prod_{n\in \N} \O_K \varpi^{f(n)}.
\end{equation}
Indeed, by evaluating at an extremally disconnected set $S$, $\underline{V}(S)= \Cont(S,\underline{V})=\widehat{\bigoplus}_{n\in \N} \Cont(S,K)$ has a natural (classical) Banach space structure, for which a function $\phi: S \to \underline{V}$ can be written in a unique way as a sum $\phi=\sum_n \phi_n$ with $\phi_n\in \Cont(S,K)$, such that $|\phi_n|\to 0$ as $n\to \infty$. Then, from (\ref{eqBanachColimitofSmiths}) and Proposition \ref{PropTensprodSmith} we deduce that
\[
\underline{V}\otimes_{K_\sol} \underline{V'} = \varinjlim_{f,g\in \mathscr{S}} \prod_{n\in \N}  \O_K \varpi^{f(n)} \otimes_{\O_{K,\sol}} \prod_{m\in \N} \O_K \varpi^{g(m)} = \varinjlim_{f,g\in \mathscr{S}} \prod_{n, m \in \N \times \N}  \O_K \varpi^{f(n)+g(m)}.
\]
Given $f,g\in \mathscr{S}$ define the function $h_{f,g}: \N \times \N \to \Z$  as $h_{f,g}(n,m) = f(n) + g(m)$.  Let $\mathscr{S}'$ be the direct set of functions $h: \N\times \N \to \Z$ such that $h(n,m)\to +\infty$ as $\max\{n,m\} \to +\infty$. Then the set $\{h_{f,g}\}_{f,g \in \mathscr{S}}$ is a cofinal family in $\mathscr{S}'$. Indeed,  given $h:  \N\times \N \to \Z$, if we define $f(n)=\frac{1}{2}\min_{m} h(n,m)$ and $g(m)=\frac{1}{2} \min_{n} h(n,m)$, then $h \preceq h_{f,g}$.    Therefore, 
\[
 \varinjlim_{f,g\in \mathscr{S}} \prod_{n, m \in \N \times \N}  \O_K \varpi^{h_{f,g}(n, m)} = \varinjlim_{h\in \mathscr{S}'} \prod_{n,m} \O_K \varpi^{h(n,m)}= \widehat{\bigoplus}_{n,m}K. 
\]
\end{proof}

Let us recall the concept of a nuclear solid $K$-vector space.

\begin{definition}[{\cite[Definition 13.10]{ClauseScholzeanalyticspaces}}]
Let $V \in \SolidK$. We say that $V$ is nuclear if, for all extremally disconnected  set $S$, the natural map 
\[
 \underline{\Hom}_K(K_{\sol}[S], K)\otimes_{K_{\sol}} V\to \underline{\Hom}_K(K_{\sol}[S], V)
\]
is an isomorphism.
\end{definition}

\begin{remark}
We warn the reader that this notion of nuclearity differs from the classical one, say in \cite[\S 19]{SchneiderNFA}. Indeed,   if a classical Banach space is nuclear in the classical sense then it is finite dimensional (cf. \textit{loc. cit.} \S 19).  On the other hand,  solid $K$-Banach spaces are always nuclear in the condensed sense.  
\end{remark}

\begin{corollary}[\cite{ClausenScholzeCM}]
\label{CoroBanachisNuclear}
Let $V$ be a Banach space over $K$, then $V$ is a nuclear $K_{\sol}$-vector space.
\end{corollary}

\begin{proof}
Let $S$ be an extremally disconnected set. By taking basis we can write  $K_{\sol}[S]= \prod_J \O_K[\frac{1}{\varpi}]$ and $V= \widehat{\bigoplus}_{I} K$.  Then
\begin{eqnarray*}
\iHom_{K}(K_{\sol}[S],  V) & = & \iHom_{K}(K_{\sol}[S]\otimes_{K_{\sol}} V^{\vee},  K) \\
            & = & \iHom_{K}\bigg( \prod_{J\times I} \n{O}_K [\frac{1}{\varpi}],  K \bigg) \\ 
            & = & \widehat{\bigoplus}_{J\times I} K \\
            & = & \iHom_{K}(K_{\sol}[S],K) \otimes_{K_{\sol}} V,
\end{eqnarray*}
where the first and third equalities  follow from the duality between Banach and Smith spaces of Lemma \ref{LemmaDuals1} and the $\otimes$-$\Hom$ adjunction, the second equality follows from Proposition \ref{PropTensprodSmith}, and the last equality is a consequence of Lemma \ref{CorotensorBanach}. 
\end{proof}

\begin{corollary} \label{CoroHomBanachSmith}
Let $V$ and $W$ be a Banach and a Smith space respectively.   Then
\[ \iHom_K(V, W) = V^\vee \otimes_{K_\sol} W, \]
\[ \iHom_K(W, V) = W^\vee \otimes_{K_\sol} V. \]
\end{corollary}

\begin{proof}
The second equality follows from nuclearity of $V$ (Corollary \ref{CoroBanachisNuclear}). For the first equality, tensor-Hom adjunction gives
\[ \iHom_K(V, W) = \iHom_K(V \otimes_{K_\sol} W^\vee, K), \]
and the results follows immediately from the description of the tensor product of two Banach spaces (Lemma \ref{CorotensorBanach}) and duality between Banach and Smith spaces (Lemma \ref{LemmaDuals1}). 
\end{proof}

We finish this section with some  elementary lemmas that will be needed later.

\begin{lemma} \leavevmode
The functor $V \mapsto V^{qs}$ defines a left adjoint to the inclusion of the category of quasiseparated solid $K$-vector spaces into the category of all solid $K$-vector spaces.
\end{lemma}

\begin{proof}

By Lemma \ref{LemmaQuasiseparatedQuotient} we only need to show that  $V^{qs}$ is a solid $K$-vector space.  Let $S$ be an extremally disconected set,  we have  a commutative diagram with exact arrows 
\begin{equation}
\label{eqDiagramQsSolid}
\begin{tikzcd}
 \Hom_{K}(K[S],  V ) \ar[r] &  \Hom_K(K[S],  V^{qs}) \ar[r] &   0 \\
 \Hom_{K}(K_{\sol}[S],  V ) \ar[r] \ar[u, "\wr"] &  \Hom_K(K_{\sol}[S],  V^{qs})  \ar[u] & 
\end{tikzcd}
\end{equation}
where the left vertical arrow is an isomorphism since $V$ is solid. This implies that $\Hom_{K}(K_{\sol}[S],  V^{qs}) \to \Hom_K(K[S],  V^{qs})$ is surjective.   To prove injectivity,  it is enough to show that $(K_{\sol}[S]/K[S])^{qs}=0$,  for this,  it is enough to show that $(K[S])(*)_{\mathrm{top}}$ is dense in $(K_{\sol}[S])(*)_{\mathrm{top}}$ which is is clear by definition of $K_{\sol}[S]$.   Therefore, the right vertical arrow of  \eqref{eqDiagramQsSolid} is a bijection,  proving that $V^{qs}$ is a solid $K$-vector space. 
\end{proof}

We say that a  map of quasiseparated condensed sets has dense image if the map of underlying topological spaces does.  One has the following lemma 

\begin{lemma} \leavevmode
\label{Lemmanoqs}
\begin{enumerate}  \item Let $V$   be a  solid $K$-vector space such that the maximal quasiseparated quotient  $V^{qs}$ is zero. Then $\underline{\Hom}(V,K)=0$. 

\item  A map of Banach spaces $V\to V'$ is injective (resp. with dense image) if and only if its dual $V'^{\vee}\to V^{\vee}$ has dense image (resp. is injective). 
\end{enumerate}
\proof
Let $S$ be a extremally disconnected set, then 
\[
\underline{\Hom}_K(V,K)(S)=\Hom_K(V\otimes K_{\sol}[S], K)=\Hom_K(V,\underline{\Cont}(S,K)).
\]
But $\underline{\Cont}(S,K)$ is a Banach space. Then,  by adjunction (Proposition \ref{propFunctorTopCond}),  we get \[\underline{\Hom}_K(V,K)(S)=\Hom_K(V(*)_{\mathrm{top}}, \Cont(S,K)).\]
Since $V^{qs}=0$,  the maximal Hausdorff quotient of $V(*)_{\mathrm{top}}$ is zero.  This implies that  \[\Hom_K(V(*)_{\mathrm{top}}, \Cont(S,K))=0\] proving (1). 

To prove (2), let $f: V\to V'$ be a map of Banach spaces. Suppose that $f$ has not dense image and let $\overline{f(V)}\subset V'$ be the closure of its image (i.e.  the solid $K$-Banach space associated to the closure of $f(V(*)_{\mathrm{top}})$ in. $V'(*)_{\mathrm{top}}$). Then $V'/\overline{f(V)}$ is a non zero Banach space and we have a short exact sequence 
\[
0\to \overline{f(V)}\to V'\to V'/\overline{f(V)}\to 0
\]
which splits as any Banach space over a local field is orthonormalizable.  Taking duals we get a short exact sequence 
\[
0 \to (V'/\overline{f(V)})^\vee \to V'^{\vee} \to \overline{f(V)}^\vee \to 0.
\]
Since $f^{\vee}: V'^\vee\to V^{\vee}$ factors through $\overline{f(V)}^{\vee}$, the map $f^{\vee}$ is not injective.  Conversely, suppose that the map $f^{\vee}$ is not injective, then its kernel $\ker f^{\vee}$ is a closed subspace of $V'^{\vee}$  which is a Smith space by Proposition \ref{PropGuidoNotes}. Since the quotient $V'^\vee / \ker(f^\vee)$ is also a Smith space, there is a retraction $r:V'^{\vee}\to \ker(f^\vee)$.  Taking duals one sees that the composition $V\to V' \to \ker(f^{\vee})^{\vee}$ is zero and that the last map is surjective (because of $r$), this implies that $f$ has not dense image.  

Finally, suppose that $f: V\to V'$ is injective. If $f^{\vee}: V'^{\vee}\to V^{\vee}$ does not have dense image, the closure  $\overline{f(V'^{\vee})}\subset V^{\vee}$ is a closed Smith subspace and its quotient $V^{\vee}/\overline{f(V'^{\vee})}$ is a non zero Smith space.  Taking duals we get a short exact   sequence \[
0\to (V^{\vee}/\overline{f(V'^{\vee})})^{\vee}\to V \to \overline{f(V'^{\vee})}^{\vee}\to 0. 
\]
But $f: V\to V'$ factors through $\overline{f(V'^{\vee})}^{\vee}$, this is a contradiction with the injectivity of $f$.  Conversely, Suppose that $f^{\vee}: V'^{\vee} \to V^{\vee}$ has dense image, consider the quotient 
\[
0 \to V'^{\vee} \to V^{\vee}\to Q \to 0.
\]
Taking duals one obtains an exact sequence 
\[
0 \to \underline{\Hom}_K(Q,K)\to V \to V'.
\]
But part (1) implies that $\underline{\Hom}_K(Q,K)=0$, proving that $f$ is injective. 
\endproof
\end{lemma}

\subsection{Quasiseparated solid $K$-vector spaces}

We shall use the following results  throughout  the text without explicit mention. They are due to Clausen and Scholze, and explained to us by Guido Bosco in the study group of La Tourette.

\begin{proposition}[\cite{ClausenScholzeCM}]
\label{Propqs}
 Let $V$ be a solid $K$-vector space. The following are equivalent, 
\begin{enumerate}
    \item $V$ is quasiseparated.
    
    \item $V$ is equal to the filtered colimit of its Smith subspaces. 
\end{enumerate}
\proof
Let $V$ be a quasiseparated $K_{\sol}$-vector space, let $W_1,W_2$ be  Smith subspaces of  $V$.   As $V$ is quasiseparated, $W_1\cap W_2$ is a closed Smith subspace of $W_1$, and the sum $W_1+W_2\subset V$ is isomorphic to $(W_1\oplus W_2)/W_1\cap W_2$.  This shows that the Smith subspaces of $V$ form a direct system, let $V_0$ denote their colimit.  We claim  that $V/V_0=0$, let $W'$ be a Smith space and $f: W' \to V/V_0$ a map of solid $K_\sol$-vector spaces. As $W'$ is projective, there is a lift $f': W'\to V$.  But $\ker f'\subset W'$ is a closed Smith subspace since $V$ is quasiseparated. This implies that $f'$ factors through $V_0$ and that $f=0$.  Since the Smith spaces form a family of compact projective generators of $K_\sol$-vector spaces, one must have $V/V_0=0$ proving $(1)\Rightarrow (2)$.

Conversely, let $V=\varinjlim_{i\in I} V_i$ be a vector space written as a filtered colimit of Smith spaces by injective transition maps. Let $S_1$, $S_2$ be two profinite sets and $f_j:S_j\to V$ be two maps for $j=1,2$.  As the $S_i$ are profinite, there exists $i\in I$  such that $f_j$ factors through $V_i$ for $j=1,2$. Then, as the map  $V_i\to V$ is injective, one has 
\[
S_1\times_{V} S_2= S_1\times_{V_i} S_2.
\]
The implication $(2)\Rightarrow (1)$ follows as a Smith space is quasiseparated. 
\endproof 
\end{proposition}

\begin{lemma}[\cite{ClausenScholzeCM}]
\label{LemmaqsimpliesFlat}
A quasiseparated $K$-solid space is flat. In other words, if $V$ is a quasiseparated $K_{\sol}$-vector space, then $- \otimes_{K_\sol} V = - \otimes_{K_\sol}^L V$. 
\proof
Let $V$ be a quasiseparated $K_{\sol}$-vector space. Since filtered colimits are exact in the category of condensed abelian groups, and the solid tensor product commutes with colimits, by Proposition \ref{Propqs} it is enough to prove the lemma for $V=\prod_{I}\O_K[\frac{1}{p}]$ a Smith space. Let $W \in \SolidK$. We want to show that $V \otimes_{K_\sol}^L W$ is concentrated in degree zero. As the Smith spaces are compact projective generators, $W$ can be written as a quotient  $0\to W'' \to W' \to W\to 0$ where $W'$ is a direct sum of Smith spaces. Then we are reduced to showing that $0\to W''\otimes_{K_\sol}V\to W' \otimes_{K_\sol} V$ is injective.  Since $W''$ is quasiseparated, by Proposition \ref{Propqs}, it can be written as filtered colimit of its Smith subspaces. Therefore, by compacity of Smith spaces,  the arrow $W''\to W'$ is a filtered colimit of injections of Smith spaces.  It is hence enough to show that if $\prod_{J_1}\O_K[\frac{1}{p}] \to \prod_{J_2}\O_K[\frac{1}{p}]$ is an injective map, then 
\[
\prod_{J_1}\O_K[\frac{1}{p}] \otimes_{K_\sol} V \to \prod_{J_2}\O_K[\frac{1}{p}] \otimes_{K_\sol}V
\]
is injective. This follows from the tensor product of two Smith Spaces (Proposition \ref{PropTensprodSmith}) 
\[
(\prod_{I} \O_K [\frac{1}{p}])\otimes_{K_\sol} (\prod_{J}\O_K[\frac{1}{p}])= \prod_{I\times J} \O_K [\frac{1}{p}]. \]
This finishes the proof.
\endproof
\end{lemma}

\subsection{Fr\'echet and $LS$ spaces}

Our next goal is to extend  the duality between Banach and Smith spaces to a larger class of solid $K$-vector spaces.  We need a definition 

\begin{definition} \leavevmode
\begin{enumerate} 
    \item  A solid Fr\'echet space   is a solid $K$-vector space which can be written as a countable cofiltered limit of Banach spaces. 
    
    \item A solid $LS$ (resp. $LB$, resp. $LF$) space  is a solid $K$-vector space which can be written as a countable filtered colimit of Smith (resp. Banach, resp. solid Fr\'echet) spaces with injective transition maps. 
\end{enumerate}
\end{definition}

\begin{example}
Let $\mathring{\mathbb{D}}^1_K = \bigcup_{r < 1} \mathbb{D}^1_K(r)$ be the rigid analytic open unit disc written as the increasing union of the closed affinoid discs of radius $r < 1$. The global sections $\mathcal{O}(\mathring{\mathbb{D}}^1_K) = \varprojlim_{r < 1} \mathcal{O}(\mathbb{D}^1_K(r))$ is a Fr\'echet space and we will see that its dual $\mathcal{O}(\mathring{\mathbb{D}}^1_K)^\vee = \varinjlim_{r < 1} \mathcal{O}(\mathbb{D}^1_K(r))^\vee$ is a solid $LS$ space (cf. Theorem \ref{theorem:duality}). The tensor product $\mathcal{O}(\mathring{\mathbb{D}}^1_K) \otimes_{K_\sol} \mathcal{O}(\mathring{\mathbb{D}}^1_K)$ is isomorphic to $\mathcal{O}(\mathring{\mathbb{D}}^2_K)$ (cf. Lemma \ref{LemmaTensorProduct}).
\end{example}

\begin{lemma} \leavevmode
\begin{enumerate}
    \item  The functor $V \mapsto V(*)_{\mathrm{top}}$ induces an equivalence of categories between solid and classical Fr\'echet spaces such that  $V = \underline{V(*)_{\mathrm{top}}}$.
    
    \item An $LS$ space is quasiseparated. Conversely, a quasiseparated $K_{\sol}$-vector space $W$ is an $LS$ space if and only if it is countably compactly generated, i.e.  for every surjection $\bigoplus_{i\in I} P_i \to W$  by direct sums of Smith spaces, there is a countable index subset $I_{0}\subset I$ such that $\bigoplus_{i\in I_0} P_{i}\to W$ is surjective. 
\end{enumerate}
\proof
Part (1) follows from the fact that a classical  Fr\'echet space is complete for a countable family of seminorms (i.e. it can be written as a countable cofiltered limit of  classical Banach spaces), Proposition \ref{propFunctorTopCond} (1), and Proposition \ref{RemarkBanachSmith}.  For part (2), the fact that an $LS$ space is quasiseparated follows from Proposition \ref{Propqs}.  Let $W$ be quasiseparated $K_{\sol}$-vector space. Assume it is countably compactly generated. Write $W= \varinjlim_{W'\subset W} W'$ as a the colimit of its Smith subspaces. As $W$ is quasiseparated,  the sum of two Smith subspaces is Smith, so the colimit is filtered. By hypothesis, there are countably many $W'$ such that $W= \varinjlim_{s \in \N} W_s'$. Moreover, we can assume that $W_0\subset W_1\subset \cdots$.  This proves  that $W$ is an $LS$ space. Conversely,  let $W$ be an $LS$ space and let $\bigoplus_{i \in I} P_i \to W$ be a surjective map with $P_i$ Smith. The image $P_i'$ of $P_i$ in $W$ is a Smith space since $W$ is quasiseparated, hence $W= \sum_{i} P_i'$. Thus, without  loss of generality we can assume that $\varinjlim_{i\in I} P_i'$ is filtered with injective transition maps  and equal to $W$.    Let $W= \varinjlim_{s\in \N} W_s$ be a presentation as a countable colimit of Smith spaces by injective transition maps.  By compactness of the Smith spaces, for all $s$ there exists  $i_s$ such that $W_n\subset P'_{i_s}\subset W$.  We can assume that $P'_{i_s}\subset P'_{i_{s+1}}$ for all $s\in \N$.  Thus, $\bigoplus_{s\in \N} P_{i_s}\to W$ is surjective, this finishes the lemma. 
\endproof
\end{lemma}

The following lemma says that we can  always choose a presentation of a solid Fr\'echet space as an inverse limit of Banach spaces with dense transition maps. 

\begin{lemma} \label{lemmaFrechetdense}
Let $V$ be a solid Fr\'echet space, then we can write $V=\varprojlim_{n\in \N} V_n$ with $V_n$  Banach spaces such that $V(*)_{\mathrm{top}} \to V_n(*)_{\mathrm{top}}$ has dense image for all $n\in \N$.  Conversely, let $\{V_n\}_{n\in \N}$ be a cofiltered limit of Banach spaces such that $V_{n+1}(*)_{\mathrm{top}} \to V_n(*)_{\mathrm{top}}$ has dense image,  and let $V=\varprojlim V_n$ be its inverse limit.  Then $V(*)_{\mathrm{top}} \to V_n(*)_{\mathrm{top}}$ has dense image for all $n\in \N$.
\proof
Let $V=\varprojlim_{n\in \N} V_n$ be a presentation of the solid Fr\'echet space as an inverse limit of Banach spaces.  Changing $V_n$ by the closure of the image of $V$ (i.e. the solid Banach space corresponding to the closure of the image of $V(*)_{\mathrm{top}}$ in $V_n(*)_{\mathrm{top}}$), we obtain a desired presentation with $V(*)_{\mathrm{top}} \to V_n(*)_{\mathrm{top}}$ of dense image.  Conversely, let $\{V_n\}_{n\in \N}$ be an inverse system of Banach spaces with maps $V_{n+1}(*)_{\mathrm{top}}\to V_n(*)_{\mathrm{top}}$ of dense image,  and let $V=\varprojlim_n V_n$ be a solid Fr\'echet space, we want to show that the image of $V(*)_{\mathrm{top}} \to V_n(*)_{\mathrm{top}}$ is dense for all $n\in \N$. Fix $n_0\in \N$, let $w\in V_{n_0}(*)_{\mathrm{top}}$ and $0 < \epsilon < 1$.  Let  $|\cdot |_n$ denote the norm of $V_n(*)_{\mathrm{top}}$. Without loss of generality we assume that $|\cdot |_n \leq |\cdot|_{n+1}$ on $V_{n+1}(*)_{\mathrm{top}}$, where here we identify $|\cdot|_n$ with the seminorm on $V_{n+1}(*)_{\mathrm{top}}$ defined by composing with the map $V_{n+1}(*)_{\mathrm{top}} \to V_{n}(*)_{\mathrm{top}}$.  By density of the transition maps $\phi^{n+1}_n: V_{n+1}(*)_{\mathrm{top}}\to V_{n}(*)_{\mathrm{top}}$, there exists $v_{n_0+1}\in V_{n+1}(*)_{\mathrm{top}}$ such that $|\phi^{n_0+1}_{n_0}(v_{n_0+1})-w|\leq \epsilon$.  By induction, for all $n\geq n_0+1$ we can find $v_n\in V_{n}(*)_{\mathrm{top}}$ such that $|\phi^{n+1}_n(v_{n+1}-v_n)|\leq \epsilon^{n}$.  Let $n\geq n_0+1$ be fixed and let $k\geq 0$,  then  by construction the sequence $\{\phi^{n+k}_{n}(v_{n+k})\}$ converges in $V_{n}(*)_{\mathrm{top}}$ to an element $v'_n$.  Moreover, it is immediate to check that $\phi_{n}^{n+1}(v_{n+1}')=v_n'$ so that $v'=(v_n')\in V(*)_{\mathrm{top}}$,  and $|\phi^{n_0+1}_{n_0}(v_{n_0+1}'-w)|\leq \epsilon$. This proves the lemma. 
\endproof
\end{lemma}

\begin{convention}
From now on we refer  to solid Fr\'echet spaces simply   as Fr\'echet spaces,  we call the underlying topological space a classical Fr\'echet space.  
\end{convention}

\subsection{Properties of Fr\'echet spaces} We now  present some basic properties of Fr\'echet spaces, most of the results in the context of condensed mathematics  are due to Clausen and Scholze \cite{ClausenScholzeCM}. 

\begin{lemma}[Topological Mittag-Leffler \cite{ClausenScholzeCM}]
\label{LemmaML}
Let $V=\varprojlim_n V_n$ be a  Fr\'echet space written as an inverse limit of Banach spaces with dense transition maps. Then
\[ R^j \varprojlim_n V_n = 0 \]
for all $j > 0$. In particular, $V= R\underline{\Hom}_K(V^\vee, K)$. 
\end{lemma}

\begin{proof}
See \cite[Lemma A.18]{GuidoDrinfeld}.
\end{proof}

\begin{lemma}[\cite{ClausenScholzeCM}]  \label{LemmaTensorProduct}
Let $(V_n)_{n\in \N}$ and $(W_{m})_{m\in \N}$ be countable families of Banach spaces.

\begin{enumerate} 

\item We have 
\[
\big(\prod_{n} V_n\big)\otimes_{K_\sol}^L \big(\prod_{m} W_m \big) = \prod_{n,m} V_n\otimes_{K_\sol} W_m. 
\]

\item More generally, if $V =\varprojlim_n V_n$ and $W= \varprojlim_m W_m$ are Fr\'echet spaces written as inverse limits of Banach spaces by dense transition maps (which is always possible thanks to Lemma \ref{lemmaFrechetdense}), one has 
\[
V\otimes_{K_\sol}^L W = \varprojlim_{n, m} V_n \otimes_{K_\sol} W_m. 
\]
\end{enumerate}
\end{lemma}

\begin{proof}
Property (AB6) of Theorem \ref{TheoScholzeAB} and  Proposition \ref{Propqs} imply that products of quasiseparated solid $K$-vector spaces are quasiseparated. Then, by Lemma \ref{LemmaqsimpliesFlat},  all the derived tensor products in the statements are already concentrated in degree $0$. By Lemma \ref{LemmaML} we have a short exact sequences 
\begin{gather*}
0\to \varprojlim_n V_n \to \prod_{n} V_n \to \prod_n V_n \to 0 \\
0\to \varprojlim_m W_m \to \prod_{m} W_m \to \prod_m W_m \to 0 .
\end{gather*}
Then (2) follows from (1) by taking the tensor product of the above sequences.  

For (1), suppose that the statement is true for all $V_n$ and $W_n$ possessing a countable orthonormal basis.  Let $V_n \cong \widehat{\bigoplus}_{I_n} K $ and $W_m \cong \widehat{\bigoplus}_{J_m} K$ for all $n,m\in N$.  We write $V_n=\varinjlim_{I'_n\subset I_n} \widehat{\bigoplus}_{I'_n} K$ and $W_m=\varinjlim_{J'_m\subset J_m} \widehat{\bigoplus}_{J'_m} K$ with $I'_n$  and $J'_m$ running among all the countable subsets of $I_n$ and $J_m$ respectively. Then 
\begin{eqnarray*}
(\prod_n V_n)  \otimes_{K_\sol} (\prod_m W_m) & = & \bigg( \prod_n(\varinjlim_{I'_n\subset I_n} \widehat{\bigoplus}_{I'_n} K) \bigg) \otimes_{K_\sol} \bigg( \prod_m(\varinjlim_{J'_m\subset J_m} \widehat{\bigoplus}_{J'_m} K)\bigg) \\
         & = & \varinjlim_{\substack{\forall (n,m)\in \N\times \N \\ I'_n\times J'_m \subset I_n\times J_m}} \bigg( \prod_n (\widehat{\bigoplus}_{I'_n} K)\bigg) \otimes_{K_\sol}\bigg(\prod_m(\widehat{\bigoplus}_{J'_m} K) \bigg)  \\
         & = & \varinjlim_{\substack{\forall (n,m)\in \N\times \N \\ I'_n\times J'_m \subset I_n\times J_m}} \prod_{n,m}( \widehat{\bigoplus}_{I_n'\times J_m'} K ) \\ 
         & = & \prod_{n,m}(\varinjlim_{I'_n\times J'_m\subset I_n\times J_m} \widehat{\bigoplus}_{I_n'\times J_m'} K   ) \\
         & = & \prod_{n,m}(V_n\otimes_{K_\sol} W_m).
\end{eqnarray*}
Hence, we are left to prove (1) for $W_m=V_n= \widehat{\bigoplus}_{\N} K$ for all $n,m\in \N$. Let $\mathscr{S}$ be the filtered set of functions $f: \N \to \Z$ such that $f(k)\to +\infty$ as $k\to +\infty$. For all $n,m\in \N$ we can write 
\begin{gather*}
V_n= \widehat{\bigoplus}_{\N}K = \varinjlim_{f_n\in  \mathscr{S}} \prod_{k\in \N} \O_K \varpi^{f_n(k)}\\
W_m= \widehat{\bigoplus}_{\N}K = \varinjlim_{g_m\in  \mathscr{S}} \prod_{s\in \N} \O_K \varpi^{g_m(s)}. 
\end{gather*}
Therefore we get 
\begin{eqnarray*}
(\prod_{n}V_n) \otimes_{K_\sol}( \prod_m W_m) & = & \bigg( \varinjlim_{\forall n,\; f_n\in \mathscr{S}} \prod_n  \prod_k \O_K \varpi^{f_n(k)} \bigg) \otimes_{K_\sol}  \bigg( \varinjlim_{\forall m, \; g_m\in \mathscr{S}} \prod_m  \prod_s \O_K \varpi^{g_m(s)} \bigg) \\
            & = &  \varinjlim_{\substack{\forall (n,m)  \\f_n,g_m\in \mathscr{S}}} \bigg( \prod_{n,m}  (\prod_{k}\O_K \varpi^{f_n(k)})\otimes_{K_{\sol}} (\prod_s \O_K \varpi^{g_m(s)}) \bigg).
\end{eqnarray*}
Given $f,g\in \mathscr{S}$ we  define $h_{f,g}: \N \times \N \to \Z$ as $h_{f,g}(k,s)=  f(k)+ g(s)$. Let $\mathscr{S}'$ be the set of functions $h: \N\times \N \to \Z$ such that $h(n,m)\to \infty$ as $\min\{n,m\}\to \infty$.   Then the family of functions $\{h_{f,g}\}_{f,g\in \mathscr{S}}$  is cofinal in  $\mathscr{S}'$ (see the proof of Lemma \ref{CorotensorBanach}).  One obtains  
\begin{eqnarray*}
(\prod_{n}V_n) \otimes_{K_\sol}( \prod_m W_m) & = & \varinjlim_{\substack{\forall (n,m)  \\h_{n,m}\in \mathscr{S}' }} \prod_{n,m} \prod_{k,s} \O_K \varpi^{h_{n,m}(k,s)} \\ 
& =& \prod_{n,m} \bigg( \varinjlim_{h_{n,m}\in \mathscr{S}'}   \prod_{k,s} \O_K \varpi^{h_{n,m}(k,s)} \bigg) \\
& = & \prod_{n,m} \widehat{\bigoplus}_{k,s} K \\ & = & \prod_{n,m}( V_n\otimes_{K_\sol} W_m),
\end{eqnarray*}
this finishes the proof. 
\end{proof}

\begin{prop}[\cite{ClausenScholzeCM}] \label{PropFrechetNuclear}
  A Fr\'echet space is a nuclear $K_\sol$-vector space. 
\proof
Let $V=\varprojlim_n V_n$ be a Fr\'echet space written as an inverse limit of Banach spaces with dense transition maps. Let $S$ be an extremally disconnected set, then by Corollary \ref{CoroBanachisNuclear}
\begin{eqnarray*}
\underline{\Hom}_{K_\sol}(K_\sol[S], V) & = & \varprojlim_n \underline{\Hom}_{K_\sol}(K_\sol[S],  V_n) \\
            & = &  \varprojlim_n ( \underline{\Hom}_{K_\sol}(K_{\sol}[S],K )\otimes_{K_\sol} V_n) \\
            & = & \underline{\Hom}_{K_\sol}(K_{\sol}[S],K) \otimes_{K_\sol} (\varprojlim_n V_n )  \\ 
            & = & \underline{\Hom}_{K_\sol}(K_{\sol}[S],K) \otimes_{K_\sol} V,
\end{eqnarray*}
this finishes the proof. 
\endproof
\end{prop}

\begin{remark} \label{remarkFrechet}
Let $S$ be an extremally disconnected set and $T$ a condensed set,  by definition   $T(S)$ can be written as $\underline{\Cont}(S, T)(*)$,  in particular,  it carries a natural compactly generated  topology by Proposition \ref{propFunctorTopCond}. If $V = \varprojlim V_n$ is a solid Fr\'echet space written as an inverse limit of Banach spaces $V_n$, then $ V(S) = \varprojlim_n V_n(S)$.  By Proposition \ref{PropFrechetNuclear} we have that
\[ V_n(S) = \iHom_{K}(K_{\sol}[S],V_n)(*)= (\iHom_{K}(K_{\sol}[S],K)\otimes_{K_{\sol}} V_n) (*), \]
which shows that the topological spaces $V_n(S)_{\mathrm{top}}$ are naturally classical Banach spaces and hence that $V(S)_{\mathrm{top}}$ is a classical Fr\'echet space.

\end{remark}


The following two lemmas describe the maps between $LF$, Fr\'echet and Banach spaces

\begin{lemma}
\label{lemmaFrechetToNormed}
Let $V=\varprojlim_n V_n$ be a Fr\'echet space written as a countable cofiltered limit of Banach spaces with   projection  maps $V\rightarrow V_n$ of dense image. Let $W$ be a Banach space. Then any continuous linear map $f: V\rightarrow W$ factors through some $V_n$.  More generally, we have $\underline{\Hom}_K(V,W)= \varinjlim_n \underline{\Hom}_K(V_n,W)$.
\proof
First, evaluating $\underline{\Hom}_K(V,W)$ at an extremally disconnected set, using adjunction and the nuclearity of $W$, one reduces to showing that $\Hom_K(V,W)=\varinjlim_n \Hom_K(V_n,W)$. Since all the spaces involved come from compactly generated  topological $K$-vector spaces, we might assume that $V$ and $W$ are classical Fr\'echet and Banach spaces respectively. Let $|\cdot|_n$ be the seminorm of $V$ given by $V_n$, without loss of generality we may assume that $|\cdot|_{n}\leq |\cdot |_{n+1}$. We denote the norm of $W$ by $||\cdot ||$. The map $f$ factors through $V_n$ if and only if it is continuous with respect to the seminorm $|\cdot|_n$. Suppose that $f$ does not factor through any $n$, then there exist  sequences of vectors $(v_{n,m})_m$ in $V$ for all $n$ such that 
\[
|v_{n,m}|_n\xrightarrow{m\to \infty} 0 \mbox{ and } ||f(v_{n,m})||\geq 1 \;\; \forall m. 
\]
Moreover, we may assume that $|v_{n,n}|_n<\frac{1}{n}$. Then the sequence $(v_{n,n})_n$ converges to $0$ in $V$ but $||f(v_{n,n})||\geq 1$ for all $n$, which is a contradiction with the continuity of $f$. 
\endproof
\end{lemma}

\begin{lemma}
\label{LemmaLFmaps}
Let $W= \varinjlim_n W_n $ and $W'= \varinjlim_m W'_m$ be $LF$  spaces presented as a filtered colimit of Fr\'echet spaces by injective transition maps.  Then 
\[
\underline{\Hom}_K(W, W')= \varprojlim_n \varinjlim_m \underline{\Hom}_K(W_n, W'_m). 
\]
\proof
First observe that formally
\[ \underline{\Hom}_K(W, W')= \varprojlim_n \underline{\Hom}_K(W_n, W'), \]
 so we can assume that $W= W_0$ is a Fr\'echet space. Let $S$ be an extremally disconnected set, then by nuclearity of $W'$
\[ \underline{\Hom}_K(W, \varinjlim_m W'_m) (S) = \Hom_K(W \otimes_{K} K[S], \varinjlim_m W'_m) = \Hom_K(W, \varinjlim_m W'_m \otimes_{K_\sol} \underline{\Cont}(S, K)), \]
which shows that one can reduce to proving
\[ \Hom(W , \varinjlim_m W'_m) = \varinjlim_m \Hom(W, W'_m). \]
So let
\[ W \to \varinjlim_m W'_m \]
be a map of solid $K$-vector spaces.     Evaluating at an extremally disconnected set $S$, by Remark \ref{remarkFrechet} we get a map
\[ f : W(S)_{\mathrm{top}} \to \varinjlim_m W'_m(S)_{\mathrm{top}}  \]
between a (classical) Fr\'echet space and a (classical) $LF$ space. We claim that this map factors through some $m$. Indeed, this follows from \cite[Corollary 8.9]{SchneiderNFA}, but we also give a direct argument. Assume not. Then there exists a sequence $(x_m)_{m \geq 1}$ in $W(S)_{\mathrm{top}} $ such that $f(x_m) \notin W'_m(S)_{\mathrm{top}} $. Multiplying $x_m$ by big powers of $p$, we can assume $x_m \to 0$ as $m \to +\infty$. This translates into the existence of a map
\[ \N \cup \{ \infty \} \to \varinjlim_m W'_m(S)_{\mathrm{top}} , \;\;\; m \mapsto f(m), \infty \mapsto 0\]
from the profinite set $\N \cup \{ \infty\}$ into an $LF$-space. Since $\N\cup\{\infty\}$ is profinite,  this maps must factorise through some $m$, which is a contradiction. This shows that, for each extremally disconnected set $S$, there exists a smallest $n(S) \in \N$ such that the map
\[ W(S)_{\mathrm{top}}  \to \varinjlim_m W'_m(S)_{\mathrm{top}}  \]
factors through $W(S)_{\mathrm{top}}  \to W'_{n(S)}(S)_{\mathrm{top}}  \to \varinjlim_m W'_m(S)_{\mathrm{top}} $. We conclude the proof by showing that the $n(S)'s$ are uniformly bounded. We argue again by contradiction. Assume that there are extremally disconnected sets  $S_1, S_2, \hdots$ such that $n(S_i) \to +\infty$ as $i \to +\infty$. Let $S = \prod_i S_i$, which is a profinite set. Let $\widetilde{S}$ be an extremally disconnected set surjecting to $S$. Let $i \in \N$ be such that $n(S_i) > n(\widetilde{S})$ and let $S_i \to \widetilde{S}$ be a section of the surjection $\widetilde{S} \to S \to S_i$. Then the map $W(S_i)_{\mathrm{top}}  \to \varinjlim_m W'_m(S_i)_{\mathrm{top}} $ factors through
\[ W(S_i)_{\mathrm{top}}  \to W(\widetilde{S})_{\mathrm{top}}  \to W'_{n(\widetilde{S})}(\widetilde{S})_{\mathrm{top}}  \to W'_{n(\widetilde{S})}(S_i)_{\mathrm{top}} , \]
which is a contradiction. This finishes the proof.
\endproof
\end{lemma}

\subsection{Spaces of compact type}

Before proving the duality between Fr\'echet and $LS$ spaces let us recall the definition of (classical) nuclear Fr\'echet space and a $LB$  space of compact type.  Due to the fact that Fr\'echet spaces are always nuclear in the world of solid $K$-vector spaces (Proposition \ref{PropFrechetNuclear}), we will say that a classical nuclear Fr\'echet space is a  Fr\'echet space of compact type.  We recall the definition of trace class maps and compact maps for $K_{\sol}$-vector spaces.

\begin{definition}[{\cite[Definition 13.11]{ClauseScholzeanalyticspaces}}] \leavevmode
\begin{enumerate}
    \item  A trace class map of Smith spaces is a $K$-linear map $f:Q_1\to Q_2$ such that there is a map $g: K\to Q_1^{\vee}\otimes_K Q_2$ such that $f$ is the composition 
    $Q_1\xrightarrow{1\otimes g} Q_1\otimes Q_1^{\vee}\otimes Q_2 \to Q_2$. 
    
    \item  A map of Banach spaces is compact if its dual is a trace class map. 
\end{enumerate}
\end{definition}

\begin{definition} \leavevmode
\label{defiBanachificationSmith}
\begin{enumerate}
    \item A  Fr\'echet space $V$ is of compact type if it has a presentation $V=\varprojlim_n V_n$ as an inverse limit of Banach spaces where the maps $V_{n+1}\to V_n$ are compact. 
    
    \item A $LS$ space is of compact type if it admits a presentation $W=\varinjlim_{n\in \N} W_n$ by injective trace class maps of Smith spaces. 
    
    \item Let $W$ be a Smith space with lattice $W^{0}\subset W$. We denote by $W^{B}$ the Banach space whose underlying space is
    \[
    W^{B}=   \varprojlim_{s\in\N} (\underline{W^0 (*)_{\mathrm{dis}}}/\varpi^{s} )[\frac{1}{\varpi}]. 
    \]
    In other words, $W^{B}$ is the Banach space  attached to the underlying set $W(*)$ with unit ball $W^0(*)$.  Note that there is a natural injective map with dense image $W^{B}\to W$. 
    
    \item Let $V$ be a Banach space, we denote $V^{S}=( V^{\vee, B})^{\vee}$ and call this space the ``Smith completion of $V$''. Note that there is an injective map  with dense image $V\to V^{S}$. 
\end{enumerate}
\end{definition}

\begin{remark}
Let $W$ be a Smith space,  the formation of $W^{B}$ is independent of the lattice $W^{0}\subset W$,  namely,  if $W^{'0}$ is a second lattice  we can find $s\in \N$ big enough such that $\varpi^s W^0\subset W^{'0} \subset \varpi^{-s} W^0$.   On the other hand, note that $W^B(*)=W(*)$ as sets but not as topological spaces,  this follows from the fact that $W^0(*)_{\mathrm{top}}$ is $\varpi$-adically complete.  
\end{remark}

The following two results will be very useful later when studying spaces of compact type. The reader can compare them to the ideas appearing in \cite[\S 16]{SchneiderNFA} (see, e.g., the discussion after \cite[Proposition 16.5]{SchneiderNFA}). 

\begin{lemma} \leavevmode
\label{LemmaBanafication}
\begin{enumerate} \item A map of Smith spaces $W\to W'$ is  trace class if and only if it factors as $W\to W'^{B}\to W'$. Dually, a map of Banach spaces $V\to V'$ is of compact type if and only of it can be extended to  $V\to V^{S}\to V'$. 

\item Let $f: V\to W$ be a map from a Banach space to a Smith space, then $f$ extends uniquely  to a commutative diagram 
\[
\begin{tikzcd}
V \ar[r] \ar[d] & V^{S}\ar[d] \\ W^{B}\ar[r] & W.
\end{tikzcd}
\]

\end{enumerate}
\proof
(1) The  factorization for a morphism of Banach spaces follows from the one of Smith spaces. Let $f: W\to W'$ be a map of Smith spaces and suppose that it factors through $W'^{B}$. Then $f$ belongs to $\Hom(W, W'^{B})= W^{\vee}\otimes_{K_{\sol}} W'^{B}(*)$ as $W'^{B}$ is Banach. This shows that $f$ if a trace class map.  Conversely, let $f: W\to W'$ be trace class, then there is $g: K \to W^{\vee} \otimes_{K_{\sol}} W'$ such that $f$ factors as a composition $W\xrightarrow{1\otimes g} W \otimes_{K_{\sol}} W^{\vee} \otimes_{K_{\sol}} W' \rightarrow W'$.  Since $K$ is $\varpi$-adically complete, the map $g$ factors through the $\varpi$-completion of $(W^{\vee } \otimes_{K_{\sol}} W')_{dis}$ (i.e. the condensed object associated to the the $\varpi$-adic completion of the underlying discrete objects),  which we denote by  $(W^{\vee } \otimes_{K_{\sol}} W')^{B}$.  We claim that  $(W^{\vee } \otimes_{K_{\sol}} W')^{B}= W^{\vee}\otimes_{K_{\sol}} W^{'B}$,  this shows part $(1)$ since $f$ factors as 
\[
W \xrightarrow{1\otimes g} W \otimes_{K_{\sol}}  W^{\vee}\otimes_{K_{\sol}}  W^{'B} \to W^{'B} \to W'.
\]

Let us prove the claim,  by piking  basis we can write $W^{\vee}= \widehat{\bigoplus}_{I} K$ and $ W^{'}=\prod_{J} \O_{K} [\frac{1}{\varpi}]$.  One computes that $ W^{\vee}\otimes_{K_{\sol}} W^{'} = \widehat{\bigoplus}_{I} \prod_{J}\O_K [\frac{1}{\varpi}]$,  so that $(W^{\vee}\otimes W^{'})^{B} = \widehat{\bigoplus}_{I}(\prod_{J}\O_K)^{B} [\frac{1}{\varpi}]$,  where $(\prod_{J}\O_K)^{B}$ is the $\varpi$-adic completion of the discrete space $(\prod_{J}\O_K)_{dis}$. On the other hand,  one has 
\begin{eqnarray*}
W^{\vee}\otimes_{K_{\sol}} W^{'B}  & = & ( \widehat{\bigoplus}_{I} \O_K \otimes_{\O_{K, \sol}} (\prod_{J}\O_K)^{B} )[\frac{1}{\varpi} ] \\ 
                                    & = & \varprojlim_{s} ({\bigoplus}_I  \O_K/\varpi_{s} \otimes_{\O_{K}} (\prod_{J}\O_K/\varpi^s)_{dis})[\frac{1}{\varpi}] \\
                                    & = &  \widehat{\bigoplus}_{I}(\prod_{J}\O_K)^{B} [\frac{1}{\varpi}].  
\end{eqnarray*}

(2) Let $W^0$ be a lattice of $W$. Since $W^0(*)_{\mathrm{top}}$ is open in $W(*)_{\mathrm{top}}$,  there is a lattice $V^0\subset V$ such that $f(V^0) \subset W^0$.  Moreover,  $W^0(*)_{\mathrm{top}}$ is $\varpi$-adically complete,  so that the restriction of $f$ to $V^0$ factors through $V^0\to  \varprojlim_{s} (\underline{W^0(*)_{dis}}/\varpi^s) \to W^0$.  One gets the factorization  $V\to W^{B}\to W$ by inverting $\varpi$.  Taking duals we see that $f^{\vee}: W^{\vee}\to V^{\vee}$ factors as $W^{\vee}\to V^{\vee,B} \to V$, taking duals again one gets the factorization $f: V\to V^{S}\to W$. The uniqueness   follows from the density of $V_{\mathrm{top}}$ in $V^S(*)_{\mathrm{top}}$ (resp.  $W^{B}(*)_{\mathrm{top}}$ in $W(*)_{\mathrm{top}}$) and the fact that all the condensed sets involved are quasi-separated.   
\endproof
\end{lemma}

\begin{remark}
Part (2) of the previous lemma shows that $W^B$ is the final object in the category of all Banach spaces $V$ equipped with an arrow $V \to W$  and  morphisms given by maps   of $K_{\sol}$-vector spaces making the obvious diagram commutes. Dually, $V^S$ is the initial object in the category of all Smith spaces $W$ equipped with an arrow  $V \to W$ and  morphisms given by maps commuting with the obvious diagram.
\end{remark}

\begin{corollary} \leavevmode
\label{coroFLScompact}
\begin{enumerate}
    \item Let $V=\varprojlim_n V_n$ be a Fr\'echet space of compact type.  Then we can write $V=\varprojlim_{n\in N} V_n^{S}$ as an inverse limit of Smith spaces with trace class transition maps. Conversely, any such vector space is a Fr\'echet space of compact type. 
    
    \item Let $W=\varinjlim_n W_n$ be an $LS$ space of compact type, then $W= \varinjlim_n W_n^{B}$ can be written  as  a filtered colimit of Banach spaces with injective compact maps. Conversely, a colimit of Banach spaces by injective compact maps is an $LS$ space of compact type.  In particular,  being  a $LB$ or a  $LS$ space of compact type is equivalent. 
\end{enumerate}

\end{corollary}

\begin{lemma} \label{LemmaSES}
Let
\begin{equation} \label{Eqses}
0 \to V' \to V \to V'' \to 0
\end{equation}
be a short exact sequence of Fr\'echet (resp. $LS$) spaces. Then \eqref{Eqses} can be written as an inverse limit (resp. direct limit) of short exact sequences of Banach (resp. Smith) spaces.
\end{lemma}

\begin{proof}
Let suppose that \eqref{Eqses} is a short exact sequence of $LS$ spaces, let $n\in \N$ and write $V''= \varinjlim_{n} V''_{n}$ as a colimit of Smith spaces by injective transition maps.   The pullback of \eqref{Eqses} by $V''_{n} \to V''$ is a short exact sequence 
\[
0\to V' \to V\times_{V''} V''_{n} \to V''_{n} \to 0
\]
where $V\times_{V''} V''_{n}$ is still an $LS$ space. In fact,  by writting $V= \varinjlim_{n} \widetilde{V}_n$ as a colimit of Smith spaces one has $V\times_{V''} V''_n = \varinjlim_{n} \widetilde{V}_{n} \times_{V''} V''_{n}$ and each $\widetilde{V}_{n}\times_{V''} V''_{n}$ is a Smith space.  By compacity of $V''_{n}$,  there exits a Smith subspace $V_{n}\subset V\times_{V''} V''_{n}$ such that the induced arrow $V_{n} \to V''_{n}$ is surjective. By Proposition \ref{PropGuidoNotes} the kernel $V'_{n}:= \ker (V_{n}\to V''_{n})$ is a Smith subspace of $V'$.  Moreover,  we can take $V_n$ such that $V=\varinjlim_{n} V_n$. Then \eqref{Eqses} can be written as the colimit of the short exact sequences 
\[
\varinjlim_{n} [0 \to V'_{n} \to V_{n} \to V_{n}'' \to 0]. 
\]

For the Fr\'echet case,  it is enough to show that  the map $V \to V'' \to 0$ can be written as an inverse limit of surjective maps of Banach spaces $V_{n}\to V''_{n}\to 0$.  In fact, by setting $V'_n:= \ker (V_{n} \to V''_{n})$ one has $V'= \varprojlim_n V_{n}'$.   Let us write $V=\varprojlim_{n} \widetilde{V}_n$ and  $V'' = \varprojlim_{n} V_{n}''$  as inverse limits  of Banach spaces with dense transition maps.  By Lemma \ref{lemmaFrechetToNormed} the composition $V\to V'\to V'_n$ factors through a map $V\to \widetilde{V}_{n}\to V_n'$ for some $n \gg 0$,  in particular, the image of $\widetilde{V}_n(*)_{\mathrm{top}}$ in $V'_n(*)_{\mathrm{top}}$ is dense.  Let $V_{n}\subset \widetilde{V}_n\oplus V'_{n}$ be the solid $K$-Banach space   associated to the closure of the image of $V(*)$  in  $\widetilde{V}_n(*)_{\mathrm{top}}\oplus V_{n}'(*)_{\mathrm{top}}$.  Then, $V \to V_n$ has dense image by construction and $V= \varprojlim_{n} V_n$.  Now,  notice that $(\widetilde{V}_n \oplus 0) + V_n = \widetilde{V}_n \oplus V'_n$,  namely,  the first sum is a  dense   subspace of the direct sum by construction of $V_n$, but it is also a closed subspace (the sum of two closed subspaces of a $K$-Banach space is closed, as  it is checked by using well chosen orthonormal basis). This shows that the map $V_n\to V_n'$ is surjective as wanted.  
\end{proof}

\subsection{Duality}

We conclude with the main result of this chapter.

\begin{theorem} \leavevmode \label{theorem:duality}
\begin{enumerate}

\item  The functor $V \mapsto \underline{\Hom}_K(V,K)$ induces an exact antiequivalence between Fr\'echet and LS spaces such that $\underline{\Hom}_K(V,V')=\underline{\Hom}_K(V'^{\vee},V^{\vee})$, extending the one between Banach and Smith spaces. Moreover, $V$ is Fr\'echet of compact type if and only if $V^{\vee}$ is an $LS$ space of compact type.  
    
\item Let $V=\varprojlim_{n\in N} V_n$ be a Fr\'echet space and $W= \varinjlim_{n\in \N} W_n$ an LS space. Then
\[
\underline{\Hom}_K(W,V)= W^{\vee}\otimes_{K_{\sol}} V \mbox{ and } \underline{\Hom}_{K}(V,W)= V^{\vee} \otimes_{K_{\sol}} W.
\]
    In particular, if $V$ and $V'$ are Fr\'echet spaces (resp. $LS$ spaces) then \[ (V\otimes_{K_{\sol}} V')^{\vee}= V^{\vee}\otimes_{K_{\sol}} V'^{\vee}. \] 
    
\end{enumerate}
\proof
(1) Let $V$ be a Fr\'echet space and let $V=\varprojlim_n V_n$ be a presentation as an inverse limit of Banach spaces  with  transition maps of dense image. Let $S$ by an extremally disconnected set, we want to compute 
\[
\underline{\Hom}_K(V,K)(S)=\Hom_K(V \otimes_{\O_{K,\sol}} \O_{K,\sol}[S], K  )= \Hom_K(\varprojlim_n V_n, \underline{\Cont}(S, K)).
\]
By Lemma \ref{lemmaFrechetToNormed} we have 
\[
\Hom_K(\varprojlim_n V_n, \underline{\Cont}(S,K))= \varinjlim_n \Hom_K(V_n, \underline{\Cont}(S,K))= \varinjlim_n \Hom_K(V_n\otimes_{\O_{K,\sol}} \O_{K,\sol}[S],K).
\]
In other words, we have a natural isomorphism 
\[
\underline{\Hom}_K(V, K) = \varinjlim_n V_n^{\vee}.
\]
By Lemma \ref{Lemmanoqs} (2),  the transition maps $V_{n}^{\vee}\to V_{n+1}^{\vee}$ are injective, proving that $V^{\vee}$ is an $LS$ space.

Conversely, let $W$ be an $LS$ space and $W=\varinjlim_n W_n$ a presentation as a colimit of Smith spaces with injective transition maps. Then if follows formally that 
\[
\underline{\Hom}_K(W,K) = \varprojlim_n  W_n^{\vee}.
\]
By Lemma \ref{Lemmanoqs} (2) again, the transition maps $W_{n+1}^{\vee}\to W_{n}^{\vee}$ have dense image.  It is clear from the construction that $(V^{\vee})^{\vee}=V$ and $(W^{\vee})^{\vee}=W$ for $V$ Fr\'echet and $W$ an $LS$ space.   This implies  formally that $\Hom_K(V,V')= \Hom_K(V^{'\vee}, V^\vee)$, which gives  the antiequivalence between Fr\'echet and $LS$ spaces.  Finally, by Corollary \ref{coroFLScompact} and the previous computation, the duality restricts to Fr\'echet and $LS$ spaces of compact type. 

We now extend the equality of homomorphisms to the internal $\iHom$.  Let $V$ and $V'$ be Fr\'echet spaces  and  $S$  an extremally disconnected set.  We have   
\begin{eqnarray*}
\underline{\Hom}_K(V,V')(S) & = & \Hom_K(V\otimes_{K_\sol} K_\sol[S],  V' )\\ 
& =&  \Hom_K(V, \underline{\Cont}(S,K)\otimes_{K_\sol} V') \\ 
& =  &  \Hom_K(( \underline{\Cont}(S,K)\otimes_{K_\sol} V')^{\vee}, V^{\vee}) \\
& = & \Hom_K(V^{'\vee}\otimes_{K_\sol} K_\sol[S], V^\vee) \\
& =& \underline{\Hom}_K(V^{'\vee}, V^{\vee})(S).
\end{eqnarray*}
The second equality follows from adjunction and nuclearity of $V'$.
The third equality is the duality between Fr\'echet and LS spaces.   The fourth equality  follows from the compatibility of the tensor product and duality between Fr\'echet and LS spaces of part (2).

To prove exactness of the functor $V\mapsto V^{\vee}$, it is enough to see that it sends short exact sequences
\begin{equation}
\label{eqShortseqFreLS}
0 \to V' \xrightarrow{f} V \xrightarrow{f} V''\to 0
\end{equation} 
of Fr\'echet spaces (resp.  $LS$ spaces) to short exact sequences of $LS$ spaces (resp. Fr\'echet spaces).  Note that we are taking exact sequences in the category of solid $K$-vector spaces, in particular,  $V''(*)_{\mathrm{top}}$ is the topological quotient  $V(*)_{\mathrm{top}}/ V'(*)_{\mathrm{top}}$ and the sequence 
\[
0 \to V'(*)_{\mathrm{top}} \xrightarrow{f(*)} V(*)_{\mathrm{top}} \xrightarrow{g(*)} V''(*)_{\mathrm{top}}\to 0
\]
is strict exact in  classical  terms,  i.e, it is exact as a sequence of $K$-vector spaces and the map $V'(*)_{\mathrm{top}} \to \ker(g(*))$ is an isomorphism of topological $K$-vector spaces. We first proof the result for Smith and Banach spaces. Given a short exact sequence  \eqref{eqShortseqFreLS} of Smith spaces, by projectivity of $V''$ one has a section $s:  V''\to V$ splitting the short exact sequence, and hence the claim follows. Dually,  if \eqref{eqShortseqFreLS} is a short exact sequence of Banach spaces, one can construct a section $s: V''\to V$ by taking orthonormalizable  basis. We finally reduce the general case to Smith and Banach spaces to conclude the proof of the statement. Indeed, if \[ 0 \to V' \to V \to V'' \to 0 \]
is an exact sequence of $LS$ spaces, the
by Lemma \ref{LemmaSES} it can be written as a direct limit
\[ \varinjlim_{n \in \N} [ 0 \to V_n' \to V_n \to V_n'' \to 0 ] \]
of short exact sequences of Smith spaces. Taking duals and using exactness for the case of Smith spaces, we get an inverse limit of short exact sequences
\[ 0 \to (V_n'')^\vee \to V_n^\vee \to (V_n'')^\vee \to 0. \]
Since $R^1\varprojlim_{n} (W'_{n})^{\vee}=0$ by topological Mittag-Leffler,  one has a short exact sequence 
\[
0\to (V')^{\vee} \to V^{\vee} \to (V'')^{\vee} \to 0
\]
of Fr\'echet spaces as desired. The reduction for exact sequences of Fr\'echet spaces to the case of Banach spaces is deduced similarly.

(2) Let $V=\varprojlim_n V_n$ and $W=\varinjlim_m W_m$ be a Fr\'echet and an $LS$ space respectively.   We can write 
\begin{eqnarray*}
\underline{\Hom}_K(W,V) & = & \varprojlim_m \underline{\Hom}_K(W_m, V) \\ 
& = & \varprojlim_m W_m^{\vee}\otimes_{K_{\sol}} V \\
& = & W^{\vee}\otimes_{K_{\sol}} V
\end{eqnarray*}
where the first equality is formal, the second equality follows from nuclearity of Fr\'echet spaces, and the third equality from the tensor product of Fr\'echet spaces (Proposition \ref{LemmaTensorProduct}).  Dually, we have
\begin{eqnarray*}
\underline{\Hom}_K(V,W) & = &  \underline{\Hom}_K(V\otimes_{K_{\sol}}W^{\vee}, K  ) \\
 & = & \underline{\Hom}_K(\varprojlim_{n, m} V_n\otimes_{K_{\sol} } W_m^{\vee}, K) \\
 & = &  \varinjlim_{n, m} \underline{\Hom}_K( V_n\otimes_{K_{\sol}} W_m^{\vee}, K) \\
 & = & \varinjlim_{n, m} V_n^{\vee} \otimes_{K_{\sol}} W_m  \\
 & = & V^{\vee} \otimes_{K_{\sol}} W, 
\end{eqnarray*}
where the first equality follows from self duality of $LS$ spaces and the tensor-Hom adjunction, the second equality from the tensor product of Fr\'echet spaces, the third equality from Lemma \ref{lemmaFrechetToNormed}, and the last two equalities from Corollary \ref{CoroHomBanachSmith} and the commutativity between tensor product and colimits. 
\endproof
\end{theorem}

We conclude this chapter with a conjecture which is a derived enhancement of Lemma \ref{lemmaFrechetToNormed}. This will only be used in Proposition \ref{PropUseConjecture} and the reader is invited to skip it in a first lecture. 

\begin{conj}
\label{ConjectureLocallyanalytic}
Let $V=\varprojlim_n V_n$ be a Fr\'echet space of compact type written as a limit of Banach spaces with compact and dense transition maps, then
\[
R\underline{\Hom}_K(V, K)= \underset{n}{\hocolim}\; R\underline{\Hom}_{K}(V_n,K).
\]

\end{conj}

\begin{cor}
\label{CoroDualityDerived}
Suppose that Conjecture \ref{ConjectureLocallyanalytic} is true. Let $V=\varprojlim_n V_n$ be a  Fr\'echet space of compact type and $B$ a Banach space,  then
\begin{enumerate}
    \item $R\underline{\Hom}_K(V,B)= \hocolim_n R\underline{\Hom}_K(V_n,B) $. 
    
    \item We have 
    \[ R\underline{\Hom}_K(V,B)= V^{\vee}\otimes_{K_{\sol}} B. \] In particular $V^{\vee}$ is the derived dual of $V$. 
    \item The duality Theorem \ref{theorem:duality} gives a derived duality between bounded complexes of Fr\'echet spaces of compact type and bounded complexes of $LS$ spaces of compact type.  
    \item Let $C$ be a bounded complex of Fr\'echet spaces of compact type and $D$ a bounded complex of $LS$ spaces of compact type. Then \[ R\underline{\Hom}_K(C,D)= R\underline{\Hom}_K(C,K)\otimes_{K_{\sol}}^{L} D \]
    and
    \[ R\underline{\Hom}_K(D,C)= R\underline{\Hom}_K(D,K)\otimes_{K_{\sol}}^{L} C. \]  
    
    \item  Let $C$ and $C'$ be bounded complexes of Fr\'echet spaces of compact type. Then \[ R\underline{\Hom}_K(C,C')= R\underline{\Hom}_K( C^{'\vee}, C^{\vee}). \]
    
\end{enumerate}
\proof
Part (1) follows from Conjecture \ref{ConjectureLocallyanalytic} using the fact that $K_{\sol}[S]\otimes_{K_{\sol}}^L K_{\sol}[S']= K_{\sol}[S\times S']$ for $S$ and $S'$ profinite sets, that 
\[
R\underline{\Hom}_{K}(V, K)(S)= R\Hom(V\otimes_{K_{\sol}} K_{\sol}[S], K)= R\Hom(V, \underline{\Cont}(S,K))
\]
and that  any Banach space over $K$ is a direct sumand of a space of the form $\underline{\Cont}(S,K)$ for $S$ a profinite set.  

If $V$ is Fr\'echet of compact type, we can write $V=\varprojlim_n V_n^{S}$ where $V_n^{S}$ is the Smith completion of $V_n$ (cf. Corollary \ref{coroFLScompact}).  Then 
\begin{eqnarray*}
R\underline{\Hom}_K(V,B) & = &  \underset{n}{\hocolim}\;R\underline{\Hom}_K(V_n, B) \\ 
            & = & \underset{n}{\hocolim}\; R\underline{\Hom}_K(V_n^{S}, B) \\
            & = & \underset{n}{\hocolim}\; (V_n^{S})^{\vee} {\otimes_{K_{\sol}}} B \\ 
            & = & \varinjlim_n V_n^{\vee} {\otimes_{K_{\sol}}} B = V^{\vee}\otimes_{K_{\sol}} B. 
\end{eqnarray*}

Part (3) follows immediately from part (2) and an easy induction via the stupid truncation. For part (4), notice that if $C$ is a bounded Fr\'echet complex of compact type, and $D$ is a bounded  $LS$ complex of compact  type, then 
\begin{eqnarray*}
R\underline{\Hom}_{K}(C,D) & = &  R\underline{\Hom}_K(C\otimes_{K_{\sol}}^L D^{\vee}, K) \\ 
  & = &   C^{\vee}\otimes_{K_{\sol}}^L D
\end{eqnarray*}
since $C\otimes_{K_{\sol}}^{L} D^{\vee}$ is a bounded  Fr\'echet  complex of compact type.  Similarly one shows $R\underline{\Hom}_K(D,C)= D^{\vee}\otimes_{K}^L C$.

Finally, by a devisage using the stupid filtration, (5) is reduced to showing that if $V$ and $V'$ are Fr\'echet spaces of compact type then 
\[
R\underline{\Hom}_K(V,V')= R\underline{\Hom}_{K}(V^{'\vee}, V^{\vee}).
\]
Writing $V'= \varprojlim_n V_n'$ one gets, by (2),
\[
R\underline{\Hom}_K(V, V')= R\varprojlim_n R\underline{\Hom}_K(V, V_n')= R\varprojlim_n (V^{\vee}\otimes_{K_{\sol}} V_n').  
\]
Dually, we have that 
\begin{eqnarray*}
R\underline{\Hom}_{K}(V^{'\vee}, V^{\vee} ) & = & R\varprojlim_n R\underline{\Hom}_K(V_n^{',\vee},  V^{\vee}) \\
            & = & R\varprojlim_n (V_{n}' \otimes_{K_{\sol}} V^{\vee})
\end{eqnarray*}
which gives the desired equality. 
\endproof
\end{cor}

\section{Representation Theory}
\label{SectionRepTheory}
Let $G$ be a compact $p$-adic Lie group.  In this section we translate the theory of analytic and locally analytic representations of $G$ from the classical framework to condensed mathematics. We hope that this new point of view could simplify some proofs and provide a better understanding of the theory.   Our main sources of inspiration are the works of Lazard \cite{Lazard},  Schneider-Teitelbaum \cite{SchTeitGl2,SchTeitDist} and Emerton \cite{Emerton}.

We begin with the introduction of various algebras of distributions, each one serves  a particular purpose in the theory.  Namely, there are distribution algebras arising from affinoid groups which appear naturally in the analytification functor of Corollary  \ref{PropsixFunctorTate},  distribution algebras which are localizations of the Iwasawa algebra,  and  distribution algebras algebras attached to Stein analytic groups relating the previous two.    

Next, we introduce the category of solid $G$-modules, it is a generalisation of the category of non-archimedean topological spaces endowed with a continuous action of $G$ (Banach, Fr\'echet, $LB$, $LF$, ...).  

We continue with the definition of analytic representations for the affinoid and Stein analytic groups of Definition \ref{defAffinoidgroups}.  We recall how the analytic vectors for a Banach representation are defined, and how it serves as motivation for the derived analytic vectors of solid $G$-modules. We prove the main Theorem \ref{TheoMain} which, roughly speaking, says that being analytic for a Stein analytic group as in Definition \ref{defAffinoidgroups} is the same  as being a module over its distribution algebra. 

We end with an application to locally analytic representations, reproving a theorem of Schneider-Teitelbaum describing a duality between   locally analytic representations on $LB$ spaces of compact type, and Fr\'echet modules of compact type over the algebra of locally analytic distributions, cf. Proposition \ref{PropBanachlocally}.  Under the assumption of Conjecture \ref{ConjectureLocallyanalytic},  We  state a   generalisation to a duality between locally analytic  bounded complexes of $LB$ compact type   and bounded Fr\'echet complexes  of compact type endowed with an action of the algebra of locally analytic distributions.

\subsection{Function spaces and distribution algebras}
\label{subsecFuncDist}

In the following paragraph  we define different classes of  spaces of functions and distributions that will be used throughout this text. These are algebras already appearing in the literature (\cite[\S 4]{SchTeitDist}, \cite[\S 5]{Emerton}) that we introduce in a way adapted to our interests and purposes.

\subsubsection{Rigid analytic neighbourhoods}
Let $G$ be a compact $p$-adic Lie group of dimension $d$.

\begin{definition}
The Iwasawa algebra of $G$ is defined as the solid ring
\[ \O_{K, \blacksquare}[G] = \varprojlim_{H \unlhd G }\O_K[G/H] \in \SolidOK,
\]
\[K_\blacksquare[G] = \O_{K, \blacksquare}[G][1/p] = (\varprojlim_{H \unlhd G} \O_K[G/H])[1/p] \in \SolidK, \]
where $H$ runs over all the open and normal subgroups of $G$. 
\end{definition}

\begin{remark}
Classically, the Iwasawa algebra is denoted by $\Lambda(G, \O_K)$ and $\Lambda(G, K)$. We decide to adopt the new notations from the theory of condensed mathematics to best fit our results in this language.
\end{remark}

\begin{remark}
According to the notations of Example \ref{Exampleanalyticrings}, the Iwasawa algebra corresponds precisely to the evaluation at $G$ of the functor of measures of the analytic rings $(\O_K, \O_K)_\blacksquare$ and $(K, \O_K)_\blacksquare$ which, in turn, are commonly denoted by $\O_{K, \blacksquare}$ and $K_\blacksquare$ respectively. 
\end{remark}

In order to define the space of  locally analytic functions attached to $G$, we need to work with coordinates locally around the identity.   By \cite[Corollary 8.34]{Analyticprop}, there exists an open normal subgroup $G_0$ of $G$ which is a uniform pro-$p$ group \footnote{Recall that a uniform pro-$p$ group $H$ is a pro-$p$ group which is finitely generated, torsion free and powerful, i.e., $[H, H] \subseteq H^p$ if $p > 2$ or $[H, H] \subseteq H^4$ if $p = 2$.}. Such a group can be equipped with a $p$-valuation $w$ and an ordered basis $g_1,\ldots, g_d\in G_0$ in the sense of \cite[III.2]{Lazard} (or cf. \cite[\S 4]{SchTeitDist}).  By \cite[Proposition III.3.1.3]{Lazard},  after shrinking $G_0$ if necessary,  we can assume that  $w(g_1) = \hdots = w(g_d) = w_0 > 1$ is an integer.  This basis induces  charts
\[ \phi : \Z_p^d \to G_0 , \;\;\; (x_1, \hdots, x_d)  \mapsto g_1^{x_1} \cdot \hdots \cdot g_d^{x_d} \]
such that $\phi$ is an homeomorphism between $G_0$ and $\Z_p^d$ with $w(g_1^{x_1} \cdot \hdots \cdot g_d^{x_d}) = w_0 + \min_{1 \leq i \leq d} v_p(x_i)$,   and such that the map $\psi: G_0 \times G_0 \to G_0$, $(g, h) \mapsto g h^{-1}$, defining the group structure of $G_0$, is given by an analytic function $\psi: \Z_p^d \times \Z_p^d \to \Z_p^d$ with coefficients in $\Z_p$. In other words, the function $\psi: (x,y)\mapsto \phi^{-1}(\phi(x) \phi(y)^{-1})$  can be written as a tuple of power series with coefficients in $\Z_p$ converging to $0$.      By further shrinking $G_0$ if necessary we can also assume that conjugation on $G_0$ by any element of $g\in G$ is given by  a family of  power series with bounded coefficients. 

Let $r=p^{-s}>0$ for $s\in \Q$,   and $\bb{D}^d_{\Q_p}(r)\subset \bb{A}^d_{\Q_p}$ the affinoid polydisc of radius $r$.  If $s\in \Z$ then $\bb{D}^d_{\Q_p}(r)$ is  the affinoid space defined by the algebra $\Q_p\langle \frac{T_1}{p^s}, \ldots \frac{T_d}{p^s} \rangle $, where $T_1,\ldots, T_d$ are  the coordinates of $\bb{A}^d_{\Q_p}$.   We let $\mathring{\bb{D}}^d(r)= \bigcup_{r'<r} \bb{D}^d(r')$ denote the open polydisc of radius $r$.  

\begin{definition}
\label{defAffinoidgroups}
We define the following rigid analytic groups,  see Lemma \ref{LemmaNormalG0} below
\begin{enumerate}
    \item $\mathbb{G}_0 = (\mathbb{D}_\qp^d(1), \psi)$;  the affinoid group defined by the group law $\psi$ of $G_0$.   
    
    \item For any $h \in \Q_{\geq 0}$, the affinoid groups   $\bb{G}_{h}= (\mathbb{D}_{\qp}^{d}(p^{-h}),  \psi)$ of radius $p^{-h}$.   We also denote  $\mathbb{G}_0^{(h)} = G_0 \bb{G}_{h}\subset \bb{G}_0$.

    \item For any $h\in \Q_{\geq 0}$ the Stein groups $\bb{G}_{h^+}=\cup_{h' > h} \mathbb{G}_{h'}$ and   $\mathbb{G}_0^{(h^+)} = G_0 \bb{G}_{h^+} =\cup_{h' > h} \mathbb{G}_0^{(h)}$.  
\end{enumerate}
\end{definition}

\begin{example} \label{exampleGL2}
If $G=\GL_2(\Z_p)$, an example of a uniform pro-p-subgroup $G_0$ is 
\[
G_0= \left( \begin{array}{cc} 1+ p^n \Z_p & p^n\Z_p \\p^n \Z_p & 1+p^n \Z_p \end{array} \right)
\]
for $n\geq 2$ if $p=2$ and $n\geq 1$ if $p>2$. Let's take $p>2$ and $n=1$.    In this case $\bb{G}_{h}$ is the rigid analytic group 
\[
\bb{G}_h=\left( \begin{array}{cc}
    1+ \bb{D}^1_{\Q_p}(p^{-h-1}) & \bb{D}^1_{\Q_p}(p^{-h-1})   \\
    \bb{D}^1_{\Q_p}(p^{-h-1})  &1+\bb{D}^1_{\Q_p}(p^{-h-1}) 
\end{array} \right) = \left( \begin{array}{cc}
    1+ p^{h+1}\bb{D}^1_{\Q_p}(1) & p^{h+1}\bb{D}^1_{\Q_p}(1)   \\
    p^{h+1}\bb{D}^1_{\Q_p}(1)  &1+p^{h+1}\bb{D}^1_{\Q_p}(1)
\end{array} \right),
\]
whereas the Stein group $\bb{G}_{h^+}$ is equal to 
\[
\bb{G}_{h^+}=\left( \begin{array}{cc}
    1+ p^{h+1}\mathring{\bb{D}}^1_{\Q_p}(1) & p^{h+1}\mathring{\bb{D}}^1_{\Q_p}(1)   \\
    p^{h+1}\mathring{\bb{D}}^1_{\Q_p}(1)  &1+p^{h+1}\mathring{\bb{D}}^1_{\Q_p}(1)
\end{array} \right).
\]
\end{example}


The following lemma says that the rigid spaces of   Definition \ref{defAffinoidgroups} are indeed rigid analytic groups.

\begin{lemma}
\label{LemmaNormalG0}
The  affinoid $\mathbb{G}_{h}$ is an open normal subgroup of $\mathbb{G}_0$ stable by  conjugation of  $G$.
\end{lemma}

\begin{proof}
By simplicity we suppose that $h\in \bb{N}$.  The map
\[ \psi: \mathbb{G}_0 \times \mathbb{G}_0 \rightarrow \mathbb{G}_{0} \;\; :\;\; (x,y)\mapsto xy^{-1} \] is defined by a family of power series $(Q_1(X,Y),\ldots, Q_d(X,Y))$ with integral coefficients satisfying the group axioms. In particular, one has $Q_i(0,0)=0$. The inclusion $\mathbb{G}_h \rightarrow \mathbb{G}_0$ is given by the map $\Q_p\langle  T_1, \ldots, T_d\rangle \rightarrow \Q_p\langle \frac{T_1}{p^h},\ldots, \frac{T_d}{p^h} \rangle $. Thus, the image of $\frac{T_i}{p^h}$ by the multiplication map $\psi$ is equal to
\[ \frac{1}{p^h}Q_i(X,Y)= \frac{1}{p^h}\sum_{(\alpha,\beta)\neq 0} a_{\alpha,\beta} X^{\alpha}Y^{\beta}= \sum_{(\alpha,\beta)\neq 0} a_{\alpha,\beta} p^{h(|\alpha|+|\beta|-1)} \left(\frac{X}{p^h} \right)^{\alpha} \left(\frac{Y}{p^h}\right)^{\beta}. \]  This shows that $\psi$ restricts to a map $\mathbb{G}_h \times \mathbb{G}_h \rightarrow \mathbb{G}_h$,  proving that $\mathbb{G}_h$ is a subgroup of $\mathbb{G}_0$. A similar argument shows  that $\mathbb{G}_h$ is normal in $\bb{G}_0$ and that it is stable by the conjugation of elements of $G$.  
\end{proof}

\subsubsection{Analytic distributions} Classically, the analytic vectors of Banach representations are defined via the  affinoid algebras of the analytic groups of Definition \ref{defAffinoidgroups}, see \cite[\S 3]{Emerton}. In order to develop properly this theory for solid $K$-vector spaces we shall need to introduce some notation for the algebras of the analytic groups, as well as for their analytic distributions. Recall once more that we see all complete locally convex $K$-vector spaces as solid objects.

\begin{definition} \leavevmode 
\label{DefifunctDist}
\begin{enumerate} 
\item We consider the following spaces of functions 
\begin{itemize}
    \item[i.] $C(\mathbb{G}_0^{(h)}, \O_K) := \O^+(\mathbb{G}_0^{(h)}) \otimes_{\zp} \O_K$; the power bounded analytic functions of the affinoid group. 
    
    \item[ii.] $C(\mathbb{G}_0^{(h)}, K): = \O(\mathbb{G}_0^{(h)}) \otimes_{\qp} K$; the regular functions of the affinoid group.

    \item[iii.] $C(\mathbb{G}_0^{(h^+)}, K): = \varprojlim_{h'> h^+} C(\mathbb{G}_0^{(h)}, K)$; the regular functions of the Stein group. 
    \end{itemize}
    \item We define the following spaces of distributions 
\begin{itemize}
    \item[i.] $\mathcal{D}^{(h)}(G_0, \O_K) := \iHom_{\O_K}(C(\mathbb{G}_0^{(h)}, \O_K), \O_K)$.
    \item[ii.] $\mathcal{D}^{(h)}(G_0, K): = \iHom_{K}(C(\mathbb{G}_0^{(h)}, K), K)$.
    \item[iii.] $\mathcal{D}^{(h^+)}(G_0, K): = \iHom_{K}(C(\mathbb{G}_0^{(h^+)}, K), K)$.
\end{itemize}
\end{enumerate}
\end{definition}

\begin{remark}
 Let $h\in \bb{N}$,  the algebra $C(\mathbb{G}_0^{(h)}, K)$ is the space of functions on $G_0$ with values in $K$ which are analytic of radius $p^{-h}$. More precisely,  using the coordinates $g_1,\ldots, g_d \in G_0$, and identifying $G_0$ with $\Z_p^d$,   $C(\mathbb{G}_0^{(h)}, K)$ is the space of functions $f: \Z_p^d \to K$ whose restriction at cosets $x+ p^{h}\Z_p^d$  is given by a convergent power series with coefficients in $K$. Thus all the function spaces and distributions of Definition \ref{DefifunctDist} depend on the choice of the ordered basis  $g_1, \ldots, g_d$.  
\end{remark}

\begin{remark} \label{remarkDhplus}
Observe that, by Theorem \ref{theorem:duality}, one has
$ C(\mathbb{G}_0^{(h)}, K) = \mathcal{D}^{(h)}(G_0, K)^\vee $
as well as
$ \mathcal{D}^{(h^+)}(G_0, K) = \varinjlim_{h'>h} \mathcal{D}^{(h')}(G_0, K)$.  In addition,  for $h'>h$,  the inclusions $\bb{G}^{(h^{'+})}_0\subset \bb{G}^{(h')}_0\subset \bb{G}^{(h^+)}_0$ induce maps of distributions $\n{D}^{(h^{'+})}(G_0,K)\to \n{D}^{(h')}(G_0,K)\to \n{D}^{(h^+)}(G_0,  K)$.  Moreover, since the inclusion $\mathbb{G}^{(h')}_0 \subset \mathbb{G}^{(h)}_0$ is strict, the restriction map $C(\mathbb{G}^{(h)}_0, K) \to C(\mathbb{G}^{(h')}_0, K)$ is compact and hence factors through $C(\mathbb{G}^{(h)}_0, K)^S$. Dually, one has a factorisation $\mathcal{D}^{(h')}(G_0, K) \to \mathcal{D}^{(h)}(G_0, K)^B \to \mathcal{D}^{(h)}(G_0, K)$. In particular, $\mathcal{D}^{(h^+)}(G_0, K) = \varinjlim_{h'>h} \mathcal{D}^{(h')}(G_0, K)^B$ is an $LS$ space of compact type. Finally, observe that these algebras are all either solid Fr\'echet or $LS$ spaces, so they are the condensed set associated to their underlying topological spaces, which are nothing but  the classical distribution algebras endowed with the compact-open topology.  
\end{remark}

\begin{lemma}
The solid $\O_K$-module $\mathcal{D}^{(h)}(G_0, \O_K)$ has a natural structure of associative unital $\O_K$-algebra induced by the multiplication map $\bb{G}_0^{(h)} \times \bb{G}_0^{(h)} \rightarrow \bb{G}_0^{(h)}$. In particular, $\mathcal{D}^{(h)}(G_0, K)$ and $\mathcal{D}^{(h^+)}(G_0, K)$ are associative unital  $\O_K$-algebras. 
\end{lemma}

\begin{proof}
The  multiplication  $\bb{G}_0^{(h)} \times \bb{G}_0^{(h)} \rightarrow \bb{G}_0^{(h)}$ defines a comultiplication map \[\nabla: C(\mathbb{G}_0^{(h)},\O_K) \rightarrow C(\mathbb{G}_0^{(h)}, \O_K)\otimes_{\O_{K,\sol}} C(\mathbb{G}_0^{(h)},\O_K).\] As $C(\mathbb{G}_0^{(h)}, \O_K)$ is an orthonormalizable Banach $\O_K$-module, taking the dual of $\nabla$ one obtains a map $\n{D}^{(h)}(G_0, \O_K)\otimes_{\O_{K,\sol}} \n{D}^{(h)}(G_0, \O_K) \to \n{D}^{(h)}(G_0, \O_K)$ which is easily seen to be the convolution product.

\end{proof}

\subsubsection{Localizations of the Iwasawa algebra}

Recall that we have fixed an open normal subgroup $G_0$ which is a uniform pro-$p$-group with basis $g_1,\ldots, g_d$ of constant valuation $w_0>1$.   Let  $\mathbf{b}_{i}= [g_i] - 1\in \O_{K, \blacksquare}[G_0](*)$, one has \cite[\S 4]{SchTeitDist} 
\[
\O_{K, \blacksquare}[G_0](*)_{\mathrm{top}} = \prod_{\alpha \in \bb{N}^{d}} \n{O}_K \bbf{b}^{\alpha}
\]
where $\bbf{b}^{\alpha}=\bbf{b}_1^{\alpha_1}\cdots \bbf{b}_d^{\alpha_d}$ for $\alpha=(\alpha_1,\ldots, \alpha_d)\in \bb{N}^d$, and the right hand side is a profinite $\O_{K}$-module. After taking the associated condensed sets we will simply write
\[
\O_{K, \blacksquare}[G_0]= \prod_{\alpha \in \bb{N}^{d}} \n{O}_K \bbf{b}^{\alpha}.
\]

\begin{remark}
Taking Mahler expansion of continuous functions on $\phi: G_0 \cong \Z_p^d$, an explicit computation of finite differences shows that  the elements $\bbf{b}^\alpha$ correspond to the dual basis of the Mahler basis ${x \choose \alpha}$  i.e.,
 $ \bbf{b}^\alpha (f) = c_\alpha  $
for any continuous function $f \in \Cont(G_0, K)$ such that $\phi^*(f) = \sum_{\alpha \in \N^d} c_\alpha {x \choose \alpha}$.
\end{remark}

We now introduce a second family  of distribution algebras using the basis $(\bbf{b}_i)_{1 \leq i \leq d}$,  these can be thought of as localizations of the Iwasawa algebra $K_{\sol}[G_0]$.

\begin{definition} 
\label{DefiDistributionhenbas}
Let $h>0$ be rational. We define the condensed rings $\mathcal{D}_{(h)}(G_0, \O_K)$ and $\mathcal{D}_{(h)}(G_0, K)$ so that, for any extremaly disconnected set $S$, one has
\begin{itemize}
    \item $\mathcal{D}_{(h)}(G_0, \O_K)(S) = \{ \sum_{\alpha \in \N^d} a_\alpha \bbf{b}^\alpha \; : \;  \sup_{\alpha}\{ |a_\alpha |   p^{-\frac{p^{-h}}{p - 1} |\alpha|} \}  \leq 1, \;\;  a_\alpha \in \Cont(S,\O_K)  \},$
    \item $\mathcal{D}_{(h)}(G_0, K)(S) = \{ \sum_{\alpha \in \N^d} a_\alpha \bbf{b}^\alpha \; : \; \sup_{\alpha}\{ | a_\alpha |  p^{-\frac{p^{-h}}{p - 1} |\alpha|} \} < +\infty, \;\; a_\alpha \in \Cont(S,K)  \}$,
\end{itemize}
\end{definition}

\begin{remark}
The condensed module  $\mathcal{D}_{(h)}(G_0, \O_K)$ is in fact a profinite  $\O_K$-module provided $b(h) = \frac{p^{-h}}{p - 1} \in   v_p(K)\otimes_\Z \Q$. Indeed, if $b(h)$ is the valuation of an element of $K$,  we have an isomorphism of profinite $\O_K$-modules
\[ \mathcal{D}_{(h)}(G_0, \O_K) = \prod_{\alpha \in \N^d}  \O_K \left( \frac{\bbf{b}_1}{p^{b(h)}} \right)^{\alpha_1} \hdots \left( \frac{\bbf{b}_d}{p^{b(h)}} \right)^{\alpha_d} . \]
\end{remark}

\begin{remark}
The distribution algebras $\n{D}_{(h)}(G_0,K)$ are variations of the ($p$-adic completions of the) rings $\hat{A}^{(m)}$ of \cite[\S 5.2]{Emerton},  adapted to the Iwasawa algebra instead of the enveloping algebra of $\Lie G$.  
\end{remark}

\begin{lemma}
The multiplication map of $K_{\sol}[G_0]$ extends uniquely to a multiplication map of $\n{D}_{(h)}(G_0, \O_K)$.  
\proof
This follows directly from \cite[Proposition 4.2]{SchTeitDist} after taking the associated condensed sets of the corresponding topological spaces.  
\endproof
\end{lemma}

One can describe the analytic distributions of Definition \ref{DefifunctDist} in terms of  the basis $\bbf{b}_i$ of the Iwasawa algebra.  

\begin{proposition}
Let $S$ be an extremally disconnected set,  then 
\[ \mathcal{D}^{(h)}(G_0, K)(S) = \bigg\{  \sum_{\alpha \in \N^d} a_\alpha \bbf{b}^\alpha \; : \;  \sup_{\alpha} \{ | a_{\alpha} | p^{- \frac{p^{-h}|\alpha| - s(\alpha)}{p - 1}} \} < + \infty,   \;\;  a_{\alpha}\in \Cont(S,K)  \bigg\}, \]
where $s(\alpha) = \sum_{1 \leq i \leq d} s(\alpha_i)$ and $s(\alpha_i)$ is the sum of the $p$-adic digits of $\alpha_i$. In particular, one has
\[ \mathcal{D}^{(h^+)}(G_0, K)(S) = \bigg\{ \sum_{\alpha \in \N^d} a_\alpha \bbf{b}^\alpha \; : \; \sup_{\alpha} \{ | a_\alpha | p^{- \frac{p^{-h'}|\alpha| - s(\alpha)}{p - 1}}\} < + \infty \text{ for some } h' > h ,  \;\; a_{\alpha}\in \Cont(S,K) \bigg\}. \]
\end{proposition}

\begin{proof}
Let $\phi: \Z_p^d\to G_0$ be the chart defined by the basis $g_1,\ldots, g_d$.  By a theorem of Amice (c.f. \cite[III.1.3.8]{Lazard}), a continuous function $f: G_0 \to K$ is $h$-analytic (i.e. belongs to $C(\mathbb{G}_0^{(h)}, K)$) if and only if
\[ v(c_\alpha) - \frac{p^{-h}|\alpha| - s(\alpha)}{p - 1} \to + \infty \]
whenever $\alpha \to +\infty$, where $\phi^*f(g) = \sum_{\alpha \in \N^d} c_\alpha {x \choose \alpha}$.
After dualizing this gives the claimed result as  the algebra of  analytic distributions is attached to its underlying topological space.
\end{proof}

\begin{remark}
Using the formula
\[ v_p(\alpha!) = \frac{|\alpha| - s(\alpha)}{p - 1}, \]
one can rewrite the above condition on the valuation of the coefficients as
\[ \sup_{\alpha}\{ | a_\alpha \alpha ! | \, p^{-a(h) |\alpha|} \}< + \infty, \]
where $a(h) = \frac{p^{-h} - 1}{p - 1}$.
\end{remark}

\begin{corollary}
\label{coroDhpluscolimit}
There is an isomorphism of solid $\O_K$-algebras
\[ \mathcal{D}^{(h^+)}(G_0, K) \cong \varinjlim_{h' > h} \mathcal{D}_{(h')}(G_0, K). \]
\end{corollary}

\begin{proof}
Let $h'' > h'$, then one can write
\[ p^{-\frac{p^{-h'}|\alpha| - s(\alpha)}{p - 1}} = p^{-\frac{p^{-h''}}{p - 1}|\alpha|} p^{-\frac{(p^{-h'} - p^{-h''})|\alpha| - s(\alpha)}{p - 1}}. \]
Since $p^{-h'} - p^{-h''} > 0$, we see that  $(p^{-h'}-p^{-h''})|\alpha|- s(\alpha) \to  +\infty $  as $|\alpha| \to +\infty$.  This implies that for any $h''> h'$ we have
\[ \mathcal{D}_{(h'')}(G_0, K) \subseteq \mathcal{D}^{(h')}(G_0, K) \subset  \mathcal{D}_{(h')}(G_0, K) . \]
Taking limits as $h'\to h^+$ and $h''\to h^+$ one obtains the corollary. 
\end{proof}

\subsubsection{Distribution algebras over $G$}
\label{SubsubSectionDistG}

The algebras we have defined can be extended to the whole compact group $G$ in a obvious way.  Indeed, by Lemma \ref{LemmaNormalG0}  the spaces $\n{D}^{(h)}(G_0,K)$ and $\n{D}_{(h)}(G_0,K)$ admit an action of $G$ extending the inner action of $K_{\sol}[G_0]$. Let us define the distributions
\[ \n{D}^{(h)}(G,K)= K_{\sol}[G] \otimes_{K_{\sol}[G_0]} \n{D}^{(h)}(G_0,K), \]
\[
\n{D}^{(h^+)}(G,K)= \varinjlim_{h'>h} \n{D}^{(h')}(G,K),
\]
\[ \n{D}_{(h)}(G,K)= K_{\sol}[G] \otimes_{K_{\sol}[G_0]} \n{D}_{(h)}(G_0,K), \]
where in the tensor products we see  $\n{D}_{(h)}(G_0,K)$ as a left $K_{\sol}[G_0]$-module.  They are unital associative algebras admitting $K_{\sol}[G]$ as a dense subspace. Notice that even though the distributions algebras over $G$ depend on the choice of the open normal subgroup $G_0\subset G$,  the projective systems $\{\n{D}^{(h)}(G,K)\}_{h>0}$,  $\{ \n{D}^{(h^+)}(G,K)\}_{h>0}$ and $\{\n{D}_{(h)}(G,K)\}_{h>0}$ do not. Moreover,  by Remark \ref{remarkDhplus} and Corollary \ref{coroDhpluscolimit} these inverse systems are isomorphic.    We define the algebra of locally analytic distributions of $G$ as the Fr\'echet algebra 
\[
\n{D}^{la}(G,K)= \varprojlim_{h\to \infty} \n{D}^{(h^+)}(G,K)=\varprojlim_{h\to \infty} \n{D}^{(h)}(G,K),
\]
by Remark \ref{remarkDhplus} it is a Fr\'echet space of compact type.

For future reference let us  introduce  algebras of  analytic functions over $G$.   Let $h\geq 0$  be a rational number and $\bb{G}_0^{(h)}$ the rigid analytic groups of Definition \ref{defAffinoidgroups},  recall that they are stable under conjugation by $G$.  Let $\bb{G}^{(h)}$ be the rigid analytic group given by $G\bb{G}_0^{(h)}$, i.e. if $s_1, \ldots, s_n$ are representatives of the cosets $G/G_0$  then
\[ \bb{G}^{(h)}= \bigsqcup_{i=1}^n s_i\bb{G}_0^{(h)}. \]
We also define
\[ \bb{G}^{(h^+)}= \bigcup_{h'>h} \bb{G}^{(h)}. \]
Let $C(\bb{G}^{(h)}, K)= \mathscr{O}(\bb{G}^{(h)})\otimes_{\Q_p}K$ be the affinoid algebra of analytic functions of $G$ over $K$ of radius $p^{-h}$,  it is immediate to check that $\n{D}^{(h)}(G,K)= \underline{\Hom}_K(C(\bb{G}^{(h)}, K),K)$.   We define $C(\bb{G}^{(h)},\O_K)$ and  $C(\bb{G}^{(h^+)}, K)$ in the obvious way.  The space of $K$-valued locally analytic functions of $G$ is the $LB$ space $C^{la}(G,K  )= \varinjlim_{h\to \infty} C(\bb{G}^{(h)},  K)$.  Note that,  since the inclusion $\bb{G}^{(h')}\subset \bb{G}^{(h)}$ is a strict immersion for $h'>h$,   $C^{la}(G,K)$ is an $LB$ space of compact type,  or equivalently,  an $LS$ space of compact type (Corollary \ref{coroFLScompact}).  By  Theorem  \ref{theorem:duality},   $\n{D}^{la}(G,K)= \iHom_K(C^{la}(G,K), K)$ is a Fr\'echet space of compact type. Summarizing, for $h'>h$,  we have the following maps of distribution algebras and and their corresponding dual spaces of analytic functions
\begin{equation} \label{EqDiagramDistributions}
\begin{aligned}
 K_{\sol}[G]\to  \n{D}^{la}(G,K) \to  \cdots \to   \n{D}^{(h')}(G,K) \to  \n{D}^{(h^+)}(G,K) \to \n{D}^{(h)}(G,K) \\
\underline{\Cont}(G,K) \leftarrow   C^{la}(G,K ) \leftarrow \cdots \leftarrow C(\bb{G}^{(h')},K)  \leftarrow  C(\bb{G}^{(h^+)},K) \leftarrow C(\bb{G}^{(h)},K). 
\end{aligned}
\end{equation}

\subsection{Solid $G$-modules} 
Let $G$ be a profinite group and $\O_{K,\sol}[G]$ its Iwasawa algebra over $\O_K$.

\begin{lemma}
\label{LemmaDefSolidGmod}
Let $V$ be a solid $\O_K$-module. An  $\n{O}_{K,\sol}[G]$-module structure on $V$ is equivalent to the following data:   for all extremally disconnected set $S$,  an  $\O_K$-linear action $C(S,G)\times V(S)\rightarrow V(S)$ which is functorial on $S$.
\end{lemma}

\begin{proof}

Let $V$ be an  $\O_{K,\sol}[G]$-module and $S$ an extremally disconnected set.  There is a natural map of condensed sets $[\cdot]:G\to \O_{K,\sol}[G]$.  The action of $\O_{K,\sol}[G]$ over $V$ is provided by a linear map $\O_{K,\sol}[G]\otimes_{\O_{K,\sol}} V\to V$ satisfying the usual axioms.  Composing with $[\cdot]$ and evaluating at $S$ we obtain a map $C(S,G)\times V(S)\to V(S)$,  it is easy to check that this provides a functorial action as in (2).

Conversely, to have such a functorial action is equivalent to having a map of condensed sets $G\times V\to V$ making  the usual diagrams commutative. By adjunction of  the functor $X\to \Z[X]$,  this provides a linear map $\Z[G]\otimes_{\Z} V \to  V $. As $V$ is a solid $\O_{K}$-module it extends uniquely to a map $\O_{K,\sol}[G]\otimes_{\O_{K,\sol}}V\to V$ which is easily seen to satisfy  the obvious diagrams of an  $\O_{K,\sol}[G]$-module. 
\end{proof}

\begin{definition}
A  solid $G$-module  over $\O_K$ (or a solid $\O_{K,\sol}[G]$-module) is a solid abelian group endowed with an  $\O_{K,\sol}[G]$-module structure.      We denote the category of solid $\O_{K, \blacksquare}[G]$-modules by $\SolidOKG$ and its derived category by $D(\O_{K,\blacksquare}[G])$.  
\end{definition}

Let us extend the definition of the condensed set of ``continuous functions'' to complexes:
\begin{definition}
Let $V$ be a solid $\n{O}_K$-module and $S$ a profinite set, we denote by $\underline{\Cont}(S,V)$ the $\O_K$-module $\underline{\Hom}(\Z[S], V)= \underline{\Hom}_{\O_K}(\O_{K, \blacksquare}[S], V)$.  More generally, for $C \in D(\O_{K,\blacksquare})$, we denote $\underline{\Cont}(S,C):= R\underline{\Hom}_{\O_K}(\O_{K, \blacksquare}[S],  C)$, which is consistent with the previous definition as $\O_{K, \blacksquare}[S]$ is a projective module.
\end{definition}

\begin{prop}
The functor $V \mapsto \underline{\Cont}(G,V)$ is exact and factors through a functor $\SolidOK \rightarrow \Mod_{\O_{K,\sol}[G^2]}^{\solid}$ induced by the left and right regular actions respectively.  Moreover, it extends to an exact  functor of derived categories $D(\O_{K,\sol}) \rightarrow D(\O_{K,\sol}[G^2])$. 
\proof
As $G$ is profinite,  $\O_{K,\sol}[G]$ is a compact projective $\O_{K,\sol}$-module. This makes the functor $\underline{\Cont}(G,V)$ exact. Thus, the second statement reduces to the first one. We prove the first statement. Let $S$ be an extremally disconnected set, then $\underline{\Cont}(G, V)(S)= \Cont(G\times S, V)=V(G\times S)$.  Let us define a map 
\[
\Cont(S,G^2)\times V(G\times S)\to V(G\times S)
\]
as in (2) of Lemma \ref{LemmaDefSolidGmod}.  Let $f=(f_1,f_2): S\to G^2$ and $v:  G\times S \to V$ be objects in $\Cont(S,G^2)$ and $V(G\times S)$ respectively.  We define the product $f\cdot v$ to be the composition 
\[
\begin{tikzcd}[row sep= 5pt]
S\times G \ar[r]&  S\times G  \ar[r] & V \\
(s,g)  \ar[r, mapsto]& (s,f_1(s)^{-1}g f_2(s)) \ar[r, mapsto] & v(s, f_1(s)^{-1} g f_2(s)).  
\end{tikzcd}
\]
It is immediate so check that this endows $\underline{\Cont}(G,V)$ with a action of $G^2$ which is the left and right regular action on the first and second component respectively. 
\endproof
\end{prop}

\begin{remark}
If we suppose in addition  that $V$ is a $\O_{K,\sol}[G]$-module, then $\underline{\Cont}(G,V)$ is naturally a $\O_{K,\sol}[G^3]$-module.  Namely,  for $S$ be an extremally disconnected set,    $f_1,f_2,f_3\in \Cont(S,G)$ and $v\in V(G\times S)$, we have the action
\[
[(f_1,f_2,f_3)\cdot v](g,s)= f_3(s) v(f_1(s)^{-1}g f_2(s) ,s).
\]
This action induces an exact functor of derived categories $D(\O_{K,\sol}[G])\rightarrow D(\O_{K,\sol}[G^3])$.  
\end{remark}

\begin{definition}
\label{defiStaraction}
Let $V$ be a solid $G^n=\overbrace{G \times \cdots \times G}^{n} $-representation over $\O_K$.  Given $I\subset \{ 1, 2, \ldots , n \}$ a non empty subset,  we denote by $\star_I$ the diagonal action of $G$  on $V$ induced by the  embedding $\iota_I:G\rightarrow G^n$ in the components of $I$.   We denote by  $V_{\star_I}$   the module $V$ endowed with the action $\star_I$.  If $I=\emptyset$ we write $V_0:=V_{\emptyset}$ for the $\O_{K,\sol}$-module $V$ endowed with the trivial action of $G$. 
\end{definition}

The following proposition basically says that any action on a solid module is continuous (compare it with \cite[Definition 3.2.8]{Emerton}).

\begin{prop}
\label{PropCoinductionTrivial}
Let $C$ be an object in $D(\O_{K,\sol}[G])$. Then there is a natural quasi-isomorphism  of $\O_{K,\sol}[G]$-modules
\begin{equation} \label{EqisomCont}
R\underline{\Hom}_{\O_{K, \blacksquare}[G]}(\O_K, \underline{\Cont}(G,C)_{\star_{1,3}})\xrightarrow{\sim} C,
\end{equation}
where the action of $\O_{K,\sol}[G]$ in the left-hand-side is via the  $\star_2$-action. The inverse of this map is called the orbit map of $C$. 
\proof
First, we claim that there exists a natural  quasi-isomorphism $\underline{\Cont}(G,C)_{\star_{1,3}} \simeq \underline{\Cont}(G,C)_{\star_1}$ for $C\in D(\O_{K,\sol}[G])$.   Suppose  that the previous is true, then we have \begin{eqnarray*}
R\underline{\Hom}_{\O_{K,\sol}[G]}(\O_K, \underline{\Cont}(G,C)_{\star_{1,3}}) & \simeq &  R\underline{\Hom}_{\O_{K,\sol}[G]}(\O_K, \underline{\Cont}(G,C)_{\star_{1}}) \\
  & = & R\underline{\Hom}_{\O_{K,\sol}[G]}(\O_K, R\underline{\Hom}_{\O_K}(\O_{K,\sol}[G],  C)_{\star_1}) \\
  & = & R\underline{\Hom}_{\O_K}(\O_K\otimes^L_{\O_{K,\sol}[G]} \O_{K,\sol}[G],  C) \\
  & = & R\underline{\Hom}_{\O_K}(\O_K, C) = C. 
\end{eqnarray*}

To prove the claim, it is enough to define a natural isomorphism $\underline{\Cont}(G,V)_{\star_{1,3}}\rightarrow \underline{\Cont}(G,V)_{\star_1}$ for $V\in \SolidOKG$.  Let $S$ be an extremally disconnected set, and take $v\in V(G\times S)$.   Consider the inverse map $u: G \rightarrow G \;\; g\mapsto g^{-1}$  and the multiplication map $m_V: G\times V\rightarrow V$. Define $\psi_V(v)$ to be the composition 
\[
\psi_V(v):  G\times S \xrightarrow{u\times v} G\times V \xrightarrow{m_V} V. 
\]
The application
\begin{gather*}
   \psi_V: V(G\times S) \rightarrow V(G\times S) \\
    v \mapsto  \psi_V(v)
\end{gather*}
induces an isomorphism of solid $\O_{K}$-modules $\psi_{V}: \underline{\Cont}(G,V)\to \underline{\Cont}(G,V) $.  It is easy to check that it transfers the $\star_{1,3}$-action to the $\star_{1}$-action and the $\star_{2}$-action to the $\star_{2,3}$-action. This proves the claim and that the isomorphism \eqref{EqisomCont} is $G$-equivariant.
\endproof
\end{prop}

\begin{remark}
The previous proof shows that if $V$ is a $\O_{K,\sol}[G]$-module arising from a topological space,  the isomorphism  $V\rightarrow (\underline{\Cont}(G,V)_{\star_{1,3}})^{G}$ is given by the usual orbit map $v\mapsto (g\mapsto gv)$. 

\end{remark}

\subsection{Analytic representations}
\label{subsecAnalyticrep}

Let $h\geq 0$ and $\bb{G}^{(h)}$ the rigid analytic group of \S \ref{subsecFuncDist} extending the group law of $G$.  We recall that $\bb{G}^{(h)}$ depends on the choice of an open normal uniform pro-$p$-subgroup $G_0\subset G$.    To motivate the forthcoming definitions of analytic vectors let us first recall how this works for Banach spaces, where we follow \cite[\S 3]{Emerton}.

Let $V$ be a $K$-Banach space endowed with a continuous action of $G$, the space of $V$-valued  $\bb{G}^{(h)}$-analytic functions is by definition the projective  tensor product $C(\bb{G}^{(h)},V):= C(\bb{G}^{(h)},K)\widehat{\otimes}_K V$. As $V$ and $C(\bb{G}^{(h)},K)$ are Banach spaces, the projective tensor product coincides with the solid tensor product $ C(\bb{G}^{(h)},K)\otimes_{K_\sol} V$ (Lemma  \ref{CorotensorBanach}).   This space has an action of $G^2$ given by the left and right regular actions of $G$, and an extra action of $G$ induced by the one of  $V$.  Following the notation of Definition \ref{defiStaraction}, the $\bb{G}^{(h)}$-analytic vectors of $V$ is the Banach space 
\begin{equation}
\label{eqanBanach}
V^{\bb{G}^{(h)}-an}:= (C(\bb{G}^{(h)},V)_{\star_{1,3}})^G.
\end{equation}
There is a natural map $V^{\bb{G}^{(h)}-an}\to V$ given by evaluating at $1\in \bb{G}^{(h)}$, and $V$ is $\bb{G}^{(h)}$-analytic if the previous arrow is an isomorphism.

To generalise the previous construction of analytic vectors to  solid $K_{\sol}[G]$-modules we need to  rewrite (\ref{eqanBanach}) in a slightly different way.  Consider the affinoid ring  $(C(\mathbb{G}^{(h)}, K), C(\mathbb{G}^{(h)}, \O_K))$, it is a finite product of Tate power series rings in $d$-variables.  In Example \ref{Exampleanalyticrings}  (4) and (5),   we saw how the affinoid ring provides a natural analytic ring that we denote as $C(\mathbb{G}^{(h)}, K)_{\sol}$.  We also denote by $C(\mathbb{G}^{(h)},\O_K)_{\sol}$ the analytic ring attached to its subalgebra of power-bounded elements.  Now,  as $V$ is a $K$-Banach vector space,  one has 
\[
C(\bb{G}^{(h)},V)= C(\bb{G}^{(h)},K)\otimes_{K_{\sol}}V=  C(\bb{G}^{(h)},K)_{\sol}\otimes_{K_{\sol}}V,
\]
where the last tensor product is the completion functor with respect to the measures of $C(\bb{G}^{(h)},K)_{\sol}$. The last equality follows  by Corollary  \ref{PropsixFunctorTate};  one has that $C(\bb{G}^{(h)},K)_{\sol}\otimes_{K_{\sol}}V= \iHom_{K}(\n{D}^{(h)}(G,K),  V)=C(\bb{G}^{(h)},K)\otimes_{K_{\sol}}V$,  where in the last equality we use the nuclearity of $V$ (cf. Corollary \ref{CoroBanachisNuclear}).  Hence, we can write the $\bb{G}^{(h)}$-analytic vectors of $V$ in the form
\[
V^{\bb{G}^{(h)}-an}= \underline{\Hom}_{K_{\sol}[G]}(K,  (C(\bb{G}^{(h)},K)_{\sol}\otimes_{K_{\sol}}V)_{\star_{1,3}}).
\]

In order to generalise the construction of analytic vectors we need some basic properties of the tensor  $ C(\bb{G}^{(h)},\O_K)_{\sol}\otimes_{\O_{K,\sol}}-$.

\begin{prop} Consider the functor $V \mapsto C(\bb{G}^{(h)},\O_{K})_{\sol}\otimes_{\O_{K,\sol}} V $ for $V\in \SolidOK$. The following statements hold. 
\begin{enumerate}
\item The functor is exact. 

\item It induces an exact functor of derived categories $D(\O_{K_{\sol}})\rightarrow D(\O_{K, \sol}[G^2])$ given by the left and right regular actions.

\item There is a functorial map $C(\mathbb{G}^{(h)},\O_K)_{\sol} \otimes^L_{\O_K} C \rightarrow \underline{\Cont}(G, C)$ for $C\in D(\O_{K_{\sol}})$ compatible with the left and right regular actions. 
\end{enumerate}
\proof
Exactness follows from Corollary \ref{PropsixFunctorTate}.  Indeed, as $\bb{G}^{(h)}$ is a finite disjoint union of polydiscs one has 
\begin{equation}
\label{eqSixfunctGh}
C(\bb{G}^{(h)}, \O_K)_{\sol}\otimes^L_{\O_{K,\sol}} V= R\underline{\Hom}_{\O_K}(\n{D}^{(h)}(G,\O_K), V).
\end{equation}
But $\n{D}^{(h)}(G,\O_K)$ is a projective $\O_K$-module, this implies that 
\[
C(\bb{G}^{(h)},  \O_K)_{\sol} \otimes^{L}_{\O_{K,\sol}} V = \underline{\Hom}_{\O_K}(\n{D}^{(h)}(G,\O_K), V)=C(\bb{G}^{(h)},  \O_K)_{\sol} \otimes_{\O_{K,\sol}} V
\]
is exact.

To prove (2), it is enough to show that $C(\bb{G}^{(h)},\O_K)_{\sol}\otimes_{\O_{K,\sol}} V$ has  natural left and right regular actions for $V\in \SolidOK$.   Writing $V$ as a quotient $P_1 \rightarrow P_0 \rightarrow V$ of objects of the form $P_i= \bigoplus_{I_i} \prod_{J_i} \O_K $, we have an exact sequence 
\[
C(\mathbb{G}^{(h)},\O_K)_{\sol}\otimes_{\O_K} P_1 \xrightarrow{f}   C(\mathbb{G}^{(h)},\O_K)_{\sol} \otimes_{\O_K} P_0 \rightarrow C(\mathbb{G}^{(h)}, \O_K)_{\sol}\otimes_{\O_K}V\rightarrow 0. 
\]
The functor $C(\bb{G}^{(h)},\O_K)_{\sol}\otimes_{\O_K}-$ commutes with colimits and, by Equation \eqref{eqSixfunctGh}, it also commutes with products. Hence
\[ C(\mathbb{G}^{(h)},\O_K)_{\sol}\otimes P_i= \bigoplus_{I_i} \prod_{J_i} C(\mathbb{G}^{(h)}, \O_K) \]
and these modules are equipped with the natural left and right regular actions of $G$. Moreover, the map $f$ is equivariant for these actions. We endow $C(\mathbb{G}^{(h)}, \O_K)_{\sol}\otimes_{\O_K}V$ with the action induced by the quotient map.  It is easy to check that this action is independent of the presentation of $V$, and that it is functorial.

For the last statement, it is enough to construct a functorial equivariant map 
\[
C(\mathbb{G}^{(h)}, \O_K)_{\sol}\otimes_{\O_K} V \rightarrow \underline{\Cont}(G,V)
\]
for $V\in \SolidOK$. Recall that by definition $\underline{\Cont}(G,V)=\underline{\Hom}_{\O_K}(\O_{K, \sol}[G], V )$.  Similarly as before, we are reduced  to constructing the map for an object of the form $P=\bigoplus_I \prod_J \O_K$. As both functors commute with colimits and products, one reduces to treat the case $P=\O_K$, for which we have the natural inclusion $C(\bb{G}^{(h)}, \O_K) \rightarrow \underline{\Cont}(G,\O_K)$ provided by $G\subset \bb{G}^{(h)}$,  which is equivariant for the left and right regular actions of $G$. This ends the proof.
\endproof
\end{prop}

\begin{remark}
It is clear that if $V$ is a solid $G$-module then $C(\bb{G}^{(h)},  \O_K)_{\sol}\otimes_{\O_{K,\sol}} V$ can be endowed with an action of $G^3$.  
\end{remark}

\begin{definition} 
\label{DefiAnalyticVectors}
Let $h\geq 0$
\begin{enumerate}
\item  Let $V\in \SolidKG$,  the space of   $\bb{G}^{(h)}$-analytic vectors of $V$ is   the solid $K_{\sol}[G]$-module 
\[
V^{\bb{G}^{(h)}-an}:=\underline{\Hom}_{K_{\sol}[G]}(K, (C(\bb{G}^{(h)},K)_{\sol} \otimes_{K_{\sol}} V)_{\star_{1,3}})
\]
where the action of $G$ is induced by the $\star_2$-action.  Similarly, we define the $\bb{G}^{(h^+)}$-analytic   vectors of $V$ to be \[ V^{\bb{G}^{(h^+)}-an}:= \varprojlim_{h'>h} V^{\bb{G}^{(h')}-an}, \]
where the transition maps are induced by the base change of analytic rings   $C(\bb{G}^{(h')},K)_{\sol}\otimes_{K_{\sol}} V \to C( \bb{G}^{(h'')}, K)_{\sol}  \otimes_{K_{\sol}} V $  for $h''>h'>h$.   
    
\item Given a complex $C\in D(K_{\sol}[G])$ we define the derived $\bb{G}^{(h)}$-analytic vectors of $C$ as the complex in $D({K_{\sol}}[G])$
\[
C^{R\bb{G}^{(h)}-an}:=R\underline{\Hom}_{K_{\sol}[G]}(K, (C(\bb{G}^{(h)},K)_{\sol} {\otimes_{K_{\sol}}}^{L}   C)_{\star_{1,3}})
\]
where the action of $G$ is induced by the $\star_2$-action.  Similarly,  we define the derived $\bb{G}^{(h^+)}$-analytic vectors of $C$ to be
\[ C^{R \bb{G}^{(h^+)}-an}:= R\varprojlim_{h'>h}C^{R \bb{G}^{(h')}-an}. \] 
\end{enumerate}
\end{definition}

\begin{remark}
As $C(\bb{G}^{(h)},K)_{\sol}\otimes_{K_{\sol}}-$ is exact,  $C \mapsto C^{R \bb{G}^{(h)}-an}$ is the right derived functor of $V\mapsto V^{\bb{G}^{(h)}-an}$.  Similarly,  $C \mapsto C^{R \bb{G}^{(h^+)}-an}$ is the right derived functor of $V\mapsto V^{\bb{G}^{(h^+)}-an}$.
\end{remark}

\begin{lemma}
Let $h\geq 0$ and $C\in D(K_{\sol}[G])$.   There is a natural morphism of objects in   $D(K_{\sol}[G])$
\[
 C^{ R\bb{G}^{(h)}-an}\rightarrow C. 
\]
\proof
Notice that we have a natural map 
\[
C(\bb{G}^{(h)}, K)_{\sol}\otimes^L_{K}C \rightarrow \underline{\Cont}(G, C)
\]
which commutes with the three actions of $G$. Taking $\star_{1,3}$-invariant  one gets the lemma by Proposition \ref{PropCoinductionTrivial}.  
\endproof
\end{lemma}

\begin{definition} Let $h\geq 0$
\begin{enumerate}
    \item A solid $K_{\sol}[G]$-module $V$ is called $\bb{G}^{(h)}$-analytic  if the natural map $V^{\bb{G}^{(h)}-an}\to V$ is an isomorphism.  Similarly, it is called $\bb{G}^{(h^+)}$-analytic if $V^{\bb{G}^{(h^+)}-an}\to V$ is an isomorphism. 
    
    \item A complex $C\in D(K_{\sol}[G])$ is called derived $\bb{G}^{(h)}$-analytic if the natural map $C^{R \bb{G}^{(h)}-an}\rightarrow C$ is a quasi-isomorphism. Similarly, it is derived $\bb{G}^{(h^+)}$-analytic if the map $C^{R \bb{G}^{(h^+)}-an}\rightarrow C$  is a quasi-isomorphism. 
    
\end{enumerate}
\end{definition}

So far we  have introduced two definitions of analytic vectors depending on whether we choose the radius to be closed or open.  It turns out that to have a theory in terms of modules over distribution algebras, we need  to work with the Stein analytic groups $\bb{G}^{(h^+)}$. The main reason is that these algebras are localizations of the Iwasawa algebra, as it is reflected in the following Lemma (to be proved in \S \ref{subsectionKeyLemmas}, see  Corollary \ref{CoroTechnicalDhplus}).

\begin{lemma}
\label{LemmaDhtensor}
One has $\n{D}^{(h^+)}(G, K) \otimes^L_{K_{\sol}[G]} \n{D}^{(h^+)}(G,K)=\n{D}^{(h^+)}(G,K)$. 
\end{lemma}

An immediate consequence of the previous result is the following fully-faithfulness property.

\begin{cor}
\label{coroAnalyticDh}
The category $\Mod_{\n{D}^{(h^+)}(G,K)}^{\solid}$ (resp.  $D( \n{D}^{(h^+)}(G,K))$) is a full subcategory of $\SolidKG$ (resp. $D(K_{\sol}[G])$).  In other words, if $V,V'\in \Mod_{\n{D}^{(h^+)}(G,K)}^{\solid}$ and $C,C'\in D( \n{D}^{(h^+)}(G,K))$ then 
\begin{gather*}
\underline{\Hom}_{K_{\sol}[G]}(V, V')=\underline{\Hom}_{\n{D}^{(h^+)}(G,K)}(V, V') \\
R\underline{\Hom}_{K_{\sol}[G]}(C,C')= R\underline{\Hom}_{\n{D}^{(h^+)}(G,K)}(C,C'). 
\end{gather*}
\proof
We first address the second statement, the first one follows similarly. It follows from the usual extension of scalars:
\begin{eqnarray*}
R\underline{\Hom}_{K_{\sol}[G]}(C,C') &= & R\underline{\Hom}_{\n{D}^{(h^+)}(G,K)}(\n{D}^{(h^+)}(G,K)\otimes^{L}_{K_{\sol}[G]} C,C')  \\
                & = &  R\underline{\Hom}_{\n{D}^{(h^+)}(G,K)}(\n{D}^{(h^+)}(G,K)\otimes^{L}_{K_{\sol}[G]}  (\n{D}^{(h^+)}(G,K) \otimes_{\n{D}^{(h^+)}(G,K)}^L  C),C') \\ 
                & = & R\underline{\Hom}_{\n{D}^{(h^+)}(G,K)}((\n{D}^{(h^+)}(G,K)\otimes^{L}_{K_{\sol}[G]}  \n{D}^{(h^+)}(G,K) )\otimes_{\n{D}^{(h^+)}(G,K)}^L  C,C')\\  
                & = & R\underline{\Hom}_{\n{D}^{(h^+)}(G,K)}(\n{D}^{(h^+)}(G,K) \otimes_{\n{D}^{(h^+)}(G,K)}^L  C,C') \\
                & = & R\underline{\Hom}_{\n{D}^{(h^+)}(G,K)}( C,C').
\end{eqnarray*}
\endproof
\end{cor}

\begin{remark}
As it will be seen in Corollary \ref{CoroTechnicalDhplus},  the same holds true for the distribution algebra $\n{D}^{la}(G,K)$.   In particular, $D(\n{D}^{la}(G,K))$ is a full subcategory of $D(K_{\sol}[G])$.
\end{remark}

We can now state the main theorem of this section.

\begin{theorem}
\label{TheoMain}
Let $W\in D(K_{\sol})$ and $C \in D(K_{\sol}[G])$. The following holds.

\begin{enumerate}
\item There are  natural isomorphisms of $K_{\sol}[G]$-modules 
\begin{gather*}
R\underline{\Hom}_{K_{\sol}[G]}(\n{D}^{(h)}(G,K) {\otimes^{L}_{K_\sol}} W, C) = R\underline{\Hom}_{K}(W, C^{R\bb{G}^{(h)}-an}) \\
R\underline{\Hom}_{K_{\sol}[G]}(\n{D}^{(h^+)}(G,K)  {\otimes^{L}_{K_\sol}} W, C) = R\underline{\Hom}_{K}(W, C^{R \bb{G}^{(h^+)}-an}) .
\end{gather*}
The $K_{\sol}[G]$-module structure of the terms inside the $R\underline{\Hom}_{K_{\sol}[G]}$ in the LHS are the left multiplication on the distribution algebras and the action of $C$.  The $G$-action of the LHS $R\underline{\Hom}_{K_\sol[G]}(-,-)$ is induced by the right multiplication on the distribution algebras.

\item The category of $\bb{G}^{(h^+)}$-analytic representations of $G$ is equal to $\Mod_{\n{D}^{(h^+)}(G,K)}^{\solid}$.  In other words, a $K_{\sol}[G]$-module $V$ is $\bb{G}^{(h^+)}$-analytic if and only if the action of $K_{\sol}[G]$ extends to an action of $\n{D}^{(h^+)}(G,K)$.  

\item Furthermore, a complex $C\in D(K_{\sol}[G])$ is derived $\bb{G}^{(h^+)}$-analytic if and only if for all $n\in \Z$ the cohomology groups $H^n(C)$ are $\bb{G}^{(h^+)}$-analytic.  Equivalently,  $C$ is derived $\bb{G}^{(h^+)}$-analytic if and only if it belongs to the essential image of $D(\n{D}^{(h^+)}(G,K))$. 

\end{enumerate}

\proof
$(1)$ Let $W\in D(K_{\sol})$ and $C\in D(K_{\sol}[G])$, by Corollary \ref{PropsixFunctorTate} there is a natural quasi-isomorphism 
\begin{equation}
\label{eqtheo1}
R\underline{\Hom}_K( \n{D}^{(h)}(G,K) \otimes^L_{K_{\sol}} W, C ) = R\underline{\Hom}_K (W, C(\bb{G}^{(h)},K)_{\sol} \otimes^L_{K} C). 
\end{equation}
It is easy to verify that the left and right regular actions of the RHS are translated in the left and right multiplication of the distributions in the LHS. Indeed, one can reduce to $W= \prod_i \O_K [\frac{1}{p}]$ and $C= K_{\sol}[G]$ in which case it is straightforward.  Then,  the $\star_{1,3}$-action in the RHS translates in the  left multiplication on the distributions and the action on $C$ in the LHS.  Taking $R\underline{\Hom}_{K_{\sol}[G]}(K,-)$ in (\ref{eqtheo1}) one gets
\[
R\underline{\Hom}_{K_{\sol}[G]}(\n{D}^{(h)}(G,K) {\otimes^{L}_{K_\sol}} W, C) = R\underline{\Hom}_{K}(W, C^{R\bb{G}^{(h)}-an}). 
\]
Taking derived inverse limits and using that $R\underline{\Hom}$ commutes with colimits in the first factor and limits in the second factor, one gets 
\[
R\underline{\Hom}_{K_{\sol}[G]}(\n{D}^{(h^+)}(G,K)  {\otimes^{L}_{K_\sol}} W, C) = R\underline{\Hom}_{K}(W, C^{R \bb{G}^{(h^+)}-an}).  
\]

$(2)$ Consider the pre-analytic ring $(K_{\sol}[G],  \n{M}^{(h^+)})$ such that for any extremally disconnected $S$  one has 
\[
\n{M}^{(h^+)}(S)= \n{D}^{(h^+)}(G,K)\otimes_{K_{\sol}} K_{\sol}[S].  
\]
Corollary \ref{coroAnalyticDh} implies that it is  in fact an analytic ring.  Indeed,  let $P^\bullet$ be a complex of $K_{\sol}[G]$-modules concentrated in positive homological degrees whose terms are direct sums of $\n{M}^{(h^+)}[S_i]$ for $\{S_i\}_{i\in I}$ a family of profinite sets.  Let $S$ be a profinite set,  then
\begin{eqnarray*}
R\underline{\Hom}_{K_\sol[G]}(\n{M}^{(h^+)}[S],  P^\bullet) & = & R\underline{\Hom}_{\n{D}^{(h^+)}(G,K)}(\n{M}^{(h^+)}[S],  P^\bullet) \\
            & = & R\underline{\Hom}_{K}(K_{\sol}[S], P^\bullet). 
\end{eqnarray*}
Moreover, the category of solid $(K_\sol[G], \n{M}^{(h^+)})$-modules is equal to the category of solid $\n{D}^{(h^+)}(G,K)$-modules. More precisely, by Theorem \ref{TheoLemmaAnalyticrings}, a family of compact projective generators of $\Mod^{\solid}_{(K_\sol[G], \n{M}^{(h^+)})}$ is given by $\n{M}^{(h^+)}(S)$ for $S$ extremally disconnected, which are naturally $\n{D}^{(h^+)}(G, K)$-modules, and any $K_\sol[G]$-linear map between these objects is automatically $\n{D}^{(h^+)}(G, K)$-linear again by Corollary \ref{coroAnalyticDh} again. Hence, part (1)  and Theorem \ref{TheoLemmaAnalyticrings} imply that a $K_\sol[G]$-module $V$ is $\bb{G}^{(h^+)}$-analytic if and only if it is a $\n{D}^{(h^+)}(G,K)$-module.   

$(3)$ Theorem \ref{TheoLemmaAnalyticrings} and the same argument as before tells us that a complex $C\in D(K_{\sol}[G])$ is derived $\bb{G}^{(h^+)}$-analytic if and only if it belongs to the essential image of $D(\n{D}^{(h^+)}(G,K))$, if and only if for all $n\in \bb{Z}$ the module $H^n(C)$ is a $\n{D}^{(h^+)}(G,K)$-module, finishing the proof. 
\endproof
\end{theorem}

\begin{remark}
Let us highlight the importance of the compactness assumption on $G$ in the previous theorem.  Let $G$ be a locally profinite $p$-adic Lie group,  and let $G_0\subset G$ an open compact subgroup. One can construct, as we did above, rigid analytic neighbourhoods $\bb{G}_{0}^{(h^+)}$ of $G_0$ and corresponding distribution algebras $\n{D}^{(h^+)}(G_0,K)$.  One may ask if there are analogous of these (analytic) distribution algebras $\n{D}^{(h^+)}(G,K)$ over $G$ in such a way that the $G$-representations whose restriction to $G_0$ is $\bb{G}_0^{(h^+)}$-analytic are  the same as $\n{D}^{(h^+)}(G,K)$-modules,  but this turns out to be false in general.  Indeed, let $W$ be a solid $G$-representation,    by part (1) of Theorem  \ref{TheoMain} the derived  $\bb{G}_0^{(h^+)}$-analytic vectors of $W$ are equal to 
\[
R\iHom_{K_{\sol}[G_0]}(\n{D}^{(h)}(G_0,K), W  )= R\iHom_{K_{\sol}[G]}( K_{\sol}[G] \otimes_{K_{\sol}^L[G_0]}\n{D}^{(h)}(G_0,K), W ),
\]
so that  the natural candidate is $\n{D}^{(h^+)}(G,  K):= K_{\sol}[G]\otimes_{K_{\sol}[G_0]} \n{D}^{(h^+)}(G_0,K)$,  but this $G$-representation is not an algebra unless $G$ normalizes $G_0$.

Another reason why the distribution algebras $\n{D}^{(h^+)}(G,K)$ do not exist in general is that the action of $G$ might change the radius of analyticity of a locally analytic vector.  For example, take  $G= \GL_2(\qp)$,  $G_0 = \left( \begin{array}{cc} 1+p^2\mathbf{Z}_p  & p^2 \mathbf{Z}_p \\ p^2 \zp & 1+p^2 \zp \end{array}\right)$ and $g= \left( \begin{array}{cc} p & 0 \\ 0 & 1  \end{array} \right)$.   Take $\bb{G}_0:= \left( \begin{array}{cc} 1+p^2 \bb{D}^1_{\mathbf{Q}_p}  & p^2 \bb{D}^1_{\mathbf{Q}_p}  \\ p^2 \bb{D}^1_{\mathbf{Q}_p}  & 1+p^2 \bb{D}^1_{\mathbf{Q}_p}  \end{array} \right)$,  then  
\[
g \bb{G}_0 g^{-1} = \left( \begin{array}{cc} 1+ p^2 \bb{D}^1_{\mathbf{Q}_p} & p^3 \bb{D}^1_{\mathbf{Q}_p}  \\ p \bb{D}^1_{\mathbf{Q}_p} & 1+p^2 \bb{D}^1_{\mathbf{Q}_p}   \end{array}  \right).
\]
Therefore,  if $V$ is a locally analytic representation of $G$ and $v\in V$ is a $\bb{G}_0$-analytic vector,  $gv$ is a $g\bb{G}_0 g^{-1}$-analytic vector,  but   $\bb{G}_0 \neq g \bb{G}_0 g^{-1}$.

Notice however,  that the tensor product $ \n{D}^{la}(G,  K):=K_{\sol}[G]\otimes_{K_{\sol}[G_0]} \n{D}^{la}(G_0,K)$ is an algebra since $\n{D}^{la}(G_0,K)$ admits an action of $G$ by conjugation,  this is nothing but the algebra of locally analytic distributions of $G$. 
\end{remark}

\begin{remark}
Notice that Theorem \ref{TheoMain} implies that the tensor product over $K$ of two $\bb{G}^{(h^+)}$-analytic representations is still $\bb{G}^{(h^+)}$-analytic.  Indeed,  if $V$ and $W$ are $\bb{G}^{(h^+)}$-analytic,  then they are modules over $\n{D}^{(h^+)}(G,  K)$ and so it is its tensor product $V\otimes_{K_{\sol}}^L W$,  which in turn implies that $V\otimes_{K_{\sol}}^L W$ is $\bb{G}^{(h^+)}$-analytic.  A similar result holds for locally analytic representations by writing them as colimits of their $\bb{G}^{(h^+)}$-analytic vectors.  
\end{remark}

The diagram \eqref{EqDiagramDistributions} relating the distribution algebras gives rise to the following  fully faithful forgetful functors between derived categories for $h'>h$
\[
D(K_{\sol}[G]) \leftarrow D(\n{D}^{la}(G,K)) \leftarrow D(\n{D}^{(h'^+)}(G,K)) \leftarrow D(\n{D}^{(h^+)}(G,K)).
\]
Since the forgetful functors preserve limits and colimits,  they admit  left and a right adjoints.   For example,  consider the forgetful functor $F:D(\n{D}^{(h^+)}(G,K))\to D(K_{\sol}[G])$, and let $V\in D(K_{\sol}[G])$ and $W\in D(\n{D}^{(h^+)}(G,K))$,  then we have that 
\begin{eqnarray*}
R\iHom_{K_{\sol}[G]}( V ,  F(W)) & = & R\iHom_{\n{D}^{(h^+)}(G,K)}(\n{D}^{(h^+)}(G,K)\otimes^L_{K_{\sol}[G]} V,  W ) \\
R\iHom_{K_{\sol}[G]}(F(W),  V) & = &   R\iHom_{\n{D}^{(h^+)}(G,K)}(W, R\iHom_{K_{\sol}[G] }(\n{D}^{(h^+)}(G,K),  V) )\\
& = &   R\iHom_{\n{D}^{(h^+)}(G,K)}(W,  V^{R\bb{G}^{(h^+)}-la}).  
\end{eqnarray*}

\begin{remark}
By definition $D(K_{\sol}[G])$ is the derived category of solid representations of $G$. By Theorem \ref{TheoMain}, the category $D(\n{D}^{(h^+)}(G,K))$ is the derived category of  $\bb{G}^{(h^+)}$-analytic representations of $G$. We point out that $D(\n{D}^{la}(G,K))$ is not the derived category of locally analytic representations to $G$. Indeed, one has for example that $\n{D}^{la}(G,K)$ is not a locally analytic representation (cf. the discussion just after Definition \ref{DefLocAn} below).
\end{remark}

\subsection{Locally analytic representations}
\label{SubsectionLocallyAnalytic}

We finish this section with some applications of Theorem \ref{TheoMain} to the theory of locally analytic representations.  Let us begin with the definition of the locally analytic vectors. 
\begin{definition} \leavevmode \label{DefLocAn}
\begin{enumerate} 
    \item  Let $V$ be a solid $K_{\sol}[G]$-module, the  space of locally analytic vectors of $V$ is the solid $K_{\sol}[G]$-module 
    \[
    V^{la}:= \varinjlim_{h\to \infty} V^{\bb{G}^{(h)}-an}= \varinjlim_{h\to \infty} V^{\bb{G}^{(h^+)}-an}.   
    \]
    We say that $V$ is locally analytic if the natural map $V^{la}\to V$ is an isomorphism.
    
    \item Let $C\in D(K_{\sol}[G])$, the derived locally analytic vectors of $C$ is the complex
    \[
    C^{Rla}:= \underset{h\to \infty}{\hocolim}\; C^{R\bb{G}^{(h)}-an}= \underset{h\to \infty}{\hocolim}\; C^{R\bb{G}^{(h^+)}-an}.
    \]
    We say that $C$ is derived locally analytic if the natural map $C^{Rla}\to C$ is a quasi-isomorphism. 
\end{enumerate}
\end{definition}

In \S \ref{SubsubSectionDistG} we defined the algebra of locally analytic distributions of $G$ as  the Fr\'echet algebra 
\[
\n{D}^{la}(G,K)= \varprojlim_{h\to \infty} \n{D}^{(h)}(G,K)=  \varprojlim_{h\to \infty} \n{D}^{(h^+)}(G,K). 
\]
Since a locally analytic representation $V$ is a (homotopy) colimit of $\bb{G}^{(h^+)}-$analytic  representations, Theorem \ref{TheoMain} implies that $V$ is naturally a $\n{D}^{la}(G,K)$-module.  Furthermore, this structure is unique as $\n{D}^{la}(G,K)\otimes^L_{K_\sol[G]} \n{D}^{la}(G,K)=\n{D}^{la}(G,K)$, see Corollary \ref{CoroTechnicalDhplus}.   Nevertheless, not all the $\n{D}^{la}(G,K)$-modules are locally analytic representations of $G$, e.g. $\n{D}^{la}(G,K)$ is not a locally analytic representation as it cannot be written as a colimit of $\n{D}^{(h^+)}(G,K)$-modules. Indeed, if $\n{D}^{la}(G,K)$ was locally analytic then $1\in \n{D}^{la}(G,K)$ would be analytic for certain group $\bb{G}^{(h^+)}$, this would provide a section of the map $\n{D}^{la}(G,K)\to \n{D}^{(h^+)}(G,K)$ which is a contradiction: the map $\n{D}^{la}(G,K) \to \n{D}^{(h^+)}(G,K)$ is   injective by (a variant of) Lemma \ref{Lemmanoqs} so it would be an isomorphism,  by taking duals we would have that $C^{la}(G,K)\cong C(\bb{G}^{(h^+)},K)$ which is impossible since there are locally analytic functions which are not $\bb{G}^{(h^+)}$-analytic.     In the following propositions  we  give some conditions for a $\n{D}^{la}(G,K)$-module to be a locally analytic representation of $G$. 

\begin{prop}
\label{PropBanachlocally}
Let $V$ be a Banach  $K_{\sol}[G]$-module. The following   are equivalent.
\begin{enumerate}
    \item  the $K_\sol[G]$-module structure of $V$ extends to   $\n{D}^{la}(G,K)$.
    
    \item the $K_\sol[G]$-module structure of $V^{\vee}$  extends to   $\n{D}^{la}(G,K)$.
    
    \item $V$ is $\bb{G}^{(h^+)}$-analytic for some $h\geq 0$. 
\end{enumerate}
\proof
Formally,  (1) and  (2) are equivalent as the dual of a solid $\n{D}^{la}(G,K)$ module is naturally a solid $\n{D}^{la}(G,K)$-module,  (3) implies (1) is clear from the previous discussion. Let us show that (1) implies (3).  Suppose that $V$ is a $\n{D}^{la}(G,K)$-module. Consider the multiplication map $m_V:\n{D}^{la}(G,K){\otimes_{K_{\sol}}} V \to V$. As $V$ is Banach and $\n{D}^{la}(G,K){\otimes_{K_{\sol}} }V$ is a Fr\'echet space,  Lemma \ref{lemmaFrechetToNormed} and Corollary \ref{coroFLScompact} imply that there exists $h \geq 0$ such that  $m_V$ factors as 
\[
\n{D}^{la}(G, K){\otimes_{K_{\sol}}} V \to  \n{D}^{(h)}(G,K)^{B}\otimes_{K_{\sol}} V   \to   V.
\]
  It is immediate to check that this endows $V$ with an structure of $\n{D}^{(h^+)}(G,K)$-module for $h'>h$. By Theorem \ref{TheoMain} $V$ is $\bb{G}^{(h'^+)}$-analytic.    
\endproof
\end{prop}

\begin{prop}[Schneider-Teitelbaum]
\label{PropSchTeillocally}
Let $V $ be an $LS$ space of compact type, cf. Definition \ref{defiBanachificationSmith}. The following are equivalent 
\begin{enumerate}
    \item  the $K_\sol[G]$-module structure of  $V$ extends to   $\n{D}^{la}(G,K)$-module. 
    
    \item  the $K_\sol[G]$-module structure of  $V^{\vee}$ extends to  $\n{D}^{la}(G,K)$-module. 
    
    \item $V$ is a locally analytic representation of $G$. 
    
\end{enumerate}
\proof
By Theorem \ref{theorem:duality},  (1) and (2) are equivalent. It is also clear that (3) implies (1).  Suppose that $V$ is a $\n{D}^{la}(G,K)$-module which is an $LS$ space of compact type. The multiplication map $m_V:  \n{D}^{la}(G,K)\otimes_{K_{\sol}} V \to V$ gives an element of $\Hom(\n{D}^{la}(G,K)\otimes_{K_{\sol}}V, V)$. Let $V= \varinjlim_{n} V_n$ be a presentation as a colimit of Smith spaces by injective transition maps, and let $V_n^{B}$ be the underlying Banach space  of $V_n$ (cf. Definition \ref{defiBanachificationSmith}). As $V$ is of compact type we have $V=\varinjlim_{n} V_n^{B}$.  Therefore 
\begin{eqnarray*}
\underline{\Hom}_K(\n{D}^{la}(G,K)\otimes_{K_{\sol}} V, V ) & = & \varprojlim_n \underline{\Hom}_K(\n{D}^{la}(G,K)\otimes_{K_{\sol}}V_n^B, V ) \\
    & = & \varprojlim_n  \underline{\Hom}_K(\n{D}^{la}(G,K), (V_n^B)^{\vee}{\otimes_{K_{\sol}}} V ) \\
    & = & \varprojlim_n C^{la}(G,K)\otimes_{K_{\sol}} (V_n^{B})^{\vee}\otimes_{K_{\sol}} V  \\
    & = &\varprojlim_n \varinjlim_m C^{la}(G,K) \otimes_{K_{\sol}} (V_n^B)^\vee \otimes_{K_{\sol}} V_m^B,
\end{eqnarray*}
where the first equality is formal, the second follows from the fact that $V$ is nuclear, and  the third equality follows from Theorem \ref{theorem:duality}.  This shows that given $n\in \bb{N}$ there is $m\in \bb{N}$ such that  the map $ m_{V}: \n{D}^{la}(G,K)\otimes_{K_{\sol}} V_n^B\to  V$ factors through $m_{V}:  \n{D}^{la}(G,K)\otimes_{K_{\sol}} V_n^B \to V_{m}^B$. By lemma \ref{lemmaFrechetToNormed} there exists $h\geq 0$ such that $m_{V}$ factors as $\n{D}^{la}(G,K)\otimes_{K_{\sol}} V_n^B \to \n{D}^{(h)}(G,K)\otimes_{K_{\sol}} V_n^B \to V_m^{B}$. Equivalently, there is  $h\geq 0$ (maybe different) such that $m_V$ factors as   $\n{D}^{la}(G,K)\otimes_{K_{\sol}} V_n \to \n{D}^{(h)}(G,K)\otimes_{K_{\sol}} V_n \to V_m$.  Let $V_n'$ be the image of   $\n{D}^{(h)}(G,K)\otimes_{K_{\sol}} V_n \to V_m$, it is a Smith space endowed with an action of $\n{D}^{(h)}(G,K)$ extending the one of $\n{D}^{la}(G,K)$.  It is immediate to see that  $V= \varinjlim_{n} V_n'$, this shows that $V$ is written as a colimit of $\n{D}^{(h'^+)}(G,K)$-modules (for $h'>h$), which implies that it is a locally analytic representation of $G$ by Theorem \ref{TheoMain}. 
\endproof
\end{prop}

One may wonder whether the previous propositions can be extended  to the bounded derived category. Under the assumption of Conjecture \ref{ConjectureLocallyanalytic} we can improve Propositions \ref{PropBanachlocally} and \ref{PropSchTeillocally} as follows:

\begin{proposition}
\label{PropUseConjecture}
 Suppose that Conjecture \ref{ConjectureLocallyanalytic} holds.  Let $C\in D(\n{D}^{la}(G,K)_{\sol})^b$ be a bounded solid $\n{D}^{la}(G,K)$-module. Assume that  either \begin{itemize}
\item[(a)] $C$ is quasi-isomorphic to a bounded complex of $K$-Banach spaces as  a $K_{\sol}$-complex. 

\item[(b)] $C$ is quasi-isomorphic to a bounded complex of $LS$ spaces of compact type as a $K_{\sol}$-complex.
\end{itemize}
Then $C$ is a derived locally analytic representation of $G$.   In particular there is a derived duality between locally analytic complexes quasi-isomorphic to bounded complexes of $LS$ spaces of compact type, and $\n{D}^{la}(G,K)$-complexes quasi-isomorphic to bounded complexes of Frech\'et spaces of compact type. 
\proof
We will prove in Corollary \ref{CoroTechnicalDhplus} that \begin{equation}
\label{eqDlatensor}
\n{D}^{la}(G,K)\otimes_{K_{\sol}[G]}^{L} \n{D}^{la}(G,K)= \n{D}^{la}(G,K).
\end{equation}
Let $C$ be a complex in $D(\n{D}^{la}(G,K))$ which is quasi-isomorphic to a bounded complex of Banach spaces as $K_{\sol}$-complex.  By Corollary \ref{CoroDualityDerived} (1) we have 
\[
R\underline{\Hom}_K(\n{D}^{la}(G,K), C)= \underset{h\to \infty}{\hocolim}\; R\underline{\Hom}_K(\n{D}^{(h^+)}(G,K), C).
\]
Taking $K_{\sol}[G]$-invariants,  using that $K$ is a compact $K_\sol[G]$-module (cf. Theorem \ref{LazardSerre} below) and hence that it  commutes with filtered colimits, one gets that 
\[
R\underline{\Hom}_{K_{\sol}[G]}(\n{D}^{la}(G,K), C) = \underset{h\to \infty}{\hocolim}\; R\underline{\Hom}_{K_{\sol}[G]}(\n{D}^{(h^+)}(G,K), C).
\]
But (\ref{eqDlatensor}) implies that the LHS is equal to $C$ while Theorem \ref{TheoMain} implies that the RHS is equal to $\hocolim_{h} C^{R \bb{G}^{(h^+)}}$. This shows that $C$ is derived locally analytic.   

Now suppose that $C$ is quasi-isomorphic to a bounded complex of  $LS$ spaces of compact type.  Then from Corollary \ref{CoroDualityDerived} (3) and (4) one has 
\[
R\underline{\Hom}_K(\n{D}^{la}(G,K),  C)= C^{la}(G,K)\otimes^{L}_{K_{\sol}} C= \underset{h\to \infty}{\hocolim}\; C(\bb{G}^{(h)},K)_{\sol}\otimes^L_{K_{\sol}} C.
\]
Taking $K_{\sol}[G]$-invariants one gets again by (\ref{eqDlatensor}) that $C= C^{Rla}$, i.e. that $C$ is a derived locally analytic representation of $G$.  
\endproof
\end{proposition}
\begin{remark}
In the situation (a) of the previous proposition,  one can show in addition that $C$ is  derived $\bb{G}^{(h^+)}$-analytic for some $h>0$. Indeed, it is enough to prove that for all $n\in \Z$ there is $h>0$ such that $H^n(C)$ is derived  $\bb{G}^{(h^+)}$-analytic. As $C$ is bounded, we are left to show that if $V$ is a $\n{D}^{la}(G,K)$-module which is a quotient of two $K$-Banach spaces (not necessarily $\n{D}^{la}(G,K)$-modules),   then $V$ is already a $\n{D}^{(h^+)}(G,K)$-module for some $h>0$.  Let $m_{V}: \n{D}^{la}(G,K)\otimes_{K_\sol} V \to V$ be the multiplication map, by Corollary \ref{CoroDualityDerived} (4)   one has that 
\[
R\underline{\Hom}_{K}(\n{D}^{la}(G,K)\otimes_{K_\sol} V, V) = \underset{h\to \infty}{\hocolim} R\underline{\Hom}_{K}(\n{D}(\bb{G}^{(h)}, K)\otimes_{K_\sol}V , V).
\]
Thus, $m_V$ factors as $m_V: \n{D}(\bb{G}^{(h)},K)\otimes_{K_\sol} V \to V$ for some $h>0$. After taking $h'>h$ one shows that the map $\n{D}(\bb{G}^{(h'^+)},K) \otimes_{K_\sol}  V\to  V$ is in fact an action of $\n{D}(\bb{G}^{(h'^+)},K)$, proving that $V$ is derived $\bb{G}^{(h^+)}$-analytic as desired. 
\end{remark}

\subsection{Admissible representations}

Before writing down the statements of the cohomological comparison results, let us show how the theory developed till now  together with some algebraic facts about the Iwasawa and the distribution algebras, provide a description of the locally analytic vectors of an admissible representation in terms of its dual. All results in this sections were already known (\cite{SchTeitDist}, \cite{LuePan}).

We recall the following important results of the classical Iwasawa and distribution algebras
\begin{theorem}[Lazard, Schneider-Teitelbaum]
\label{TheoLSTIwasawa}
The following holds 
\begin{enumerate}
    \item The (classical) Iwasawa  algebra $K_{\sol}[G](*)$ is a (left and right) noetherian integral domain. 
    \item The (classical) distribution algebras $\n{D}^{(h^+)}(G,K)(*)$ are flat over $\O_{K,\sol}[G](*)$ algebraically.  
    \item Let $h'>h$, then $\n{D}^{(h'^+)}(G,K)(*)$ is a flat algebra over $\n{D}^{(h^+)}(G,K)(*)$ algebraically. 
    \item The (classical) locally analytic distribution algebra $\n{D}^{la}(G,K)(*)$ is faithfully flat over $K_{\sol}[G](*)$. 
\end{enumerate}
\proof
Part (1) is \cite[Remark 4.6]{SchTeitDist}, part (2) is \cite[Proposition  4.7]{SchTeitDist},  part (3) is \cite[Theorem 4.9]{SchTeitDist} and part (4) is \cite[Theorem 4.11]{SchTeitDist}. 
\endproof
\end{theorem}

\begin{definition}
A solid Banach representation $V$ of $G$ is admissible if its dual is a finite module over the Iwasawa algebra.  Equivalently, if $V$  admits an equivariant  closed immersion into a finite direct sum of $\underline{\Cont}(G,K)$. 
\end{definition}

\begin{remark} \label{RemarkAdmissible}
By (1) of Theorem \ref{TheoLSTIwasawa},  any  finite  $K_{\sol}[G](*)$-module $W$   is of finite presentation. The module $W$ has a natural Hausdorff topology given by the quotient topology of any presentation of $W$ as a quotient of finite free $K_{\sol}[G](*)$-modules (equipped with their canonical topology induced from $K_{\sol}[G](*)_{\mathrm{top}}$),  cf.  \cite[Proposition 3.1]{SchTeitBanachIwasawa}. Moreover,  any $K_{\sol}[G](*)$-equivariant map between finite  $K_{\sol}[G](*)$-modules is continuous for this topology. If $P_1\to P_0\to W\to 0$  is a presentation of $W$,  one has an exact sequence of solid $K$-vector spaces $ \underline{P_1}\to \underline{P_0}\to \underline{W}\to 0$,  where the $\underline{P_i}$ are   finite free $K_{\sol}[G]$-modules.  In other words, the category of finite $K_{\sol}[G](*)$-modules embeds fully faithfully in the category of finite $K_{\sol}[G]$-modules.  This shows that a solid Banach representation of $G$ is admissible if and only if $V(*)_{\mathrm{top}}$ is an admissible representation in the classical sense.  
\end{remark}

\begin{prop}
\label{PropAdmissible}
Let $V$ be a Banach $G$-representation, then 
\[
V^{R\bb{G}^{(h^+)}-an} = R\underline{\Hom}_K(\n{D}^{(h^+)}(G,K)\otimes^L_{K_\sol[G]} V^\vee, K).
\]
In particular, if $V$ is admissible then
\[ V^{R\bb{G}^{(h^+)}-an}=V^{\bb{G}^{(h^+)}-an}=\underline{\Hom}_K( \n{D}^{(h^+)}(G,K)\otimes_{K_\sol[G]} V^\vee,K). \]
Furthermore, $V^{la}= \underline{\Hom}_K(\n{D}^{la}(G,K)\otimes_{K_{\sol}[G]}V^\vee, K)$. 
\proof
By Theorem \ref{TheoMain} we have 
\begin{eqnarray*}
V^{R\bb{G}^{(h^+)}-an} & = &  R\underline{\Hom}_{K[G]}(\n{D}^{(h^+)}(G,K), V) \\
            & =& R\underline{\Hom}_{K[G]}(\n{D}^{(h^+)}(G,K),  R\underline{\Hom}_K(V^\vee, K)) \\
            & = & R\underline{\Hom}_{K}(\n{D}^{(h^+)}(G,K)\otimes^L_{K_\sol[G]}V^\vee, K ),
\end{eqnarray*}
this proves the first statement.
Moreover, if $V$ is admissible then $V^\vee$ is a finite $K_{\sol}[G]$-module.  By flatness of the distribution algebra one gets that $\n{D}^{(h^+)}(G,K)\otimes^L_{K_{\sol}[G]} V^\vee= \n{D}^{(h^+)}(G,K)\otimes_{K_{\sol}[G]} V^\vee$ is concentrated in degree $0$ (we warn  that the flatness is only algebraic, and that we use the fact that $V^\vee$ is a finite module over the Iwasawa algebra), this implies the second claim. 

We now prove the last statement. Writing $V^{\vee}$ as a quotient  $K_{\sol}[G]^n\xrightarrow{f} K_{\sol}[G]^m\to V^{\vee}\to0$,    we have  an exact sequence 
\begin{equation}
\label{eqTensorVDla}
\n{D}^{(h^+)}(G,K)^n \xrightarrow{1\otimes f} \n{D}^{(h^+)}(G,K)^{m} \to   \n{D}^{(h^+)}(G,K)\otimes_{K_{\sol}[G]} V^{\vee}\to 0.
\end{equation}
Taking inverse limits in \eqref{eqTensorVDla} one obtains an exact sequence (by topological Mittag-Leffler, cf. Lemma \ref{LemmaML})
\begin{equation} \label{eqeaea}
\n{D}^{la}(G,K)^n \xrightarrow{1\otimes f} \n{D}^{la}(G,K)^{m} \to  \varprojlim_{h\to \infty } (\n{D}^{(h^+)}(G,K)\otimes_{K_{\sol}[G]} V^{\vee}) \to 0,
\end{equation}
proving that  $\varprojlim_{h\to \infty } (\n{D}^{(h^+)}(G,K)\otimes_{K_{\sol}[G]} V^{\vee})= \n{D}^{la}(G,K)\otimes_{K_{\sol}[G]} V^{\vee}$. In particular,  it also follows from \eqref{eqeaea} that $\n{D}^{la}(G,K)\otimes_{K_{\sol}[G]} V^{\vee}$  is a Fr\'echet space of compact type since it is a quotient of such. Moreover,
\begin{eqnarray*}
\n{D}^{la}(G,K)\otimes_{K_{\sol}[G]} V^{\vee} & = &  \varprojlim_{h\to \infty}( \n{D}^{(h)}(G,K)^B \otimes_{K_{\sol}[G]}V^{\vee}) 
\end{eqnarray*}
(recall that we have arrows  $ \n{D}^{(h')}(G,K)\to \n{D}^{(h^+)}(G,K) \to \n{D}^{(h)}(G,K)$ for $h'>h$,  which induce  arrows $\n{D}^{(h')}(G,K)^B\to \n{D}^{(h^+)}(G,K) \to \n{D}^{(h)}(G,K)^B$). Using all this,  one obtains
\begin{eqnarray*}
\underline{\Hom}_{K}(\n{D}^{la}(G,K)\otimes_{K_\sol[G]}V^\vee, K ) &=&  \underline{\Hom}_{K}(\varprojlim_{h\to \infty}( \n{D}^{(h)}(G,K)^B\otimes_{K_\sol[G]}V^\vee), K )  \\
& = &    \varinjlim_{h\to \infty}  
\underline{\Hom}_{K}(\n{D}^{(h)}(G,K)^B\otimes_{K_\sol[G]}V^\vee, K )  \\
&=& \varinjlim_{h\to \infty}  
\underline{\Hom}_{K}(\n{D}^{(h^+)}(G,K)\otimes_{K_\sol[G]}V^\vee, K )  \\
& = &  \varinjlim_{h\to \infty} V^{\bb{G}^{(h^+)}-an} \\
& = & V^{la},
\end{eqnarray*}
where the second equality follows from   Lemma \ref{lemmaFrechetToNormed}.
\endproof

\end{prop}

\begin{corollary}
\label{CoroadmissibleDense}
The functor $V \mapsto V^{la}$ is exact  on the category of admissible Banach $G$-representations. Moreover  $V^{la}(*)_{\mathrm{top}}\subset V(*)_{\mathrm{top}}$ is a dense subspace. 
\proof

First,  note that  the category of admissible representations is an abelian category (being the dual of the category of finite $K_{\sol}[G]$-modules,  cf.  Remark \ref{RemarkAdmissible}), and  that a closed subrepresentation of an admissible representation is admissible (being the dual of a quotient of a finite $K_{\sol}[G]$-module). 

Let us prove  that the functor $V \mapsto V^{la}$ is exact for admissible Banach $G$-representations. Indeed, by Proposition \ref{PropAdmissible} we have
\[ V^{la} = \iHom_{K}(\n{D}^{la}(G, K) \otimes_{K_\sol[G]} V^\vee, K). \]
Then, the exactness follows by the antiequivalence $V\mapsto V^{\vee}$ between admissible representations and finite $K_{\sol}[G]$-modules,  the faithful flatness of $\n{D}^{la}(G, K)$ over $K_{\sol}[G]$,  and the duality Theorem \ref{theorem:duality}.  Furthermore,  this also shows that $V^{la}\neq 0$ provided $V\neq 0$.

Finally, let $V' \subseteq V$ be the closure of $V^{la}$ (i.e. the condensed object associated to the closure of $V^{la}(*)_{\mathrm{top}}$ in $V(*)_{\mathrm{top}}$). Then $V'$ and $Q = V / V'$ are admissible representations and we have an exact sequence $0\to (V')^{la} \to V^{la} \to Q^{la}\to 0$. The density of  $V^{la}(*)_{\mathrm{top}}$ in  $V(*)_{\mathrm{top}}$  follows since the functor $V \mapsto V^{la}$ is non zero.
\endproof
\end{corollary}

\section{Cohomology}
\label{SecCohomology}

In this last chapter we present our main applications to  group cohomology.  We obtain in particular
\begin{enumerate}
    \item An isomorphism between the continuous cohomology  and  the cohomology of the derived locally analytic vectors for solid representations. This can be seen as a $p$-adic version of theorems of  P. Blanc and  G. D. Mostow for real Lie groups, cf.  \cite{Blanc, Mostow}.
    \item A comparison theorem, for a locally analytic $G$-representation $V$,  between its continuous and its locally analytic cohomology. This generalises the classical result of Lazard \cite[Th\'eor\`eme V.2.3.10]{Lazard} for finite dimensional $\qp$-representations to arbitrary solid $K$-vector spaces.
    \item A comparison theorem between locally analytic cohomology and Lie algebra cohomology. This recovers and generalises a result of Tamme \cite{Tamme} (cf. also \cite{HuberKingsNaumann} and \cite{Lechner}), which in turns was a generalisation of the other main result of Lazard (\cite[Th\'eor\`eme V.2.4.9]{Lazard}) for finite dimensional representations.
\end{enumerate}


The proof of (1) is an immediate consequence of our main Theorem \ref{TheoMain} and a result of Kohlhaase (Theorem \ref{propKohlhaase}).

The key input for the comparison result in (2) is the existence of a finite free resolution of the trivial representation, as a module over the Iwasawa algebra of a small neighbourhood of $1$ in $G$.  This is an application of a lemma of Serre used by Lazard in \cite[Definition  V.2.2.2]{Lazard}. We shall also need a version of the  lemma  proved by   Koohlhase for distribution algebras, cf. \cite[Theorem 4.4]{Kohlhaase}.    Once one has this lemma at hand, the proof of (2) is rather formal using the machinery developed throughout this text.  

Finally, to show (3) we  follow the proof of Tamme  constructing a resolution of the trivial representation in terms of the de Rham complex of the analytic groups. Then, we apply the same formal computation as before. 

\subsection{Continuous, analytic and locally analytic cohomology}
\label{subsecConAnCoho}
In this section we define different cohomology groups,   they correspond to  continuous, analytic and locally analytic cohomology in the literature. Indeed, using the bar resolutions and Corollary \ref{PropsixFunctorTate}, one verifies that  whenever $V$ is a solid representation coming from a ``classical space''  (i.e. a Banach, Fr\'echet, $LB$ or $LF$ space) our definitions coincide with the usual ones.

In the following we will use the conventions of Section  \ref{subsecFuncDist}. In particular, we fix a compact $p$-adic Lie group $G$ and an open normal uniform pro-$p$-group $G_0\subset G$,  we denote the $h$-analytic neighbourhood of $G$ as $\bb{G}^{(h)}$, and define its  open $h$-analytic neighbourhood as  $\bb{G}^{(h^+)}= \bigcup_{h'>h} \bb{G}^{(h')}$.  Recall from Lemma \ref{LemmaDefSolidGmod} that  a solid $\O_K$-module $V$  is an $\O_K$-linear $G$-representation  if and only if it is a  module over the Iwasawa algebra $\O_{K, \sol}[G]$. If $V$ is in addition a $\mathbb{G}^{(h^+)}$-analytic representation, by the  main Theorem \ref{TheoMain}, $V$ is naturally equipped with a $\Dist^{(h^+)}(G, K)$-module structure. Let $\f{g}= \mathrm{Lie} \, G$ be the Lie algebra of $G$ and $U(\f{g})$ its universal enveloping algebra. There is a natural map of solid $K$-algebras $U(\f{g}) \to \n{D}^{(h^+)}(G,K)$ given by derivations of $\f{g}$ on $C(\bb{G}^{(h^+)},K)$ (cf. the discussion after \cite[Proposition 2.3]{SchTeitGl2}).  In particular,  a solid (locally) analytic  representation of $G$ has a natural action of $U(\f{g})$.   We let $D(U(\f{g})_{\sol})$ denote the derived category of solid $U(\f{g})$-modules.  

\begin{definition} \leavevmode
\begin{enumerate}
\item Let $C \in D(\O_{K,\sol}[G])$. We define the continuous group cohomology of $C$ as
$  R\iHom_{\O_{K, \sol}[G]}(\O_K, C) $.
\item Let  $C\in D(\O_{K,\sol}[G])$ be a derived $\mathbb{G}^{(h^+)}$-analytic representation. We define the $\mathbb{G}^{(h^+)}$-analytic cohomology of $C$ as
$  R\iHom_{\Dist^{(h^+)}(G, K)}(K, C)$. 
\item Let $C\in D(\O_{K,\sol}[G])$ be a derived locally analytic representation. We define the locally analytic cohomology of $C$ as
$  \underset{h\to\infty}{\hocolim} R\iHom_{\Dist^{(h^+)}(G, K)}(K, C^{R\bb{G}^{(h^+)}-an})$.

\item Let $C \in D(U(\f{g})_{\sol})$. We define the Lie algebra cohomology of $C$ as $R\iHom_{U(\f{g})}(K,C)$.

\end{enumerate}
\end{definition}

\begin{lemma} \leavevmode 
\begin{enumerate} 

\item There is a solid  projective resolution of the trivial representation
\begin{equation}
    \label{eqbarResoludionsolid}
     \cdots \to  K_{\sol}[G^{n+1}] \xrightarrow{d_{n}} K_{\sol}[G^{n}] \to   \cdots \to  K_{\sol}[G] \to K\to 0,
\end{equation}
with differentials induced by the formula $d_n( (g_0, g_1, \ldots, g_{n}))= \sum_{i=0}^{n} (-1)^{i} (g_0, \ldots, \widehat{g_i},  \ldots, g_{n+1})$. In particular, if $V$ is a complete locally convex $G$-representation, one has that 
\[
\Ext^{i}_{K_{\sol}[G]}(K, \underline{V}) = H^i_{cont}(G,V).  
\]

\item  Let $h > 0$. The complex \eqref{eqbarResoludionsolid} has a unique extension to a projective resolution of   $\n{D}^{(h)}(G,K)$-modules 
\begin{equation}
\label{eqResolutionDistbar}
\cdots  \to \n{D}^{(h)}(G^{n+1},K) \to  \cdots \to \n{D}^{(h)}(G,K) \to   K \to 0,
\end{equation}
where  $\n{D}^{(h)}(G^{n+1},K) = \n{D}^{(h)}(G,K)\otimes_{K_{\sol}} \cdots \otimes_{K_{\sol}} \n{D}^{(h)}(G,K)$ ($(n+1)$-times).  In particular,  if $V$ is a classical $LB$ locally analytic  representation,  we have that 
\[
\varinjlim_{h\to \infty} \Ext^{i}_{\n{D}^{(h^+)}(G, K)}(K, \underline{V}^{\bb{G}^{(h^+)-an}})  = \varinjlim_{h\to \infty} \Ext^{i}_{\n{D}^{(h)}(G, K)}(K, \underline{V}^{\bb{G}^{(h)-an}}) = H^{i}_{la}(G, V) ,
\]
where $H^{i}_{la}(G, V)$ is the cohomology of locally analytic cochains of $G$ with values in $V$.  

\end{enumerate}
\end{lemma}

\begin{proof}
Let $(G^{n+1})_{n\in \N}\xrightarrow{\epsilon} *$ be   the augmented simplicial space with boundary maps $d_i^{n}:  G^{n+1}\to G^{n}$ given by $d^{n}_{i}(g_0,  \ldots,  g_{n})= (g_0,  \ldots,  \widehat{g}_i, \ldots, g_{n})$ for $n\geq 1$ and  $0\leq i\leq n$,  degeneracy maps $s_i^{n}: G^{n+1}\to G^{n+2}$  given by $s^n_i(g_0,  \ldots, g_n)= (g_0, \ldots, g_{i-1},1, g_{i+1},  \ldots, g_n)$ for $n\geq 0$ and  $0\leq i\leq n+1$, and augmentation $\epsilon: G \to  *$.     Taking the total complex of the augmented simplicial solid $K$-vector space $(K_{\sol}[G^{n+1}])_{n\in \N}\xrightarrow{\epsilon} K$ one obtains the    bar resolution with homogeneous cochains of $K$
\begin{equation}
\label{eqBarResolution1}
 \ldots \xrightarrow{d_2} K_{\sol}[G^2] \xrightarrow{d_1}  K_{\sol}[G] \xrightarrow{\epsilon} K \to 0,
\end{equation}
where  $d_n$ is the unique map extending $d_n(g_0,\ldots,g_n )=\sum_{i=0}^n (-1)^{i} (g_0, \cdots, \widehat{g_i},\cdots, g_n)$, 
as well as homotopies $h_n: K_{\sol}[G^{n+1}]\to K_{\sol}[G^{n+2}]$ (for $i\geq -1$), defined by $h_n(g_0, \ldots, g_n) \mapsto (1, g_0, \ldots, g_n)$, between  the identity of  \eqref{eqBarResolution1} and $0$. Note in addition that,  since $G$ is profinite,  the terms $K_{\sol}[G^{n+1}]$ are compact projective $K_{\sol}[G]$-modules,  where  $G$ acts diagonally. Moreover, the map $  g_0 \otimes  ( g_1, \ldots, g_n) \mapsto (g_0,g_0 g_1, \ldots, g_0 g_1 \cdots  g_n) $ induces an isomorphism $ K_{\sol}[G] \otimes_{K_\sol} (K_{\sol}[G^{n}])_0 \cong K_{\sol}[G^{n+1}] $, where the second term in the tensor product is equipped with the trivial action of $G$. The second statement then follows from the fact that $\underline{V}$ is a solid $K$-vector space if $V$ is  complete locally convex, and that
\[
\Hom_{K_{\sol}[G]}(K_{\sol}[G^{n+1}],\underline{V}) \cong \Hom_{K_{\sol}}((K_{\sol}[G^{n}])_0,\underline{V}) = \Cont(G^{n},V),
\]
where the last equality follows from adjunction (cf. Proposition \ref{propFunctorTopCond}(1)). Note that the differential on the right hand side translates into $d_n f (g_1, \ldots, g_{n}) \mapsto g_1 f(g_2, \ldots g_n) + \sum_{i = 1}^{n-1} (-1)^i f(g_1, \ldots, g_i g_{i+1}, \ldots, g_n)+ (-1)^n f(g_1, \ldots, g_{n-1})$, for $f: G^n \to V$, as usual.

For the second point we reason similarly. Consider the augmented  simplicial space $(\bb{G}^{(h),n+1})_{n\in \N}\to  *$ with boundary maps  and degeneracy maps   defined as above.     Taking the total complex of the  associated distribution spaces $(\n{D}^{(h)}(G^{n+1},K))_{n\in \N}\to K$ of $(\bb{G}^{(h),n+1})_{n\in \N} \to *$,  one obtains that   \eqref{eqResolutionDistbar} is a resolution of $K$. Notice that the term $\n{D}^{(h)}(G^{n+1},K) $ is a projective $\n{D}^{(h)}(G, K)$-module since it is isomorphic to $\n{D}^{(h)}(G,K) \otimes_{K_{\sol}} (\n{D}^{(h)}(G^n,K))_0$ where the second term in the tensor product has the trivial action of $G$,  and  $\n{D}^{(h)}(G^n,K)$ is a Smith $K$-vector space.    Furthermore, Corollary  \ref{PropsixFunctorTate} implies that for any solid $K$-vector space $W$,  
\[
C(\bb{G}^{(h),n}, W)= C(\bb{G}^{(h),n},K)_{\sol}\otimes_{K_{\sol}} W = \iHom_{K}(\n{D}^{(h)}(G^n, K),  W).
\]
This implies that 
\[
\Ext^i_{\n{D}^{(h)}(G,K)}(K, V^{\bb{G}^{(h)}-an})= H^i( C(\bb{G}^{(h),\bullet } , V^{\bb{G}^{(h)}-an})),
\]
where  $C(\bb{G}^{(h),\bullet } , V^{\bb{G}^{(h)}-an}))$ is the complex of $\bb{G}^{(h)}$-analytic cochains of $V^{\bb{G}^{(h)}-an}$.  Finally, taking colimits as $h\to \infty$ one obtains 
\[
\varinjlim_{h\to \infty}  \Ext^i_{\n{D}^{(h^+)}(G,K)}(K, V^{\bb{G}^{(h^+)}-an})=\varinjlim_{h\to \infty}  \Ext^i_{\n{D}^{(h)}(G,K)}(K, V^{\bb{G}^{(h)}-an}) = H^i_{la}(G, V) 
\]
as wanted.
\end{proof}

\subsection{Comparison results} \label{SubsecMainResults}

Next we state the main theorems of this section, which will be proved in \S \ref{Subsectionproof}. 

\begin{theorem} \label{theocohom1}
Let $C\in D(K_{\sol}[G])$, then 
\[
R\iHom_{K_{\sol}[G]}(K, C)= R\iHom_{K_{\sol}[G]}(K, C^{Rla}).
\]

\end{theorem}

\begin{remark}
More concretely, we will show that 
\[
R\iHom_{K_{\sol}[G]}(K, C)= R\iHom_{K_{\sol}[G]}(K, C^{R\bb{G}^{(h^+)-an}})
\]
for $h>>0$. 
\end{remark}

\begin{theorem} [Continuous vs. analytic vs. Lie algebra cohomology] \label{theocohom2}
Let $C \in D(K_{\sol}[G])$ be a derived  $\mathbb{G}^{(h^+)}$-analytic representation. Then
\[ 
R\iHom_{K_{\sol}[G]}(K, C) \cong R\iHom_{\n{D}^{(h^+)}(G, K)}(K, C)  \cong (R\iHom_{U(\f{g})}(K, C))^G.
\]
\end{theorem}

\begin{remark} \label{RemLiealgcohom}
The RHS term of the  equation above means the following:  if $C$ is a derived $\bb{G}^{(h^+)}$-analytic complex, then we will show that there is an open normal subgroup $G_h \subset G$ such that 
\[
R\underline{\Hom}_{U(\f{g})}(K,C)= R\underline{\Hom}_{K_{\sol}[G_h]}(K,C),
\]
and the group $G/G_h$ acts on the previous cohomology complex. As we are working in characteristic $0$ and $G/G_h$ is a finite group, taking invariants in the category of solid $K[G/G_h]$-modules is exact and one can form the complex 
\[
R\underline{\Hom}_{U(\f{g})}(K,C)^G:= R\underline{\Hom}_{U(\f{g})}(K,C)^{G/G_h}. 
\]
\end{remark}

\subsection{Key lemmas }
\label{subsectionKeyLemmas}

In the following we will work with complexes with equal terms but different differential maps. To make explicit the differentials we use the following notation:  let $C$ be a (homological) complex of $\O_{K,\sol}$-modules with $i$-th term $C_i$ and $i$-th differential $d_i:C_i\to C_{i-1}$, we  note
\[
C=[\cdots \rightarrow C_{i+1} \rightarrow C_{i} \rightarrow C_{i-1} \rightarrow \cdots ;  d_\bullet]. 
\]

\subsubsection{Iwasawa and distribution algebras}

The following result is the main input for our calculations.

\begin{theorem} [Lazard-Serre] \label{LazardSerre} 
Let $G_0$ be a uniform pro-$p$ group of dimension $d$. Then there exists a projective resolution of the trivial module $\Z_p$ of the form
\[ P:= [ 0\rightarrow \Z_{p,\sol}[G_0]^{\binom{d}{d}} \rightarrow \cdots \rightarrow \Z_{p,\sol}[G_0]^{\binom{d}{i}} \rightarrow \cdots \rightarrow \Z_{p,\sol}[G_0]^{\binom{d}{0}}   ; \alpha_{\bullet}]. \]
\end{theorem}

\begin{proof}
We briefly sketch how the complex $P$ is constructed from  \cite[D\'efinition V.2.2.2.1, Lemme V.2.1.1]{Lazard}. Let $g_1,\ldots, g_d\in G_0$ be a basis of the group and $b_i= [g_i]-1\in \Z_{p,\sol}[G_0](*)$. The valuation of $G_0$ defines a filtration in $\Z_{p,\sol}[G_0](*)$  whose graded algebra $\gr^{\bullet} (\Z_{p,\sol}[G_0](*))$ is isomorphic to $\bb{F}_p[\pi] [\overline{b}_1, \ldots, \overline{b}_d]$, where $\bb{F}[\pi]= \gr^{\bullet} (\Z_p)$ is the graduation of $\Z_p$ for the filtration induced by $(p)$.  Then, the Koszul complex of $\bb{F}[p][\overline{b}_1,\ldots, \overline{b}_d]$ with respect to the regular sequence $(\overline{b}_1,\ldots, \overline{b}_d)$ can be lifted by approximations to the complex $P$ of the theorem. Furthermore, the proof also lifts  a chain homotopy $\tilde{s}_{\bullet}$ between the identity and the augmentation map $\epsilon :K[\overline{b}_1, \ldots, \overline{b}_d]\to \bb{F}_p[\pi]$, to a chain homotopy $s_{\bullet}$ between the identity and the augmentation map $\epsilon: \Z_{p,\sol}[G_0]\to \Z_p$. 
\end{proof}

\begin{theorem} [Kohlhaase] \label{propKohlhaase}
Let $G_0$ a uniform pro-$p$ group of dimension $d$ and $h>0$. Then
\[ \O_{K} \otimes^L_{\O_{K, \sol}[G_0]} \Dist_{(h)}(G_0, K) = K. \]
More precisely, the differentials $\alpha_i : \O_{K, \sol}[G_0]^{d \choose {i}} \to \O_{K, \sol}[G_0]^{d \choose {i - 1}}$ of the resolution given by Theorem \ref{LazardSerre} extend to maps $\alpha_i : \Dist_{(h)}(G_0, K)^{d \choose {i}} \to \Dist_{(h)}(G_0, K)^{d \choose {i - 1}}$, inducing a resolution of the trivial module $K$ of the form
\[ P_{(h)}:=[ 0\rightarrow \Dist_{(h)}(G_0,K)^{\binom{d}{d}} \rightarrow \cdots \rightarrow  \Dist_{(h)}(G_0, K)^{\binom{d}{i}} \rightarrow \cdots \rightarrow \Dist_{(h)}(G_0,K)^{\binom{d}{0}}; \alpha_{\bullet}]. \]
\end{theorem}

\begin{proof}
This is essentially \cite[Theorem 4.4]{Kohlhaase}. Let $g_1,\ldots, g_d\in G_0$ be a basis and $b_i\in \Z_{p,\sol}[G_0](*)$. The idea of the proof is to show that the differentials $\alpha_{\bullet}$ and the chain homotopy $s_{\bullet}$ of Theorem \ref{LazardSerre} are continuous with respect to the norms $|\sum_{\alpha}a_{\alpha} \bbf{b}^{\alpha}|_r= \sup_{\alpha} |a_{\alpha}|r^{|\alpha|} $ for $\frac{1}{p}<r<1$. Thus, the differentials $\alpha_{\bullet}$ and the chain homotopy $s_{\bullet}$ extend to the weak completion of these norms (i.e. the completion with respect to a radius $r$ seen as a subspace in the completion of a slightly bigger radius $r<r'<1$). Note that  the distribution algebras  constructed in this way are precisely the algebras $\n{D}_{(h)}(G_0,K)$ of Definition \ref{DefiDistributionhenbas}. 
\end{proof}

\begin{remark}
\label{Remarkinductiondistribution}
By definition $\n{D}_{(h)}(G,K):= K_{\sol}[G]\otimes_{K_{\sol}[G_0]} \n{D}_{(h)}(G_0,K)$.  Therefore  $\n{D}_{(h)}(G,K) \otimes_{K_{\sol}[G]}K= \n{D}_{(h)}(G_0,K) \otimes_{K_{\sol}[G_0]}K=K$.  This implies that 
\[
\mathcal{D}_{(h)}(G, K) \otimes_{K_{\sol}[G]}^L \mathcal{D}_{(h)}(G, K) = K_{\sol}[G]\otimes_{K_{\sol}[G_0]}^L(\mathcal{D}_{(h)}(G_0, K) \otimes_{K_{\sol}[G]}^L \mathcal{D}_{(h)}(G_0, K) ). 
\]
\end{remark}

With the help of the previous theorem we can compute the following derived tensor product 

\begin{proposition}
\label{PropKeyLemma}
We have 
\[ \mathcal{D}_{(h)}(G, K) \otimes_{K_{\sol}[G]}^L \mathcal{D}_{(h)}(G, K) = \mathcal{D}_{(h)}(G, K). \]
\end{proposition}

\begin{proof}
First, we reduce to the case when $G$ is a uniform pro-$p$-group by Remark \ref{Remarkinductiondistribution}.  By Theorem \ref{LazardSerre} we can write
\[  [0 \to K_{\sol}[G] \to \cdots \to K_{\sol}[G]^{ d} \to K_{\sol}[G] ; \alpha] \simeq K. \]
Tensoring  with $\n{D}_{(h)}(G,K)$ over $K$ we get 
\[ [0 \to   K_{\sol}[G] \otimes_K\Dist_{(h)}(G, K)  \to \cdots K_{\sol}[G]^d \otimes_K\Dist_{(h)}(G, K)   \to K_{\sol}[G] \otimes_K\Dist_{(h)}(G, K)   ; \alpha \otimes 1] \simeq \mathcal{D}_{(h)}(G, K) . \]
The quasi-isomorphism above is of $K_{\sol}[G]$-modules for the diagonal action of $G$ in the terms of the complex.

 Let $\iota: K_{\sol}[G] \rightarrow K_{\sol}[G]$ be the antipode, i.e. the map induced by the inverse of the group, and denote in the same way its extension to the distribution algebra $\n{D}_{(h)}(G,K)$.   Consider the composition
\[  K_{\sol}[G] \otimes_K \Dist_{(h)}(G, K) \xrightarrow{ (1 \otimes \iota)\otimes 1}  K_{\sol}[G] \otimes_K K_{\sol}[G] \otimes_K \Dist_{(h)}(G, K)  \xrightarrow{1\otimes m }  K_{\sol}[G] \otimes_K \Dist_{(h)}(G, K),  \]
where $m$ is the  left multiplicaiton of $K_{\sol}[G]$ on the distribution algebra.   This map defines a $G$-equivariant isomorphism
\[
\phi: K_{\sol}[G]\otimes_K \n{D}_{(h)}(G,K)\cong K_{\sol}[G] \otimes_{K} \n{D}_{(h)}(G,K)_0
\]
where the action of $G$ in the image is left multiplication on $K_{\sol}[G]$ and trivial on $ \n{D}_{K}(G,K)_0$. Notice that $\phi$ can be extended naturally to a $G$-equivariant  isomorphism 
\[
\phi: \n{D}_{(h)}(G,K)\otimes_K \n{D}_{(h)}(G,K) \cong \n{D}_{(h)}(G,K)\otimes_K \n{D}_{(h)}(G,K)_0.  
\]
We define the complex
\begin{equation}
\label{eqcomplexbeta}
[0 \to K_{\sol}[G]\otimes_K \Dist_{(h)}(G, K)_0  \to \cdots \to  K_{\sol}[G]\otimes_K \Dist_{(h)}(G, K)_0 ; \beta_{\bullet} ]
\end{equation}
to be the complex whose differentials are given by $\beta_{\bullet} =   \phi \circ \alpha_{\bullet} \circ \phi^{-1}$.  Notice that the maps $\beta$ extend to respective complex with terms direct sums of $\n{D}_{(h)}(G,K)\otimes_K \n{D}_{(h)}(G,K)_{0}$.

Using this complex one can easily compute the derived tensor product  by replacing the right $\n{D}_{(h)}(G,K)$ with (\ref{eqcomplexbeta}):
\begin{eqnarray*} 
\n{D}_{(h)}(G,K)\otimes^L_{K_{\sol}[G]} \n{D}_{(h)}(G,K) &\simeq & [\cdots \to \Dist_{(h)}(G, K) \otimes_{K_{\sol}[G]} (K_{\sol}[G]^{ {d \choose i}} \otimes_{K} \Dist_{(h)}(G, K)_0 ) \to \cdots ; 1\otimes \beta] \\
&=& [\hdots \to \Dist_{(h)}(G, K)^{ {d \choose i}} \otimes_K \Dist_{(h)}(G, K)_0 \to \cdots ; \beta] \\
&\simeq& [\cdots \to \Dist_{(h)}(G, K)^{{d \choose i}} \otimes_K \Dist_{(h)}(G, K) \to \cdots ;   \alpha\otimes 1] \\
&\simeq&  \Dist_{(h)}(G, K), 
\end{eqnarray*}
In the above sequence of isomorphisms, the first quasi-isomorphism follows from the observation that the action of $G$ on the complex (\ref{eqcomplexbeta}) representing $\Dist_{(h)}(G, K)$ is trivial on the factor $\Dist_{(h)}(G, K)_0$. The second step is trivial. The third one follows by applying $\phi^{-1}$.   The fourth quasi-isomorphism follows from Theorem \ref{propKohlhaase}.  This finishes the proof.

\end{proof}

\begin{corollary}
\label{CoroTechnicalDhplus}
We have 
\begin{eqnarray*}
\n{D}^{(h^+)}(G,K)\otimes_{K_{\sol}[G]}^L \n{D}^{(h^+)}(G,K) & = & \n{D}^{(h^+)}(G,K) \\
\n{D}^{la}(G,K)\otimes_{K_{\sol}[G]}^L \n{D}^{la}(G,K) & = & \n{D}^{la}(G,K) .
\end{eqnarray*}
\proof
This follows from the previous proposition and the fact that $\n{D}^{(h^+)}(G,K)$ can be written as a colimit of distribution algebras $\n{D}_{(h')}(G,K)$, cf. Corollary \ref{coroDhpluscolimit}. The case of $\n{D}^{la}(G,K)$ follows from the same proof of Proposition \ref{PropKeyLemma} knowing that the complex of Theorem \ref{propKohlhaase} extends to $\n{D}^{la}(G,K)$. 
\endproof
\end{corollary}

\subsubsection{Enveloping and distribution algebras}

Let $\bb{G}_{h}$ be a rigid analytic group of  Definition \ref{defAffinoidgroups})  and let $G_h= \bb{G}_{h}(\Q_p)$ be its rational points. Recall that $\bb{G}_{h}$ is a polydisc centered in $1\in G$ of radius $p^{-h}$,  and that $\bb{G}^{(h)}= G \bb{G}_{h}$ is a finite  disjoint union of copies of  $\bb{G}_h$.    Note that $G_h\subset G$ is an open compact subgroup.  We denote by $\n{D}(\bb{G}_{h},K)$ the distribution algebra of $\bb{G}_{(h)}$-analytic functions, i.e. the dual of $C(\bb{G}_{h}, K)$.  We also denote $\n{D}(\bb{G}_{h^+},K)= \varinjlim_{h'>h} \n{D}(\bb{G}_{h'},K)$, in other words, the distribution algebra of the rigid analytic group defined by the open unit polydisc $\bb{G}_{h^+}= \bigcup_{h'>h} \bb{G}_{h'}$. We assume that $\bb{G}_{h^+}(\Q_p)=G_h$.

\begin{proposition} [Tamme]
\label{PropTamme}
Keep the above notation, and let
\[  CE(\f{g}) := [   0\to U(\f{g})\otimes \wedge^{d} \f{g} \to \cdots \to U(\f{g})\otimes \f{g}\to      U(\f{g}) ; d] \]
be the Chevalley-Eilenberg complex resolving the trivial representation $K$. Then $ \n{D}(\bb{G}_{h^+},K)\otimes_{U(\f{g})}^L K=K$. Moreover,   the complex  $\n{D}(\bb{G}_{h^+},K)\otimes_{U(\f{g})}^L CE(\f{g})$ is the dual of the global sections of the de Rham complex of $\bb{G}_{h^+}$. 
\proof
Let $[\Omega_{\bb{G}_{h^+}}^\bullet,d]$ be the de  Rham complex of $\bb{G}_{h^+}$, notice that the global sections of $\Omega_{\bb{G}_{h^+}}^i$ are equal to $C(\bb{G}_{h^+},K)\otimes_K \bigwedge^{i}(\f{g}^{\vee})$.  As $\bb{G}_{h^+}$ is a open polydisc, the Poincar\'e lemma holds (cf. \cite[Lemma 26]{Tamme}) and the global sections of the de Rham complex is  
\begin{equation}
\label{deRahmComplex}
C(\bb{G}_{h^+},K)\xrightarrow{d} C(\bb{G}_{h^+}, K)\otimes_K \f{g}^{\vee} \xrightarrow{d} \cdots \xrightarrow{d} C(\bb{G}_{h^+},K) \otimes_K \bigwedge^d \f{g}^{\vee}\to 0,
\end{equation}
which is quasi-isomorphic to $K$ via the inclusion of the constant functions $K\subset C(\bb{G}_{h^+},K)$.  It is easy to show that the dual of (\ref{deRahmComplex}) is equal to $\n{D}(\bb{G}_{h^+},K)\otimes_{U(\f{g})} CE(\f{g})$.  Finally, the fact that $\n{D}(\bb{G}_{h^+},K)\otimes_{U(\f{g})} CE(\f{g})$ is a projective resolution of $K$ as  $\n{D}(\bb{G}_{h^+},K)$-module follows from the exactness of (\ref{deRahmComplex}) and the duality Theorem \ref{theorem:duality}. 
\endproof
\end{proposition}

\begin{lemma}
We have
\[ \n{D}(\bb{G}_{h^+},K)\otimes^L_{U(\f{g})} \n{D}(\bb{G}_{h^+},K)= \n{D}(\bb{G}_{h^+},K). \]
In particular $\Mod_{\n{D}(\bb{G}_{h^+},K)_\sol}^{\solid}$ is a full subcategory of the category of solid $U(\f{g})$-modules. 
\end{lemma}

\begin{proof}
The proof follows identically as that of Proposition \ref{PropKeyLemma} by replacing the Lazard-Serre resolution by the Chevalley-Eilenberg resolution $CE(\mathfrak{g})$ of Proposition \ref{PropTamme}. 
\end{proof}

\subsection{Proofs}
\label{Subsectionproof}

\begin{proof}[Proof of Theorem \ref{theocohom1}]

 Let $C\in D(K_{\sol}[G])$. By Theorem \ref{TheoMain} we have 
\[
C^{R \bb{G}^{(h^+)}-an} =  \iRHom_{K_{\sol}[G]}(\Dist^{(h^+)}(G, K), C).
\]
We now compute
\begin{eqnarray*}
\iRHom_{K_\sol[G]}(K, C^{R \bb{G}^{(h^+)}-an}) &=& \iRHom_{K_\sol[G]}(K, \iRHom_{K_{\sol}[G]}(\Dist^{(h^+)}(G, K), C)) \\
&=& \iRHom_{K_\sol[G]}(K \otimes^L_{K_\sol[G]} \Dist^{(h^+)}(G, K), C) \\
&=& \iRHom_{K_\sol[G]}(K, C),
\end{eqnarray*}
where the second equality is the tensor-Hom adjunction, and the third one follows from Theorem \ref{propKohlhaase}.
\end{proof}

\begin{proof}[Proof of Theorem \ref{theocohom2}]
Let $C$ be a derived $\bb{G}^{(h^+)}$-analytic representation of $G$.  Theorem \ref{TheoMain} says that $C$ is a $\n{D}^{(h^+)}(G,K)$-module.  By Theorem \ref{propKohlhaase} one has 
\begin{eqnarray*}
 R \iHom_{K_{\sol}[G]}(K, C) &=&    R \iHom_{\Dist^{(h^+)}(G, K)}(\Dist^{(h^+)}(G, K) \otimes_{K_{\sol}[G]}^{L} K, C) \\
&=& R \iHom_{\Dist^{(h^+)}(G, K)}(K, C).
\end{eqnarray*}

On the other hand, since $\n{D}^{(h^+)}(G,K)= \O_{K,\sol}[G]\otimes_{\O_{K,\sol}[G_h]} \n{D}(\bb{G}_{h^+},K)$, one has  
\begin{eqnarray*}
R\iHom_{\n{D}^{(h^+)}(G,K)}(K,C)= R\iHom_{\n{D}(\bb{G}_{h^+},K)}(K, C)^{G/G_h}.
\end{eqnarray*}
By Proposition \ref{PropTamme} we get 
\begin{eqnarray*}
R\iHom_{U(\f{g})}(K, C) & = & R\iHom_{\n{D}(\bb{G}_{h^+},K)}(\n{D}(\bb{G}_{h^+},K)\otimes_{U(\f{g})}^L K, C)  \\
& = &  R\iHom_{\n{D}(\bb{G}_{h^+},K)}(K, C).
\end{eqnarray*}
Putting all together we obtain 
\[
R\iHom_{K_{\sol}[G]}(K,C)= R\iHom_{\n{D}^{(h^+)}(G,K)}(K,C)= R\iHom_{U(\f{g})}(K,C)^{G} 
\]
as we wanted. 
\end{proof}

\subsection{Shapiro's lemma and Hochschild-Serre}

Let $G$ be a compact $p$-adic Lie group of dimension $d$ and  $H$  a closed   subgroup of dimension $e$.  One can find an open uniform pro-$p$-group $G_0\subset G$ satisfying the following conditions: 
\begin{enumerate}
    \item $H_0:= H\cap G_0$ is an uniform pro-$p$-group. 

\item There are charts $\phi_{G_0}: \mathbf{z}_p^{d} \to G_0 $ and $\phi_{H_0}:  \mathbf{Z}_p^{e}\to H_0$ such that $\phi_{G_0}\circ \iota_{e} =\phi_{H_0}$, where $\iota_e: \mathbf{Z}_p^{e}\to \mathbf{Z}_{p}^{d}$ is the inclusion in the last $e$-components.
\end{enumerate}
Indeed, taking $\f{g}_0\subset \f{g}$ a small enough lattice as in  \cite[\S 5.2]{Emerton}  and $\f{h}_0:=\f{h}\cap \f{g}_0$,  one can take $G_0:= \exp(\f{g}_0)$ and $H_0:=\exp(\f{h}_0)$. The profinite groups $H_0$ and $G_0$ allow us to define compatible rigid analytic neighbourhoods $\bb{H}^{(h^+)}$ and $\bb{G}^{(h^+)}$ of $H$ and $G$ respectively,   with   $\bb{G}^{(h^+)}/\bb{H}^{(h^+)}$ a finite disjoint union of open polydiscs of dimension $d-e$,  and such that   
\[
C(\bb{G}^{(h^+)},  K)=  C(\bb{G}^{(h^+)}/\bb{H}^{(h^+)}, K) \otimes_{K_\sol} C(\bb{H}^{(h^+)},  K) .  
\]
In other words,  if  $\n{D}^{(h^+)}(G/H,K)$ denotes the dual of $C(\bb{G}^{(h^+)}/\bb{H}^{(h)},K)$,  we have an isomorphism of right $\n{D}^{(h^+)}(H,K)$-modules 
\[
\n{D}^{(h^+)}(G,K)= \n{D}^{(h^+)}(G/H,K)\otimes_{K_\sol}  \n{D}^{(h^+)}(H,K). 
\]
One has a similar description as left  $\n{D}^{(h^+)}(H,K)$-modules.  

\begin{definition}
For  $C \in D(K_\sol[H])$  we  define the solid induction and coinduction of $C$ from $H$ to $G$ as
\[ \mathrm{ind}_H^G(C) := K_\sol[G] \otimes_{K_\sol[H]}^L C, \]
\[ \mathrm{coind}_{H}^G(C) := R \iHom_{K_\sol[H]}(K_\sol[G], C), \]
where the action of $G$ is given by left multiplication on $K_\sol[G]$ for the induction, and by right multiplication on $K_\sol[G]$ for the coinduction. If $C$ is derived $\mathbb{H}^{(h^+)}$-analytic, define the analytic induction and coinduction as
\[ \hind_H^G(C) := \Dist^{(h^+)}(G, K) \otimes_{\Dist^{(h^+)}(H, K)}^L C, \]
\[ \hcoind_{H}^G(C) := R \iHom_{\Dist^{(h^+)}(H, K)}(\Dist^{(h^+)}(G, K), C). \]
\end{definition}

\begin{proposition} (Shapiro's lemma)
Let $C \in D(K_\sol[G])$, $C' \in D(K_\sol[H])$. Then $\mathrm{ind}_{H}^G$ (resp.  $\mathrm{coind}_{H}^G$) is the left (resp. right) adjoint of the restriction map $D(K_\sol[G])\to D(K_{\sol}[H])$. In other words, 
\[ R \iHom_{K_\sol[G]}(\mathrm{ind}_H^G(C'), C) = R \iHom_{K_\sol[H]}(C', C), \]
\[ R \iHom_{K_\sol[G]}(C, \mathrm{coind}_H^G(C')) = R \iHom_{K_\sol[H]}(C, C'). \]
Analogously, if $C$ and $C'$ are derived $\mathbb{G}^{(h^+)}$-analytic and $\mathbb{H}^{(h^+)}$-analytic representations, then
\[ R \iHom_{\Dist^{(h^+)}(G, K)}(\hind_H^G(C'), C) = R \iHom_{\Dist^{(h^+)}(H, K)}(C', C), \]
\[ R \iHom_{\Dist^{(h^+)}(G, K)}(C, \hcoind_H^G(C')) = R \iHom_{\Dist^{(h^+)}(H, K)}(C, C'). \]
\end{proposition}

\begin{proof}
The first statement follows formally:
\[ R \iHom_{K_\sol[G]}(K_\sol[G] \otimes_{K_\sol[H]}^L C, C') = R \iHom_{K_\sol[H]}(C, R \iHom_{K_\sol[G]}(K_\sol[G], C')) = R \iHom_{K_\sol[H]}(C, C'). \]
The rest of the statements are proved in a similar way. 
\end{proof}

\begin{proposition} (Hochschild-Serre)
Let $H\subset G$ be a normal closed  subgroup and $C \in D(K_\sol[G])$. Then
\[ R \iHom_{K_\sol[G]}(K, C) = R \iHom_{K_\sol[G/H]}(K, R \iHom_{K_\sol[H]}(K, C)). \]
If $C$ is derived $\mathbb{G}^{(h^+)}$-analytic, then
\[ R \iHom_{\Dist^{(h^+)}(G, K)}(K, C) = R \iHom_{\Dist^{(h^+)}(G/H, K)}(K, R \iHom_{\Dist^{(h^+)}(H, K)}(K, C)). \]
\end{proposition}

\begin{proof}
By Shapiro's lemma we have
\[ R \iHom_{K_\sol[H]}(K, C) = R \iHom_{K_\sol[G]}(K_\sol[G] \otimes_{K_\sol[H]}^L K, C). \]
Applying the functor $R \iHom_{K_\sol[G/H]}(K, -)$ to both sides,  using $K_\sol[G] \otimes_{K_\sol[H]}^L K = K_\sol[G/H]$ and the usual  adjunction one obtains
\[ R \iHom_{K_\sol[G/H]}(K, R \iHom_{K_\sol[H]}(K, C)) = R \iHom_{K_\sol[G]}(K, C), \]
as desired. The rest of the statements are proved in a similar way.
\end{proof}

\subsection{Homology and duality}

\begin{definition}
Let $C \in D(K_\sol[G])$. We define the solid group homology of $C$ as
\[ K \otimes_{K_\sol[G]}^L C. \]
Analogously, if $C$ is derived  $\mathbb{G}^{(h^+)}$-analytic, define its  $\bb{G}^{(h^+)}$-analytic homology as
\[ K \otimes_{\Dist^{(h^+)}(G, K)}^L C. \]
\end{definition}

We have the following formal duality between homology and cohomology.

\begin{lemma}
Let $C \in D(K_\sol[G])$. Then
\[ R \iHom_{K}(K \otimes_{K_\sol[G]}^L C, K) = R \iHom_{K_\sol[G]}(K, R \iHom_K(C, K)). \]
If $C$ is $\mathbb{G}^{(h^+)}$-analytic, then
\[ R \iHom_{K}(K \otimes_{\Dist^{(h^+)}(G, K)}^L C, K) = R \iHom_{\Dist^{(h^+)}(G, K)}(K, R \iHom_K(C, K)). \]
\end{lemma}

Let $K(\chi)=\bigwedge^{d} \f{g}^\vee$ denote the  determinant of the dual  adjoint representation of $G$.  Using Lazard-Serre's Theorem \ref{LazardSerre} one   easily deduces  that  $R\iHom_{K_\sol[G]}(K,K_\sol[G])$, endowed with the right multiplication of $G$, is a character concentrated in degree $-d$.  Moreover, using the de Rham complex of $\bb{G}^{(h^+)}$ one can even prove that 
\begin{equation} \label{Equationchi}
R\iHom_{K_\sol[G]}(K, K_\sol[G])= R\iHom_{\n{D}^{(h^+)}(G,K)}(K,\n{D}^{(h^+)}(G,K)) = K(\chi)[-d]. 
\end{equation}
The following theorem relates cohomology and homology in a more interesting way.

\begin{theorem}
Let $C \in D(K_\sol[G])$. Then there is a natural quasi-isomorphism
\[ R \iHom_{K_\sol[G]}(K, C) = K(\chi)[-d]\otimes^{L}_{K_\sol[G]} C . \]
Furthermore,  if $C$ is derived $\bb{G}^{(h^+)}$-analytic, we have 
\[ R \iHom_{\n{D}^{(h^+)}(G,K)}(K, C) = K(\chi)[-d] \otimes_{\n{D}^{(h^+)}(G,K)}^L  C . \]
\end{theorem}

\begin{proof}
First observe that, given any $G$-equivariant map $\alpha: {K_\sol[G]}_{ \star_1} \to {K_\sol[G]}_{\star_1}$, one has a commutative diagram
\begin{equation} \label{commutdiag1}
\begin{tikzcd}
R\iHom_{K_\sol[G]}(K_{\sol}[G]_{\star_1},C )  \ar[d,"\alpha^*"] & \iHom_{K_\sol[G]}(K_{\sol}[G]_{\star_1},K_{\sol}[G])\otimes_{K_\sol[G]}^{L} C  \ar[d, "\alpha^*\otimes 1"]  \ar[l, "\sim"'] \\
R\iHom_{K_\sol[G]}(K_{\sol}[G]_{\star_1},C ) &\iHom_{K_\sol[G]}(K_{\sol}[G]_{\star_1},K_{\sol}[G] )\otimes_{K_\sol[G]}^{L} C \ar[l, "\sim"'], \\
\end{tikzcd}
\end{equation}
where the rows are $G$-equivariant. The $G$-actions on the different terms of \eqref{commutdiag1} are given as follows: on $R\iHom_{K_\sol[G]}(K_{\sol}[G]_{\star_1},C )$ the group $G$ acts via the right regular action $\star_2$ on $K_\sol[G]$. Concerning the term $\iHom_{K_\sol[G]}(K_{\sol}[G]_{\star_1},K_{\sol}[G])\otimes_{K_\sol[G]}^{L} C$ we have: the $\Hom$ space $\iHom_{K_\sol[G]}(K_{\sol}[G]_{\star_1},K_{\sol}[G])$ is taken with respect to the left regular action $\star_1$ on each term. It is endowed with with an action of $G \times G$ given by the right regular actions. The tensor tensor product $\otimes_{K_\sol[G]}^L C$ is taken with respect to the action given by $\{ 1\} \times G$ (i.e. the $\star_2$-action of $G$ on the target of the $\Hom$ space). Finally, the $G$-action on the whole term is the one induced by the $\star_2$-action of $G$ on the source of the $\Hom$ space.

Notice that there is a natural identification of right $K_\sol[G]$-modules
\[
R\iHom_{K_\sol[G]}(K_\sol[G]_{\star_1}, K_\sol[G])= K_{\sol}[G]_{\star_2}. \]
Recall that, by Theorem \ref{LazardSerre}, we have a projective resolution $K \simeq [K_\sol[G]_{\star_1}^{d \choose \bullet}; \alpha_\bullet]$. We obtain
\begin{eqnarray*}
R \iHom_{K_\sol[G]}(K, C) &=& R \iHom_{K_\sol[G]}([K_\sol[G]_{\star_1}^{d \choose \bullet}; \alpha_\bullet], C) \\
&=& [\iHom_{K_\sol[G]}(K_\sol[G]_{\star_1}, K_\sol[G])^{d \choose \bullet}; \alpha_\bullet^* \otimes 1] \otimes_{K_\sol[G]}^L C \\
&=& R\iHom_{K_\sol[G]}([K_\sol[G]_{\star_1}^{d \choose \bullet}; \alpha_\bullet], K_\sol[G]) \otimes_{K_\sol[G]}^L C \\
&=& R \iHom_{K_\sol[G]}(K, K_\sol[G]) \otimes_{K_\sol[G]}^L C \\
& = & K(\chi)[-d] \otimes^L_{K_\sol[G]} C,
\end{eqnarray*}
where the second equality follows  by \eqref{commutdiag1},  and the last one by \eqref{Equationchi}. The statement for $\bb{G}^{(h^+)}$-analytic cohomology is proven in the same way.
\end{proof}


\bibliographystyle{alpha}
\bibliography{biblio}

\end{document}